\documentclass[a4paper, 12pt]{amsart}
\usepackage{fullpage} 
\usepackage{tcolorbox}
\usepackage{hyperref, amsmath, amsthm, amssymb, inputenc}
\allowdisplaybreaks
\newtheorem{lemma}{Lemma}[section]
\newtheorem{theorem}[lemma]{Theorem}
\newtheorem{remark}[lemma]{Remark}
\newtheorem{corollary}[lemma]{Corollary}
\newtheorem{definition}[lemma]{Definition}
\newtheorem{proposition}[lemma]{Proposition}
\title{Well-posedness for the Non-integrable Periodic Fifth Order KdV in Bourgain Spaces}
\author{Ryan McConnell}
\thanks{The author was partially supported by NSF grant  DMS-2154031.}
\begin{document}
\begin{abstract}
    We study well-posedness for a non-integrable generalization of the fifth order KdV, the second member in the KdV heirarchy. In particular, we use differentiation-by-parts to establish well-posedness for $s> 35/64$ in low modulation restricted norm spaces, as well as non-linear smoothing of order $\varepsilon < \min(2(s-35/64), 1)$. As corollaries, we obtain unconditional well-posedness for the non-integrable fifth order KdV for $s > 1$ and global well-posedness for the integrable fifth order KdV for $s\geq 1$. We also show local well-posedness for the non-integrable fifth order KdV for $s > 1/2$, contingent upon the conjectured $L^8$ Strichartz estimate. As an application of the nonlinear smoothing we obtain non-trivial upper bounds on the upper Minkowski dimension of the solution to the non-integrable fifth order KdV.
\end{abstract}
\maketitle
\tableofcontents
\section{Introduction}
We study the local well-posedness of the non-integrable fifth order KdV given by 
\begin{align}\label{Equation: General Fifth Order}
    \begin{cases}
        u_t -\partial_x^5 u +\alpha \partial_x (u^3) + \beta \partial_x(\partial_x u)^2 + \gamma \partial_x(u\partial_x^2u) = 0\\
        u(x,0) = u_0\in H^s(\mathbb{T}),
    \end{cases}
\end{align}
for $\alpha, \beta, \gamma\in\mathbb{R}$, and real initial data $u_0$. Under certain assumptions on the coefficients \eqref{Equation: General Fifth Order} is completely integrable and has been extensively studied, with well-posedness established in $L^2(\mathbb{T})$ by \cite{kappeler2018wellposedness}. Additionally, it has recently been showed that the integrable \eqref{Equation: General Fifth Order} posed on $\mathbb{R}$ is actually well-posed to $ s = -1$, in complete analogy with the KdV, \cite{bringmann2021global}.

In order to attack \eqref{Equation: General Fifth Order}, one must develop tools that don't rely on inverse scattering. In this direction, Kwon \cite{kwon2008fifth} introduced the modified energy method which enabled them to prove well-posedness for the non-integrable fifth order equation in $H^s(\mathbb{R})$ for $s > 5/2$, which was generalized by Kenig and Pilod for higher order equations in the KdV heirarchy, \cite{kenig2016local}. Well-posedness on $\mathbb{R}$ was then pushed to the $H^2$ level by Kenig and Pilod in \cite{kenig2015well}. It's also worth mentioning that \cite{kato2012well} proves local well-posedness in $H^s(\mathbb{R})\cap \dot{H}^a(\mathbb{R})$\footnote{Under the norm $\|\langle\xi\rangle^{s-a}\langle \xi\rangle^a\widehat{f}\|_{L^2_x(\mathbb{R})}$.} for $s\geq -1/4$ and a certain range of $a$, and \cite{kwon2008well} proves well-posedness in $H^{3/4}(\mathbb{R})$ for the the \textit{modified} fifth order KdV.

Similar efforts have taken place on $\mathbb{T}$, where the absence of hallmark dispersive effects causes great difficulty. Indeed, the most recent results for the periodic domain are restricted to considering $\gamma = 2\beta$ as in  \cite{kwak2018low}, where well-posedness is proved for the \textit{nearly} integrable fifth order equation at the $s = 2$ level using the modified energy method and short-time $X^{s,b}$ spaces. Similarly, \cite{kato2018unconditional} proves unconditional well-posedness in $H^s$ for $s\geq 3/2$, and \cite{tsugawa2017parabolic} proved well-posedness using the modified energy method of Kwon for certain polynomials in $\partial_x^3$, $\partial_x^2$ and $\partial_x$ for $s\geq 13$. 

In particular, \cite{tsugawa2017parabolic} begins to explore what they call ``Parabolic Resonances'', which are terms of the form
\[
in^{2\ell}\widehat{u}(n)\sum_{0=n_1+\cdots+n_k}m(n_1,\cdots, n_k)\prod_{j=1}^k\widehat{u}(n_j),
\]
where $m(-n_1, \cdots, -n_k) = -m(n_1, \cdots, n_k)$. The importance of the oddness condition on $m$ is that for real initial data $u_0$, the quantity
\begin{equation}\label{Equation: Introduction Parabolic Resonance}
in^{2\ell}\sum_{0=n_1+\cdots+n_j}m(n_1,\cdots, n_k)\prod_{j=1}^k\widehat{u}(n_j),
\end{equation}
is purely \textit{real}. This then becomes another term in the linear group that smooths the equation in one direction of time and induces ill-posedness in the other. In this language, the result of Tsugawa is that polynomials in $\partial_x^3$, $\partial_x^2$ and $\partial_x$ for which the induced  $\ell=1$ parabolic resonances decay are well-posed for $s\geq 13$ in both directions of time, whereas polynomials for which \eqref{Equation: Introduction Parabolic Resonance} does not decay are well-posed in only one direction of time, and ill-posed in the other.

Much like \cite{kenig2016local}, we will instead work with the non-integrable equation
\begin{align}\label{Equation: Toy Fifth Order}
    \begin{cases}
        u_t -\partial_x^5 u + 2\partial_x(u\partial_x^2u) = 0\\
        u(x,0) = u_0\in H^s(\mathbb{T}),
    \end{cases}
\end{align}
which not only conserves the mean, but also contains the worst portion of the non-linearity in \eqref{Equation: General Fifth Order}. Indeed, there is a $High\times Low$ interaction in which the high term will have to experience $3$ derivatives, which also demonstrates why the fifth order equation is hard to work with. In particular, the standard smoothing associated to the $X^{s,1/2}$ space will result in savings of 
\[
\sqrt{|(n_1+n_2)^5-n_1^5-n_2^5|} = \sqrt{|5nn_1n_2(n_1^2+n_1n_2+n_2^2)|}\sim |n_1|^2\sqrt{|n_2|},
\]
when $n_1\sim n\gg n_2$. As this is a whole derivative off of ameliorating the derivative losses in the non-linearity, we are forced to work a little harder.

The benefit of working with this reduction is twofold: 1) every result obtained for this equation extends to the full equation \eqref{Equation: General Fifth Order} (see the Section \ref{Section: Appendix A}), and 2) dramatically simplified equations. In particular, we show the following theorem, with \textit{no} assumption on the integrability of the equation.

\begin{theorem}\label{Theorem: Wellposedness for Toy}
Let $s> 35/64$ and $u_0\in H^s(\mathbb{T})$. Then for $0 \leq t\leq T = T(\|u_0\|_{H^s_x})$ there is a unique solution $\tilde{u}\in Y^s_T$ (defined in Section \ref{Section: background}) to \eqref{Equation: Tilde u differential equation before smoothing} with a continuous data-to-solution map from $H^s$ to $Y^s_T$. Furthermore, the following bound holds:
\[
\|\tilde{u}\|_{Y^s_T}\lesssim \|u_0\|_{H^s_x}.
\]
In particular, since the solution to \eqref{Equation: Tilde u differential equation before smoothing} is related to the solution of \eqref{Equation: Toy Fifth Order} by a bi-continuous map on $H^s$, we find that \eqref{Equation: Toy Fifth Order} is well-posed in $H^s$ in the sense of (\cite{tao2006nonlinear}, Definition 3.4)\footnote{This is a well-posedness result in the sense of gauge conjugation. See, e.g. \cite{colliander2001global, colliander2004multilinear, oh2020smoothing}.}
\end{theorem}

The lower bound of $35/64$ is not optimal, but we can say more. In that direction, we define \eqref{Definition: Property Pa} to be the claim that the $L^8$ Strichartz estimate holds with $a+$ losses. 
\begin{definition}[Property $(P_a)$]
Property $P_a$ holds for $a\geq 0$ if for all $f\in H^{a+}(\mathbb{T})$  we have the estimate
\begin{equation}\tag{$P_a$}\label{Definition: Property Pa}
\|W_tf\|_{L^{8}_{x,t}(\mathbb{T}^2)}\lesssim\|f\|_{H^{a+}}.
\end{equation}
\end{definition}

It follows by \eqref{Equation: L8 estimate} that \eqref{Definition: Property Pa} holds with $s\geq 3/32$. As it's conjectured that the $L^{12}$ Strichartz estimate should only produce $\varepsilon$ losses, we see that $a = 0$ is optimal. With this convention, we find the following enhanced version of Theorem \ref{Theorem: Wellposedness for Toy}.
\begin{theorem}\label{Theorem: Wellposedness for Toy with Pa}
Let $a\geq 0$, assume \eqref{Definition: Property Pa} holds, $s> \frac{1+a}{2}$, and $u_0\in H^s(\mathbb{T})$. Then for $0 \leq t\leq T = T(\|u_0\|_{H^s_x})$ there is a unique solution $\tilde{u}\in Y^s_T$ (defined in Section \ref{Section: background}) to \eqref{Equation: Tilde u differential equation before smoothing} with a continuous data-to-solution map from $H^s$ to $Y^s_T$. Furthermore, the following bound holds:
\[
\|\tilde{u}\|_{Y^s_T}\lesssim \|u_0\|_{H^s_x}.
\]
In particular, since the solution to $\eqref{Equation: Tilde u differential equation before smoothing}$ is related to the solution of \eqref{Equation: Toy Fifth Order} by a bi-continuous map on $H^s$, we find that \eqref{Equation: Toy Fifth Order} is well-posed in $H^s$ in the sense of (\cite{tao2006nonlinear}, Definition 3.4).
\end{theorem}

Theorem \ref{Theorem: Wellposedness for Toy} is proved through a differentiation-by-parts procedure inspired by the work of Shatah \cite{shatah1985normal} and Babin-Ilyin-Titi \cite{babin2011regularization}, with the added intricacy of attempting to perform the fixed point argument in an $X^{s,b}$ type space and not simply $H^s$. In particular, through two differentiation-by-parts applications and the extra smoothing that the $X^{s,b}$ norms are known for, we are able to obtain the above result. The well-posedness itself is proved using a contraction argument in the Bourgain space $Y^s$ on the integral equation \eqref{Equation: Contractive Operator Definition} (which is obtain through the differentiation-by-parts process). While Bourgain spaces and differentiation-by-parts have been used together in the past (e.g. \cite{erdogan2017smoothing}), this is the first occurrence, to the author's knowledge, in which well-posedness itself has been proved using a contraction argument, differentiation-by-parts, and $X^{s,b}$.

Additionally, since it's possible, in principle, to work in $C^0_tL^2_x$, the main new ingredient is fully embracing the non-integrability of the system. That is, the work of, say, \cite{erdogan2013talbot, kato2018unconditional} seek to only remove resonances that correspond to conserved quantities of the equation, whereas work on the KdV such as \cite{bourgain1993fourier, colliander2004multilinear, oh2020smoothing} begin to explore the removal of time dependent (spatially independent) factors. We fully embrace this perspective, turning the problem of solving the rather complicated non-integrable \eqref{Equation: Toy Fifth Order}, into simply applying an, albeit messy, rather straightforward procedure.

In Section \ref{Section: Appendix A} we demonstrate that no \textit{parabolic resonances} are created when considering the full equation, \eqref{Equation: General Fifth Order}. Because of this, we immediately establish the following corollary.
\begin{corollary}
Let $a\geq 0$ be such that \eqref{Definition: Property Pa} holds, $s> \frac{1+a}{2}$, and $\alpha,\beta,\gamma\in\mathbb{R}$. Then \eqref{Equation: General Fifth Order} is well-posed in the sense of (\cite{tao2006nonlinear}, Definition 3.4). In particular, \eqref{Equation: General Fifth Order} is well-posed for $s > 35/64.$
\end{corollary}
Additionally, as a consequence of the the integrable fifth order KdV sharing the same conservation laws as the KdV, we find the following global well-posedness result as a corollary.
\begin{corollary}
Let $s\geq 1$, $\alpha,\beta,\gamma\in\mathbb{R}$, $\gamma = 2\beta$, and $u_0\in H^s(\mathbb{T})$. Then the solution, $u\in C^0_tH^s_x$, of \eqref{Equation: General Fifth Order} emanating from $u_0$ extends globally.
\end{corollary}

In addition to proving well-posedness, we continue the study of the regularity of the integral term in the Duhamel representation of the solution $u$ emanating from $u_0$. Specifically, we establish the following bound on the integral term in the Duhamel representation.
\begin{theorem}\label{Theorem: Smoothing}
Let $a\geq 0$ be such that \eqref{Definition: Property Pa} holds, $s > \frac{1+a}{2}$, $\varepsilon < \min(2s-1-a, 1)$, $u_0\in H^s(\mathbb{T})$, and $u$ be the solution to \eqref{Equation: Toy Fifth Order} emanating from $u_0$ for $T = T(\|u_0\|_{H^s})$. Then there is an invertible transformation\footnote{This is the same transformation used for gauge well-posedness.} $u\mapsto \tilde{u}$ with $\|\tilde{u}\|_{C^0_tH^s_x} = \|u\|_{C^0_tH^s_x}$, $u(x,0) = \tilde{u}(x,0)$ so that
\[
\|\tilde{u} - e^{t\partial_x^5}u_0\|_{C^0_t([0,T])H^{s+\varepsilon}_x(\mathbb{T})}\lesssim C(\|u_0\|_{H^s_x(\mathbb{T})}).
\]
\end{theorem}

This phenomenon, nonlinear smoothing, was fundamental to Bourgain's  high-low method,  \cite{bourgain1998refinements}, and is of interest in periodic problems due to the lack of hallmark dispersive traits. Indeed, through the use of differentiation-by-parts, Erdogan and Tzirakis were able to establish nonlinear smoothing for the periodic KdV (see: \cite{erdougan2013global, erdogan2011long}) as well as for the fractional cubic NLS and the quintic NLS on the real line with Gurel, \cite{erdogan2017smoothing}. Oh and Stefanov, \cite{oh2020smoothing}, established smoothing for the generalized KdV using a resonant decomposition similar to Proposition \ref{Proposition: Symbol Decomp}, as well as a normal form transformation. Differentation-by-parts was also recently used to show smoothing for the periodic dNLS, \cite{isom2020growth}, solving a once intractable problem. For information about real line smoothing see, e.g. \cite{ correia2020nonlinear, erdogan2017smoothing, keraani2009smoothing}. 

In Section \ref{Section: Appendix A} we demonstrate that the high-frequency approximation to the single resonances of Equation \eqref{Equation: General Fifth Order} experiences the same cancellation effects that were necessary to establish smoothing for Equation \eqref{Equation: Toy Fifth Order}. That is, we show that there are no \textit{parabolic resonances} with $k = 0$, which requires a rather delicate cancellation that does \textit{not} depend on the value of the coefficients $\alpha, \beta, \gamma$. In particular, we immediately find the following corollary.
\begin{corollary}
Let $a\geq 0$ be such that \eqref{Definition: Property Pa} holds, $s > \frac{1+a}{2}$, $\alpha,\beta,\gamma\in\mathbb{R}$, $\varepsilon < \min(2s-1-a, 1)$, $u_0\in H^s(\mathbb{T})$, and $u$ be the solution to \eqref{Equation: General Fifth Order} emanating from $u_0$ for $T = T(\|u_0\|_{H^s})$. Then there is an invertible transformation $u\mapsto \tilde{u}$ with $\|\tilde{u}\|_{C^0_tH^s_x} = \|u\|_{C^0_tH^s_x}$, $u(x,0) = \tilde{u}(x,0)$ so that
\[
\|\tilde{u} - e^{t\partial_x^5}u_0\|_{C^0_t([0,T])H^{s+\varepsilon}_x(\mathbb{T})}\lesssim C(\|u_0\|_{H^s_x(\mathbb{T})}).
\]
\end{corollary}

As a corollary of the equivalent equation constructed to prove Theorem \ref{Theorem: Smoothing}, we are able to establish unconditional well-posedness. By unconditional well-posedness, we mean that for every real $u_0\in H^s(\mathbb{T})$ there is a solution $u\in C^0_tH^s_x$ that is \textit{unique} in $C^0_tH^s_x$-- see \cite{erdougan2016dispersive, zhou1997uniqueness} for the KdV, \cite{kwon2012unconditional} for the modified KdV, and \cite{kato2018unconditional} for the fifth order KdV for $s\geq 3/2$. This differs from the well-posedness result above in that we make no appeal to the auxiliary space $Y^s$. Specifically, we find the following corollary to hold true.

\begin{corollary}\label{Theorem: Unconditional Well-posedness}
Let $s > 1$ and $u_0\in H^s(\mathbb{T})$. Then \eqref{Equation: Toy Fifth Order} is unconditionally well-posed in $H^s$.
\end{corollary}

It then follows by the work in Appendix \ref{Section: Appendix A} that Corollary \ref{Theorem: Unconditional Well-posedness} extends to the full equation. In particular, we find the following corollary.
\begin{corollary}\label{Theorem: Unconditional Well-posedness for full equation}
Let $s > 1$, $\alpha,\beta,\gamma\in\mathbb{R}$, and $u_0\in H^s(\mathbb{T})$. Then \eqref{Equation: General Fifth Order} is unconditionally well-posed in $H^s$.
\end{corollary}

It's important to note that the decay of the $\ell=0$ \textit{Parabolic Resonances} is not important for the conclusions of Corollaries \ref{Theorem: Unconditional Well-posedness} or Corollary \ref{Theorem: Unconditional Well-posedness for full equation}. All that is required for these to hold is the decay of the $\ell=1$ parabolic resonances, showing that unconditional well-posedness is, in a sense, easier to obtain than nonlinear smoothing. 

\begin{remark}
It's possible that a finer analysis, as in \cite{kwon2012unconditional}, could obtain the endpoint $s = 1$. 
\end{remark}

The outline of the paper is as follows. In Section \ref{Section: background} we define the slightly unusual Bourgain space we will be using. In Section \ref{Section: Reductions} we set out to perform differentiation-by-parts in order to reduce the equation to one that is easier to use. In particular, we construct $\sim$ and prove Lemma \ref{Lemma: Equation Rewrite}. Section \ref{Section: Duhamel Bounds} will be comprised of all of the estimates required to close the contraction and prove Theorem \ref{Theorem: Wellposedness for Toy with Pa}, and hence Theorem \ref{Theorem: Wellposedness for Toy}. In Section \ref{Section: Smoothing} we perform another differentiation-by-parts and prove the remaining estimates required to prove Theorem \ref{Theorem: Smoothing}. Lastly, we piggyback off of the equation derived in Section \ref{Section: Smoothing} in order to establish unconditional well-posedness in Section \ref{Section: UCWP}. In the appendices we justify the reduction to \eqref{Equation: Toy Fifth Order}, and demonstrate that the same cancellation properties extend to that equation. 

\section{Notation \& Background}\label{Section: background}
Let $n$ be the dual Fourier variable to $x\in\mathbb{T}$ and $\tau$ be the dual Fourier variable to $t\in\mathbb{R}$. We denote the Fourier transform of a function $u$ on the torus $\mathbb{T}$ by $\widehat{u}(n)$. When $u$ is a spacetime function on $\mathbb{T}\times\mathbb{R}$, we denote both the space-time Fourier transform and the spacial Fourier transform as $\widehat{u}$, as there will be little confusion in context. 

We write $A\lesssim_{\epsilon} B$ when there is a constant $0 < C(\epsilon)$ such that $|A|\leq C|B|$; $A\gg_\epsilon B$ to be the negation of $A\lesssim_\epsilon B$; and $A\sim_\epsilon B$ if, in addition, $B\lesssim_\epsilon A$. We will also write $g = O(f)$ to denote $|g|\lesssim |f|$ and $a+$  (resp. $a-$) to denote $a+\epsilon$ (resp. $a-\epsilon$) for all $\epsilon > 0$, with implicit constants depending on $\epsilon$.

Define $\langle h\rangle := (1+|h|^2)^{1/2}$ and $J^s$ to be the Fourier multiplier given by $\langle n\rangle^s$. Let $W_t = e^{t\partial_x^5}$ be the propagator for the fifth order linear group, and for clarity we ignore factors of $2\pi$. We then define the Bourgain space $X^{s,b}$ first introduced in the seminal paper \cite{bourgain1993fourier}, by
\[
    \|u\|_{X^{s,b}} = \|\langle n \rangle^{s}\langle \tau + n^5\rangle^{b}\widehat{u}\|_{L^2_\tau\ell^2_n} = \|\langle n\rangle ^s\langle \tau\rangle^b \widehat{u}(n, \tau - n^5\rangle)\|_{\ell^2_n L^2_\tau}=\|W_{-t}u\|_{H^s_xH^b_t},
\]
which measures, in a sense, how much $u$ deviates from the free solution $W_tu_0$. Given $T > 0$ we define the restricted Bourgain space as the space of equivalence classes of functions endowed with the norm
\begin{equation*}
\|u\|_{X^{s,b}_T} := \inf\{\|v\|_{X^{s,b}}\,|\,u|_{[0,T ]} = v|_{[0,T]}\},
\end{equation*}
with dual space $X^{-s,-b}.$ Because of this natural pairing and the fact that our norms are defined in terms of the absolute value of the Fourier transform, we will often assume that all Fourier transforms are non-negative and invoke duality, where it will be natural to define the hyper-plane by
\[
\Gamma_p := \{\tau-\tau_1 -\cdots -\tau_p = 0,\,\,n-n_1-\cdots -n_p = 0\},
\]
with obvious inherited measure denoted by $d\Gamma$. 

Due to the non-negativity of the Fourier transforms, we note that if $p(n,n_1, \cdots, n_k)$ is a symbol associated to a multilinear operator of the form
\[
P(u_1, \cdots, u_k)(t, x) = \mathcal{F}_{n}^{-1}\bigg(\sum_{n=n_1+\cdots n_k}p(n,n_1, \cdots, n_k)\prod_{i=1}^k\widehat{u}(n_i,t)\bigg)(x),
\]
and $|p|\leq cq$ for some symbol $q(n,n_1, \cdots, n_k)$ with associated multilinear operator $Q$, then 
\[
\|P(u_1, \cdots, u_k)\|_{X^{s,b}_T}\lesssim \|Q(u_1, \cdots,u_k)\|_{X^{s,b}_T}.
\]


Since we have mean conservation and the transformation $u\mapsto u - \int_\mathbb{T} u\,dx$ simply adds another lower order term to the linear group, we will assume that $u_0$ is mean-zero and that all indexed variables $n_j$ are summed over $\mathbb{Z}\setminus \{0\}.$ We will often need to deal with tuples $(n_1, \cdots, n_k)$ for some $k > 1$ and $n_i\in \mathbb{Z}\setminus \{0\}$, where it will be convenient to denote the decreasing rearrangement of $\{|n_i|\}_{i=1}^k$ by $n_i^*$. That is, $n_i^*$ denotes the $i$'th largest term, \textit{in absolute value}, among $n_1, \cdots, n_k.$

Due to the the method we employ, we will need to handle boundary terms in $X^{s,b}$ for $b > 0$ of the form
\[
\left\|\mathcal{F}_{x}^{-1}\left(\sum_{n=n_1+n_2}\frac{\widehat{u}(n_1)\widehat{u}(n_2)}{n_1n_2}\right)\right\|_{X^{s,b}}.
\]
Because of this, we will need to choose $b$ so that for $n_1\gg n$
\[
\left|(n_1+n_2)n_1n_2((n_1+n_2)^2+n_1^2+n_3^2)\right|^b\lesssim n_1,
\]
which necessitates the choice $b = 1/4$. This space, however, fails to control the $C^0_tH^s_x$ norm. To remedy this issue, we define the norms
\[
\|\cdot\|_{Y^s} := \|\cdot\|_{X^{s,1/4}}+\|\langle n\rangle^s\mathcal{F}_{x,t}(\cdot)\|_{\ell^2_n \ell^1_\tau}
\]
for the initial data, and
\[
\|\cdot\|_{Z^s}:= \|\cdot\|_{X^{s,-3/4}} + \left\|\frac{\langle n\rangle^s}{\langle \tau + n^5\rangle}\mathcal{F}_{x,t}(\cdot)\right\|_{\ell^2_n \ell^1_\tau}
\]
for the nonlinearity. We analogously define the restricted $Y^s_T$ space.

We now record some standard facts (See e.g. \cite{ginibre1997cauchy}):
\begin{lemma}
For any $\chi\in \mathcal{S}(\mathbb{R})$ and $f\in C^\infty_x(\mathbb{T})$,
\begin{align*}
    \|\chi(t)e^{it\partial_x^2}f\|_{Y^{s}}&\lesssim \|f\|_{H^s}\\
    \left\|\chi(t)\int_0^te^{i(t-s)\partial_x^2}F(s)\,ds\right\|_{Y^{s}}&\lesssim \|F\|_{Z^s}.
\end{align*}
\end{lemma}
Additionally, for $\eta\in\mathcal{S}(\mathbb{R})$ positive, supported on $[-2,2]$, and $1$ on $[-1,1]$, Sobolev embedding implies
\[
\|\eta(t/T)u\|_{L^q_tL^p_x}\lesssim T^{\frac{1}{q}}\|u\|_{Y^s_T}\qquad\qquad s > \frac{1}{2}-\frac{1}{p}\qquad\qquad p\geq2.
\]
Similarly, we have the well-known relationship between the modulation variable and the time, $T,$ given by
\[
\|u\|_{X^{s,b}_T}\lesssim T^{b'-b}\|u\|_{X^{s,b'}_T}\qquad\qquad -\frac{1}{2} < b\leq b'< \frac{1}{2}.
\]

We will also need the following Strichartz estimates.
\begin{lemma}
For $u$ on $\mathbb{T}\times\mathbb{R}$ and $\eta\in\mathcal{S}(\mathbb{R})$ defined previously, we have
\[
\|\eta(t/T)u\|_{L^4_{x,t}(\mathbb{T}\times\mathbb{R})}\lesssim \|u\|_{X^{0, \frac{3}{10}}_T},
\]
\[
\|\eta(t/T)u\|_{L^{6}_{x,t}(\mathbb{T}\times\mathbb{R})}\lesssim \|u\|_{X^{0+, \frac{1}{2}+}_T}.
\]
and
\[
\|\eta(t/T)u\|_{L^{30}_{x,t}(\mathbb{T}\times\mathbb{R})}\lesssim \|u\|_{X^{3/10+, \frac{1}{2}+}_T}.
\]
\end{lemma}
\begin{proof}
The proof of the first estimate is exactly the same as Bourgain's in \cite{bourgain1993fourier} since we again have a factorization of 
\[
(n_1+n_2)^5-n_1^5-n_2^5 = 5/2(n_1+n_2)n_1n_2(n^2+n_1^2+(n_1+n_2)^2).
\]

The second estimate follows from \cite{hu2015local} Theorem 1.4, and the last from \cite{hughes2019discrete} Theorem 3.1.
\end{proof}
Interpolation between the $L^4$ estimate with the trivial $L^2$ estimate gives,
\[
\|\eta(t/T)u\|_{L^{24/7}_{x,t}(\mathbb{T}\times\mathbb{R})}\lesssim \|u\|_{X^{0, 1/4}_T},
\]
whereas interpolation between the $L^6$ and the $L^{30}$ estimates above gives
\begin{equation}\label{Equation: L8 estimate}
\|\eta(t/T)u\|_{L^{8}_{x,t}(\mathbb{T}\times\mathbb{R})}\lesssim \|u\|_{X^{3/32+, 1/2+}_T}.
\end{equation}
\begin{remark}
Since we're above the Sobolev embedding level, the standard game one plays with Strichartz estimates loses its potency. However, the above estimates will still be useful in going below $s = 1$. 
\end{remark}
\section{Symbol Analysis}\label{Section: Symbol Analysis}
Before we proceed with the paper, we need a more number theoretic result that will enable us to handle large multilinear estimates. In order to state the proposition we need to allow for an abuse of notation that will be more convenient. We let $n_1^* + n_2^* = 0$ denote the case when the sum of the two largest (in magnitude) frequencies is zero. That is, if 
\[
|n_1|\geq |n_2|\geq\cdots \geq |n_p|,
\]
then $n_1^* + n_2^* = 0$ denotes the situation $n_1 + n_2 = 0$. For the remainder of this section we fix $n, n_1, \cdots, n_p$ so that 
\[
n = n_1+\cdots+n_p.
\]
With that in mind, we define
\[
\Phi_p(n_1, \cdots, n_p) := \bigg(\sum_{i=1}^p n_i\bigg)-\sum_{i=1}^p n_i^5 = n^5-\sum_{i=1}^p n_i^5,
\]

for which we find the following propositions.
\begin{proposition}\label{Proposition: small symbol decomp}
Let $|n_1|\geq |n_2|\geq |n_3|$. Then at least one of the following must be true
\begin{itemize}
    \item[I)] ``Resonance''- $n_i = n$ for some $1\leq i\leq 3$,
    \item[II)] $|n_3|\gtrsim |n_1|$,
    \item[III)] $|\Phi_3(n_1,n_2,n_3)| = |n^5-n_1^5-\cdots n_3^5|\gtrsim n_1^4$.
\end{itemize}
\end{proposition}
\begin{proof}
Note that we have the factorization
\begin{equation}\label{Equation: Small SYmbol Decopm Factor}
n^5-n_1^5-n_2^5- n_3^5 = \frac{5}{2}(n_1+n_2)(n_1+n_3)(n_2+n_3)(n^2+n_1^2+n_2^2+n_3^2).
\end{equation}
If we first assume that there is no resonance, then we find that the above is non-zero. Assuming now that $|n_3|\ll |n_1|$, we find that $|(n_1+n_3)(n_1+n_2)(n_2+n_3)|\gtrsim n_1^2$, so that 
\[
|(n_1+n_2)(n_1+n_3)(n_2+n_3)(n^2+n_1^2+n_2^2+n_3^2)|\gtrsim n_1^4,
\]
as desired.
\end{proof}

Piggybacking off of Proposition \ref{Proposition: small symbol decomp} we establish the following more general decomposition lemma.

\begin{proposition}\label{Proposition: Symbol Decomp}
Let $|n_1|\geq |n_2|\geq\cdots \geq |n_p|$, and $p\geq 4$. Denote $\tilde{n}_i = \sum_{j=i}^pn_j$, so we have that $n=\tilde{n}_1$ and $n_p = \tilde{n}_p$. Then at least one of the following must be true
\begin{itemize}
    \item[I)] ``Resonance''- $n_i = n$ for some $1\leq i\leq p$,
    \item[II)] $n_1^*+n_2^*=0$,
    \item[III)] $|n_3|\gtrsim |n_1|$,
    \item[IV)] $|\Phi_p(n_1, \cdots, n_p)| = |n^5-n_1^5-\cdots n_p^5|\gtrsim n_1^4$,
    \item[V)] $n_3^4|n_4|\gtrsim n_1^4$.
\end{itemize}
\end{proposition}
\begin{remark}\label{Remark: Redundant Case}
The third condition is also contained in the last condition by mean-zero preservation, so we will often drop it and handle it with the case $n_3^4n_4\gtrsim n_1^4$. Additionally, we don't care about the specifics of the decomposition of the Fourier space into these cases-- any such will work.
\end{remark}
\begin{proof}
For $p\geq 4$ we may write
\begin{align*}
n^5-n_1^5-\cdots n_p^5 &= (n^5-n_1^5-n_2^5 -\tilde{n}_3^5) + (\tilde{n}_3^5-n_3^5-\tilde{n}_4^5)+\cdots+(\tilde{n}_{p-1}^5-n_{p-1}^5-n_p^5)\\
&=\Phi_3(n_1,n_2,\tilde{n}_3) + \sum_{j=3}^{p-1}\Phi_2(n_j,\tilde{n}_{j+1}).
\end{align*}
Hence, if there is no resonance, $n_1+n_2\ne 0$, and $|n_3|\ll |n_1|$, then by \eqref{Equation: Small SYmbol Decopm Factor}
the first term satisfies
\begin{align*}
|\Phi_3(n_1,n_2,\tilde{n}_3)| &= 5/2|(n_1+n_2)(n_1+\tilde{n}_3)(n_2+\tilde{n}_3)|(n^2+n_1^2+n_2^2+\tilde{n}_3^2)\gtrsim n_1^4.
\end{align*}
Indeed, the first term is non-zero by assumption, and the lack of resonance ensures that $n_1+\tilde{n}_3, n_2+\tilde{n_3}\ne 0$. Additionally, if $|n_1|\gg |n_2|$, then the first two factors are, in magnitude, $\gtrsim |n_1|$ each, and if $|n_2|\sim |n_1|$, then the middle factors must each be $\gtrsim |n_1|$ in magnitude. 

It follows that if $|\Phi_p(n_1, \cdots, n_p)|\ll n_1^4$, then we must have
\[
\big|\sum_{j=3}^{p-1}\Phi_2(n_j,\tilde{n}_{j+1})\big|\gtrsim n_1^4.
\]
However, the remaining terms must satisfy
\[
|\Phi_2( n_j, \tilde{n}_{j+1})| = |\tilde{n}_jn_j\tilde{n}_{j+1}|(\tilde{n}_j^2+n_j^2+\tilde{n}_{j+1}^2)\lesssim n_3^4|n_{4}|,
\]
so that 
\[
n_1^4\lesssim \big|\sum_{j=3}^{p-1}\Phi_2(n_j,\tilde{n}_{j+1})\big|\lesssim n_3^4|n_4|,
\]
which is the final case.
\end{proof}
\begin{remark}\label{Remark: P set definition}
For brevity, we denote the property that $|\Phi_p(n_1, \cdots, n_p)|\gtrsim (n_1^*)^4$ by $\mathcal{P}_p(n_1, \cdots, n_p)$. Similarly, we denote the property that $|\Phi_p(n_1,\cdots, n_p)|\ll (n_1^*)^4$ by $\mathcal{P}_p^c(n_1,\cdots,n_p)$.
\end{remark}

\section{Reductions}\label{Section: Reductions}
Before continuing with the reductions, we warn the reader: in the interest of readability we have, by and large, chosen not to track what the explicit constants are. This pattern is only broken when explicit cancellation is sought, but in most other situations we state results modulo constants.

This section is organized so that each application of differentiation-by-parts is separated into a subsection. The overarching goal of this section is the following lemma.
\begin{lemma}\label{Lemma: Equation Rewrite}
Let $a\geq 0$ be such that \eqref{Definition: Property Pa} holds, $s>\frac{1+a}{2}$ and $u_0\in H^s(\mathbb{T})$. Then there is an invertible transformation $u\mapsto \tilde{u}$ with $\|\tilde{u}\|_{C^0_tH^s_x} = \|u\|_{C^0_tH^s_x}$ and $u(x,0) = \tilde{u}(x,0)$, such that if $u$ satisfies \eqref{Equation: Fourier Toy Fifth 1}, then $\tilde{u}$ is a fixed point of \eqref{Equation: Contractive Operator Definition}.
\end{lemma}

As mentioned, the proof of the above lemma follows from a differentiation-by-parts process. This fundamentally is the observation that
\begin{multline*}
    \widehat{u}_t = \sum_{\Omega\ne0}e^{it\Omega}m(n_1,n_2)\widehat{u}(n_1)\widehat{u}(n_2)\\
                  =\sum_{\Omega\ne0}\frac{\partial_te^{it\Omega}m(n_1,n_2)}{i\Omega}\widehat{u}(n_1)\widehat{u}(n_2)
                  =\partial_t\sum_{\Omega\ne 0}\frac{e^{it\Omega}m(n_1,n_2)}{i\Omega}\widehat{u}(n_1)\widehat{u}(n_2)\\
                  -\sum_{\Omega\ne 0}\frac{e^{it\Omega}m(n_1,n_2)}{i\Omega}\partial_t\widehat{u}(n_1)\widehat{u}(n_2)
                  -\sum_{\Omega\ne 0}\frac{e^{it\Omega}m(n_1,n_2)}{i\Omega}\widehat{u}(n_1)\partial_t\widehat{u}(n_2),
\end{multline*}
where $m$ is some time-independent symbol, such as a polynomial in $n_1$ and $n_2$. In our situation we then have a formula for $\partial_t\widehat{u}$ from the differential equation, and we may substitute this in and hope to use $\Omega$ to mitigate losses from $m$.

The general strategy will then be to separate portions of the sum over which $\Omega$ is possibly very small -- \textit{Resonant} or \textit{nearly Resonant} portions -- from the portions where $\Omega$ is large. The portions over which it is small will dictate the well-posedness level, and hence will be of the most interest. 

Now, by the work in Appendix \ref{Section: Appendix A} we see that the existence of a solution to the the above equation is equivalent to a solution to the fifth order KdV, we simply resign ourselves to proving Lemma \ref{Lemma: Equation Rewrite} and the associated bounds on it, and iteratively constructing the transformation $\sim$.

\subsection{First Differentiation-by-Parts}

We symmetrize the nonlinearity and rewrite \eqref{Equation: Toy Fifth Order} on the Fourier side as 
\begin{align}\label{Equation: Fourier Toy Fifth 1}
    \begin{cases}
        \widehat{u}_t +in^5 \widehat{u} - in\sum_{n_1+n_2 = n}(n_1^2+n_2^2)\widehat{u}(n_1)\widehat{u}(n_2) = 0\\
        \widehat{u}(n, 0) = \widehat{u_0}(n).
    \end{cases}
\end{align}
We now perform the change of variables
\begin{equation}\label{Definition: v}
    \widehat{v}(n) =\widehat{W_{-t}u}(n) = e^{-itn^5}\widehat{u}(n),
\end{equation}
so that we obtain
\begin{align}\label{Equation: Fourier Toy Fifth 2}
    \begin{cases}
        \widehat{v}_t(n) = in\sum_{n_1+n_2 = n}(n_1^2+n_2^2)e^{-it\Phi_2(n,n_1,n_2)}\widehat{v}(n_1)\widehat{v}(n_2) \\
        \widehat{v}(n, 0) = \widehat{u_0}(n),
    \end{cases}
\end{align}
for 
\begin{equation}\label{Definition: Phi2}
    \Phi_2(n_1,n_2) = (n_1+n_2)^5-n_1^5-n_2^5 = \frac{5}{2}nn_1n_2(n^2+n_1^2+n_2^2),
\end{equation}
when $n=n_1+n_2$.

Notice that if either frequency is equal to $n$, then conservation of mean forces the term to vanish. Hence we may assume that the phase doesn't vanish and differentiation-by-parts yields
\begin{align}
    \partial_t\widehat{v}(n) &= -\partial_t\sum_{n = n_1+n_2}\frac{n(n_1^2+n_2^2)}{\Phi_2(n,n_1,n_2)}e^{-it\Phi_2(n_1,n_2)}\widehat{v}(n_1)\widehat{v}(n_2)\nonumber\\
    &\qquad+ 2\sum_{n = n_1+n_2}\frac{n(n_1^2+n_2^2)}{\Phi_2(n_1,n_2)}e^{-it\Phi_2(n,n_1,n_2)}\widehat{v}(n_1)\partial_t\widehat{v}(n_2)\nonumber\\
    &=-\partial_t\sum_{n = n_1+n_2}\frac{n(n_1^2+n_2^2)}{\Phi_2(n_1,n_2)}e^{-it\Phi_2(n,n_1,n_2)}\widehat{v}(n_1)\widehat{v}(n_2)\label{Equation: First DBP}\\
    &\qquad+ 2i\sum_{\substack{n = n_1+n_2+n_3\\ n_2+n_3\ne 0\\n,n_1,n_2,n_3\ne 0}}\frac{n(n_1^2+(n_2+n_3)^2)(n_2+n_3)(n_2^2+n_3^2)}{\Phi_2(n_1,n_2+n_3)}\nonumber\\
    &\qquad\qquad\qquad\times e^{-it\Phi_3(n_1,n_2,n_3)}\widehat{v}(n_1)\widehat{v}(n_2)\widehat{v}(n_3)\nonumber,
\end{align}
where 
\begin{multline}\label{Definition Phi3}
    \Phi_3(n_1,n_2,n_3) = \big(\sum_{j=1}^3n_j\big)^5-n_1^5-n_2^5-n_3^5\\
    = \frac{5}{2}(n_1+n_2)(n_2+n_3)(n_1+n_3)(n^2+n_1^2+n_2^2+n_3^2).
\end{multline}

While we had no resonances in the nonlinearity of \eqref{Equation: Fourier Toy Fifth 1}, this new trilinear term has resonance. Specifically, if $n_1+n_2 = 0$ or $n_1+n_3 = 0$ then we won't be able to use modulation considerations-- that is, $n=n_2$ or $n=n_3$. If $n_2 = n$, then by \eqref{Definition: Phi2} we have $n_1=-n_3$ and the fraction becomes
\begin{align}
    &\frac{4i}{5}\frac{n(n_3^2+(n+n_3)^2)(n+n_3)(n^2+n_3^2)}{nn_3(n+n_3)(n^2+n_3^2+(n+n_3)^2)} = \frac{4i}{5}\frac{(n_3^2+(n+n_3)^2)(n^2+n_3^2)}{n_3(n^2+n_3^2+(n+n_3)^2)}\nonumber\\
    &\qquad\quad=\frac{4i}{5}\left\{\frac{n^2+n_3^2}{n_3} - \frac{n^2}{2n_3}+\frac{1}{2}n - \frac{n^2n_3+nn_3^2}{n^2+(n+n_3)^2+n_3^2}\right\}.\label{Equation: Resonant Multiplier}
\end{align}
Incorporating the cases when $n_3=n$ and $n_2=n_3=n$ we find that the resonant sum will then be
\begin{align*}
    \frac{8i}{5}\widehat{v}(n)\sum_{n_3}&\left\{\frac{n^2+n_3^2}{n_3} - \frac{n^2}{2n_3}+\frac{1}{2}n - \frac{n^2n_3+nn_3^2}{n^2+(n+n_3)^2+n_3^2}\right\}\widehat{v}(n_3)\widehat{v}(-n_3)\\
    &\qquad\qquad- \frac{4i}{3}n\widehat{v}(n)\widehat{v}(n)\widehat{v}(-n).
\end{align*}

Since $\widehat{v}(n_3)\widehat{v}(-n_3)$ is invariant under the sign of $n_3$, we see that the resonant sum weighted by the first two summands of \eqref{Equation: Resonant Multiplier} will vanish. It follows that the full resonant portion of \eqref{Equation: First DBP} reduces to
\begin{align}\label{Equation: First Resonance}
    \frac{4i}{5}n\widehat{v}(n)&\sum_{n_3}\widehat{v}(n_3)\widehat{v}(-n_3)\\
    &\qquad-\frac{8i}{5}\widehat{v}(n)\sum_{n_3}\frac{n^2n_3+nn_3^2}{n^2+(n+n_3)^2+n_3^2}\widehat{v}(n_3)\widehat{v}(-n_3)\nonumber\\
    &\qquad-\frac{4i}{3}n\widehat{v}(n)\widehat{v}(n)\widehat{v}(-n).\nonumber
\end{align}
The second and third portions of the above are suitable to estimate, whereas the first we must remove. 

Letting,
\[
\mathcal{K}(v)(t) = \frac{4i}{5}\int_\mathbb{T}v^2(x,t)\,dx = \frac{4i}{5}n\sum_{n_3}\widehat{v}(n_3)\widehat{v}(-n_3),
\]
we define
\begin{equation}\label{Definition: W}
    \widehat{v}(n) = e^{n\int_0^t\mathcal{K}(v)\,ds}\widehat{w}(n) = e^{n\int_0^t\mathcal{K}(w)\,ds}\widehat{w}(n).
\end{equation}
\begin{remark}
As $v$ is real, the above quantity, $\mathcal{K}$, is technically an imaginary multiple of its $L^2_x$ norm. We write it in this fashion because it is not technically conserved under the flow of \eqref{Equation: Toy Fifth Order}.
\end{remark}
It follows that
\begin{equation*}
    \partial_t \widehat{v}(n) = n\mathcal{K}(v)\widehat{v}(n) + e^{n\int_0^t\mathcal{K}(v)\,ds}\partial_t \widehat{w}(n).
\end{equation*}

Using \eqref{Definition: W} and the display below it, we find $\widehat{w}$ satisfies (looking at \eqref{Equation: Fourier Toy Fifth 1})
\begin{align}\label{Equation: Fourier Toy Fifth 3: for w}
    \begin{cases}
        \widehat{w}_t(n) = in\sum_{n_1+n_2 = n}(n_1^2+n_2^2)e^{-it\Phi_2(n_1,n_2)}\widehat{w}(n_1)\widehat{w}(n_2)\\
        \qquad\qquad- n\mathcal{K}(w)\widehat{w}(n) \\
        \widehat{w}(n, 0) = \widehat{u_0}(n).
    \end{cases}
\end{align}
Performing the differentiation-by-parts procedure again, we see
\begin{align}
\partial_t w &=-\partial_t\sum_{n = n_1+n_2}\frac{n(n_1^2+n_2^2)}{\Phi_2(n_1,n_2)}e^{-it\Phi_2(n_1,n_2)}\widehat{w}(n_1)\widehat{w}(n_2)\label{Equation: Second DBP}\\
    &\qquad+2i\sum_{\substack{n = n_1+n_2+n_3\\\star}}\frac{n(n_1^2+(n_2+n_3)^2)(n_2+n_3)(n_2^2+n_3^2)}{\Phi_2(n_1,n_2+n_3)}\label{Equation: Second DBP trlinear term}\\
    &\qquad\qquad\qquad\times e^{-it\Phi_3(n_1,n_2,n_3)}\widehat{w}(n_1)\widehat{w}(n_2)\widehat{w}(n_3)\nonumber\\
    &\qquad-\mathcal{K}(w)\sum_{\substack{n = n_1+n_2}}\frac{(n_1+n_2)(n_1^2+n_2^2)}{\Phi_2(n_1,n_2)}e^{it\Phi_2(n_1,n_2)}\widehat{w}(n_1)\widehat{w}(n_2)\nonumber\\
    &\qquad -\frac{8i}{5}\widehat{w}(n)\sum_{n_3}\frac{n^2n_3+nn_3^2}{n^2+(n+n_3)^2+n_3^2}\widehat{w}(n_3)\widehat{w}(-n_3) - \frac{4i}{3}n\widehat{w}(n)\widehat{w}(n)\widehat{w}(-n)\nonumber,
\end{align}
where $\star$ denotes the non-resonance restriction-- the complement of case 1 of Proposition \ref{Proposition: Symbol Decomp}. Note the cancellation of the single resonance in the above, as opposed to \eqref{Equation: First DBP}.

\subsection{Second Differentiation-by-Parts}\label{Definition: Second DBP terms}\hfill

In this section we perform the second differentiation-by-parts and obtain the equation we will work with. Specifically, notice that when $|n_3|\gg \max(|n_1|, |n_2|)$, then the trilinear term, \eqref{Equation: Second DBP trlinear term}, has, morally, two derivatives. Since we can only recover one derivative from modulation considerations, we are forced to perform another differentiation-by-parts procedure for this term.  There's nothing wrong, in principle, with performing the differentiation-by-parts on the full trilinear term, but with the aid of Proposition \ref{Proposition: small symbol decomp} we can do better.

We now recall Remark \ref{Remark: P set definition} and define the symbols
\begin{align*}
        m_{D_1} &= \frac{n((n_1+n_2)^2+(n_3+n_4)^2)(n_3+n_4)(n_3^2+n_4^2)(n_1+n_2)(n_1^2+n_2^2)}{\Phi_2(n_1+n_2,n_3+n_4)\Phi_3(n_1+n_2,n_3, n_4)}\chi_{\mathcal{P}_3(n_1+n_2, n_3, n_4)}\\
    m_{D_2} &=\frac{n(n_1^2+(n_2+n_3+n_4)^2)(n_2+n_3+n_4)(n_2^2+(n_3+n_4)^2)(n_3+n_4)(n_3^2+n_4^2)}{\Phi_2(n_1,n_2+n_3+n_4)\Phi_3(n_1,n_2, n_3+n_4)}\\
    &\qquad\qquad\times\chi_{\mathcal{P}_3(n_1,n_2,n_3+n_4)}\\
    m_{A_i} &= m_{D_1}\chi_{n=n_i} \qquad i=1,2\\
    m_{B_i} &= m_{D_2}\chi_{n=n_i} \qquad i=3,4,
\end{align*}
and let
\begin{align}
    \widehat{B_1}(\phi)(n) &= \sum_{\substack{n = n_1+n_2\\\star}}\frac{n(n_1^2+n_2^2)}{\Phi_2(n_1,n_2)}\widehat{\phi}(n_1)\widehat{\phi}(n_2)\nonumber\\
    \widehat{B_2}(\phi)(n) &= \sum_{\substack{n=n_1+n_2+n_3\\ \mathcal{P}_3(n_1,n_2,n_3)\\\star}}\frac{n(n_1^2+(n_2+n_3)^2)(n_2+n_3)(n_2^2+n_3^2)}{\Phi_2(n_1,n_2+n_3)\Phi_3(n_1,n_2,n_3)} \widehat{\phi}(n_1)\widehat{\phi}(n_2)\widehat{\phi}(n_3)\nonumber\\
    \widehat{D_1}(\phi)(n) &=\sum_{\substack{n=n_1+n_2+n_3+n_4\\\star}}m_{D_1} \widehat{\phi}(n_1)\widehat{\phi}(n_2)\widehat{\phi}(n_3)\widehat{\phi}(n_4)\nonumber\\
    \widehat{D_2}(\phi)(n) &=\sum_{\substack{n=n_1+n_2+n_3+n_4\\\star}}m_{D_2} \widehat{\phi}(n_1)\widehat{\phi}(n_2)\widehat{\phi}(n_3)\widehat{\phi}(n_4)\nonumber\\
    \widehat{\mathcal{NR}_1}(\phi)(n)&= \sum_{\substack{n = n_1+n_2+n_3\\\mathcal{P}_3^c(n_1,n_2,n_3)\\\star}}\frac{n(n_1^2+(n_2+n_3)^2)(n_2+n_3)(n_2^2+n_3^2)}{\Phi_2(n_1,n_2+n_3)}\widehat{\phi}(n_1)\widehat{\phi}(n_2)\widehat{\phi}(n_3)\nonumber,
\end{align}
where $\star$ again corresponds to the lack of resonance. That is, we write
\begin{multline*}
    \eqref{Equation: Second DBP trlinear term} = -2\partial_t\widehat{W_{-t}}\widehat{B}_2(\widehat{W_t}w)+\widehat{W_{-t}}\big\{2i\widehat{\mathcal{NR}}_1(\widehat{W_t}w)\\
    +4i(\widehat{D}_1(\widehat{W_t}w)+\widehat{\mathcal{RE}}_3(\widehat{W}_tw)+2\widehat{D}_2(\widehat{W_t}w)+2\widehat{\mathcal{RE}}_4(\widehat{W}_tw))\\
    -2\widehat{R}_2(\widehat{W_t}w)\big\} + \text{mult. resonances},
\end{multline*}
where $R_2$ is defined below. Notice that $m_{D_1}$ corresponds to $\partial_t$ landing on the $n_1$ term, while $m_{D_2}$ corresponds to $\partial_t$ landing on $n_2$ or $n_3$. $m_{A_i}$ and $m_{B_i}$ then correspond to the symbols of the additional resonant terms. Notice that we have separated out the region $\Phi_3\ll (n_1^*)^4$, and only differentiated-by-parts the non-resonant complement of this region.

The prior terms will be the main terms of study, but we will also need to define the additional cast of characters:
\begin{align}
    \widehat{R_1}(\phi)(n) &=\mathcal{K}(\phi)\sum_{\substack{n = n_1+n_2\\\star}}m_{R}^1(n_1+n_2)\widehat{\phi}(n_1)\widehat{\phi}(n_2)\label{Definition: R1}\\
    \widehat{R_2}(\phi)(n) &=\mathcal{K}(\phi)\sum_{\substack{n = n_1+n_2+n_3\\\star}}m_{R}^2(n_1+n_2+n_3)\widehat{\phi}(n_1)\widehat{\phi}(n_2)\widehat{\phi}(n_3)\label{Definition: R4}\\
    \widehat{\mathcal{RE}_1}(\phi)(n) &=\frac{8i}{5}\widehat{\phi}(n)\sum_{n_3}\frac{n^2n_3+nn_3^2}{n^2+(n+n_3)^2+n_3^2}\phi(n_3)\widehat{\phi}(-n_3)\label{Definition: RE1}\\
    \widehat{\mathcal{RE}_2}(\phi)(n) &=\frac{4i}{3}n\widehat{\phi}(n)\widehat{\phi}(n)\widehat{\phi}(-n) + \mbox{Second Order Resonances}\label{Definition: RE2}\\
    \widehat{\mathcal{RE}_3}(\phi)(n) &=\widehat{\phi}(n)\sum_{\substack{0=n_2+n_3+n_4\\n_3,n_4\ne n\\\mathcal{P}_3(n+n_2, n_3, n_4)}}m_{A_1}\widehat{\phi}(n_2)\widehat{\phi}(n_3)\widehat{\phi}(n_4)\label{Definition: RE3}\\
    \widehat{\mathcal{RE}_4}(\phi)(n)&=\widehat{\phi}(n)\sum_{\substack{0=n_1+n_2+n_4\\n_1,n_2\ne n\\ \mathcal{P}_3(n,n_1,n_2,n+n_4)}}m_{B_3}\widehat{\phi}(n_1)\widehat{\phi}(n_2)\widehat{\phi}(n_4)\label{Definition: RE4},
\end{align}
with symbols $m_{R}^i$, $i=1,2$ defined by
\begin{align*}
    m_{R}^1 &=\frac{n(n_1^2+n_2^2)}{\Phi_2(n_1,n_2)}\\
    m_{R}^2 &= \frac{n(n_1^2+(n_2+n_3)^2)(n_2+n_3)(n_2^2+n_3^2)}{\Phi_2(n_1,n_2+n_3)\Phi_3(n_1,n_2,n_3)}\chi_{\mathcal{P}_3(n_1,n_2,n_3)}.
\end{align*}
Note that these symbols are simply the symbols that arise (up to symmetry) after performing the differentiation-by-parts procedure and substituting the correction terms into the temporally differentiated term.
\begin{remark}\label{Remark: Second Order Resonances D_1 and D-2}
The
\[
\mbox{Second Order Resonances}\]
term in $\mathcal{RE}_2$ contains the double resonances coming from $n_1=n_2=n$ in $D_1$ and $n_3=n_4=n$ in $D_2$. By the next two lemmas we find that the symbols satisfy
\[
m_{D_1}(n,n,n_3,n_4), m_{D_2}(n_1,n_2,n,n) = O(n).
\]
It follows that we will handle them in nearly the exact the same way as the first term in $\mathcal{RE}_2$.
\end{remark}
It's important to notice that the above terms lack the complex exponential terms that have phases given by $\Phi_\ell$. 
We choose the above representations for ease of representation of the Duhamel formula, \eqref{Equation: Contractive Operator Definition}, that will come. It is recommended that a reader, on a first read through of this section, first compares the terms above to the final form of the equation, \eqref{Equation: Tilde u differential equation before smoothing}, before moving on to the following lemmas. 

We now perform analysis on several of the symbols that appear above which will be useful later.

\begin{lemma}\label{Lemma: Nonresonant Weight Bound}
With $D_1$, $D_2,$ $m_{D_1}$, and $m_{D_2}$ defined as above, we have:
\begin{equation}\label{Equation: mD1 nonresonant bound}
    m_{D_1} = O\left(\frac{\max(n_1^2,n_2^2)}{\max\{|n_1+n_2|, |n_3|, |n_4|\}^2}\right),
\end{equation}
and
\begin{equation}\label{Equation: mD2 nonresonant bound}
    m_{D_2}  = O\left(\frac{\max(n_3^2,n_4^2)}{n_1\max\{|n_1|,|n_2|, |n_3+n_4|\}}\right).
\end{equation}
\end{lemma}
\begin{proof}
This follows immediately from the summation restriction, the definition of the symbols, and the factorization of $\Phi_2$.
\end{proof}

While $D_1$ and $D_2$ will correspond to the worst terms integrated in the Duhamel formula, we still have access to Proposition \ref{Proposition: Symbol Decomp} that will enable us to handle the $D_1$ and $D_2$. This leaves us with having to understand the resonances that we removed from these two terms, which occur when:
\begin{itemize}
    \item[A)] In $D_1$: $n_3,n_4\ne n$ and
    \begin{itemize}
        \item[$A_1$)] $n_2 +n_3+n_4 = 0$,
        \item[$A_2$)] $n_1+n_3+n_4=0$,
    \end{itemize}
    \item[B)] In $D_2$: $n_1,n_2\ne n$ and
    \begin{itemize}
        \item[$B_3)$] $n_1+n_2+n_4 = 0$,
        \item[$B_4)$] $n_1+n_2+n_3 = 0$.
    \end{itemize}
\end{itemize}
Note that Cases $A_i$ $i=1,2$ correspond to $D_1$ with symbol $m_{A_i}$ and similarly for cases $B_i$ and $m_{B_i}$.

\begin{corollary}\label{Corollary: Resonant Multiplier Bounds}
Suppose that $n_1,n_2,n_3,n_4$ are fixed, $n=n_1+n_2+n_3+n_4$, and $|n|\gg n_2^*$.
\begin{itemize}
\item[I)] If $n=n_i$ for $i\in\{1,2\}$, then we have
\begin{equation}\label{Equation: mA resonant bound}
    m_{A_i} = O\left(\frac{n_2^*}{n}\right).
\end{equation}

\item[II)] If $n=n_i$ for $i\in\{3,4\}$ and $j \in \{3,4\}\setminus\{i\}$, then we have
\begin{equation}\label{Equation: mB resonant bound}
    m_{B_i} +\frac{1}{25}\frac{n}{n_1n_j} + \frac{3n_2+6n_j}{25n_1n_j}= O\left(\frac{\max(|n_2|^2, |n_j|^2)}{n}\right).
\end{equation}
\end{itemize}
\end{corollary}
\begin{proof}
Before beginning, we would like to comment on the $n_i^*$ notation in the setting of the above lemma. In the context of case $I$ with $i=1$ we have that $n_i^*$ is the $i'$th largest entry of $(|n|,|n_2|, |n_3|,|n_4|)$. And, as $n_2+n_3+n_4=0$, we always have that $|n_{j_1}|\sim |n_{j_2}|$ where $j_1\ne j_2$ and $|n_{j_1}|= \max(|n_2|, |n_3|, |n_4|).$

The first portion follows almost immediately from Lemma \ref{Lemma: Nonresonant Weight Bound}. Instead, we find that
\begin{equation*}
|m_{A_i}| \lesssim \bigg|\frac{(n_1^2+n_2^2)(n_3^2+n_4^2)}{\max\{|n_1+n_2|,|n_3|,|n_4|\}^4}\bigg| \lesssim \bigg|\frac{n^2(n_2^*)^2}{n^4}\bigg|\lesssim \frac{n_2^*}{|n|}, 
\end{equation*}
as desired.

For the second portion, we assume by symmetry that $n_3=n$ and note that the expansion of the numerator of $m_{B_i}$ is
\[
n^9+(3n_2+6n_4)n^8+l.o.t.,
\]
where $l.o.t.$ denotes terms with $n$ degree strictly less than $8$ and total degree equal to $9$. Similarly, we have for for $n\gg n_2^*$ and the factorizations for $\Phi_2$ and $\Phi_3$ that
\begin{align*}
\Phi_2(n_1,n_2+n+n_4)&\Phi_3(n_1,n_2, n+n_4) \\
&= 25n^8n_1(n_1+n_2)+O(n^7n_4^2n_1+n^7n_4n_1^2)\\
&= -25n^8n_4n_1 +O(n^7n_4^2n_1+n^7n_4n_1^2)
\end{align*}
that we have the desired expansion by crude big-$O$ manipulation.




\end{proof}

\begin{remark}\label{Remark: Why Resonant Smoother Bound Holds}
By symmetry we have that 
\begin{align*}
\sum_{n_1+n_2+n_4=0}\frac{1}{n_1n_4}\widehat{u}(n_1)\widehat{u}(n_2)\widehat{u}(n_4) = \sum_{\substack{n_1+n_2+n_4=0\\ |n|\gg \max(|n_1|,|n_2|,|n_4|)}}\frac{1}{n_1n_4}\widehat{u}(n_1)\widehat{u}(n_2)\widehat{u}(n_4)
=0.
\end{align*}
Additionally, we have by construction that
\[
\frac{3n_2+6n_4}{n_1n_4} = \frac{3(n_2+n_4)}{n_1n_4} + \frac{3}{n_1} = \frac{-3}{n_4} + \frac{3}{n_1},
\]
and hence by symmetry that
\[
\sum_{n_1+n_2+n_4=0}\frac{3n_2+6n_4}{n_1n_4}\widehat{u}(n_1)\widehat{u}(n_2)\widehat{u}(n_4) = \sum_{\substack{n_1+n_2+n_4=0\\|n|\gg\max(|n_1|,|n_2|,|n_4|)}}\frac{3n_2+6n_4}{n_1n_4}\widehat{u}(n_1)\widehat{u}(n_2)\widehat{u}(n_4)= 0,
\]
for real $u$.

\end{remark}
\begin{remark}
It's important to observe that the second cancellation structure explored in Remark \ref{Remark: Why Resonant Smoother Bound Holds} is unimportant for proving local well-posedness. The use of the above is purely in obtaining a better asymptotic bound in order to establish smoothing.

Similarly, the cancellation of the first term is unimportant for the well-posedness of the system. Because there is only one power of $n$ associated with it we could have removed it in exactly the same way that we removed $\mathcal{K}$.
\end{remark}
We now undo the interaction representation. It follows that if
\begin{equation}\label{Equation: tilde u definition}
\widehat{\tilde{u}}(n) = e^{n\int_0^t\mathcal{K}(\tilde{u})\,ds}\widehat{u}(n) = e^{n\int_0^t \mathcal{K}(v)\,ds}\widehat{u}(n),
\end{equation}
then $\tilde{u}$ satisfies (modulo constants)
\begin{align}\label{Equation: Tilde u differential equation before smoothing}
    \partial_t& W_{-t}\tilde{u} = \partial_tW_{-t}B_1(\tilde{u})+\partial_tW_{-t}B_2(\tilde{u})+W_{-t}D_1(\tilde{u})+W_{-t}D_2(\tilde{u})\\
    &\qquad+W_{-t}(R_1(\tilde{u})+R_2(\tilde{u})+\mathcal{NR}_1(\tilde{u}))\nonumber\\
    &\qquad+W_{-t}(\mathcal{RE}_1(\tilde{u})+\mathcal{RE}_2(\tilde{u})+\mathcal{RE}_3(\tilde{u})+\mathcal{RE}_4(\tilde{u}))\nonumber.
\end{align}
with 
\[
\tilde{u}(x,0) = u_0(x).
\]
As the the transformation mapping $u$ to $\tilde{u}$ is bi-continuous on $H^s$ for $s>1/2$, we conclude the proof of Lemma \ref{Lemma: Equation Rewrite}.

Similarly, by invoking Duhamel we must have, modulo constants, that $\tilde{u}$ is a fixed point of $\Gamma$ defined by
\begin{align}\label{Equation: Contractive Operator Definition}
    \Gamma[\phi] &:= B_1(\phi)(x,t) + B_2(\phi)(x,t) -W_t B_1(u_0)(x) - W_t B_2(u_0)(x) + W_t u_0(x)\\ 
    &+\int_0^t W_{t-s}\big\{(D_1+D_2 +\mathcal{NR}_1)(\phi)(x,s)+(R_1+R_2)(\phi)(x,s)\nonumber\\
    &\quad+(\mathcal{RE}_1+\mathcal{RE}_2+\mathcal{RE}_3+\mathcal{RE}_4)(\phi)(x,s)\big\}\,ds.\nonumber
\end{align}

In the next few sections we will estimate the portions of $\Gamma$ in order to establish that $\Gamma$ is a contraction in $Y^s$.

\section{Duhamel Estimates \& Proof of Theorem \ref{Theorem: Wellposedness for Toy}}\label{Section: Duhamel Bounds}
Space considerations necessitate the introduction of an extra parameter, $\varepsilon.$ This parameter is a measure of extra smoothness and will help shorten the subsequent sections and reduce repetition of the arguments. It is recommended that the reader ignore the $\varepsilon$ factors on a first read through of local well-posedness, and return to them when they reach the nonlinear smoothing section. 

We begin this section with a general estimate that will reduce repetitive arguments when we have modulation considerations.
\begin{proposition}\label{Proposition: general modulation lemma}
Let $\varphi_i\in Y^{1/2+}_T$ for $i\in \{1, \cdots, m\}$ and denote $\mathcal{F}_x(\varphi_i)(n) = \widehat{\varphi}_n$. Suppose that $\varphi_i$ satisfy $supp\,\mathcal{F}_x(\varphi_i)\subset [-N_i, N_i]$ for $N_i$ dyadic with $N_1\geq \cdots\geq N_m$. Then there is $\theta > 0$ so that
\begin{align*}
\left\|P_N\mathcal{F}_x^{-1}\left(\sum_{\substack{n=n_1+\cdots+n_m\\ \mathcal{P}_m(n_1,\cdots,n_m)}}\prod_{i=1}^m \widehat{\varphi}_{n_i}\right)\right\|_{Z^0_T}\lesssim \frac{T^\theta N^{\frac{1}{2}-\frac{1}{p}}}{N_1}\|\varphi_1\|_{Y^0}\|\varphi_2\|_{Y^{1/p}}\prod_{i=2}^m\|\varphi_i\|_{Y^{1/2+}},
\end{align*}
for any $p\geq 2$.

Additionally, the above bound holds under the space-time constraint that either $\langle \tau + n^5\rangle \gtrsim N_1^4$, or $\langle \tau_i + n^5_i\rangle \gtrsim N_1^4$ for some $i\in \{1, \cdots, m\},$ without the summation constraint $\mathcal{P}_m(n_1,\cdots,n_m)$.
\end{proposition}
\begin{proof}
The proof of the above proposition is standard, but included for completeness. Before beginning we note that we may assume that the space-time Fourier transforms of $\varphi_i$ are all positive, and that $n_i^* = |n_i|\sim N_i$.

Note that the assumption $\mathcal{P}_m(n_1,\cdots,n_m)$ implies that 
\[
\max(\langle \tau + n^5\rangle, \max_j(\langle \tau_j+n_j^5\rangle))\gtrsim N_1^4,
\]
and hence we may consider both claims in the proposition at once. We now handle the two cases
\begin{itemize}
    \item[I)] $\langle \tau +n^5\rangle\gtrsim n_1^4$,
    \item[II)] $\langle \tau_j +n_j^5\rangle\gtrsim n_1^4$.
\end{itemize}
Assuming that we have $I$, we find
\begin{multline*}
    \bigg\|P_N\mathcal{F}_x^{-1}\bigg(\sum_{\substack{n=n_1+\cdots+n_m\\ \mathcal{P}_m(n_1,\cdots,n_m)}}\prod_{i=1}^m \widehat{\varphi}_{n_i}\bigg)\bigg\|_{X^{0, -3/4}_T}\lesssim T^\theta\bigg\|P_N\mathcal{F}_x^{-1}\bigg(\sum_{\substack{n=n_1+\cdots+n_m\\ \mathcal{P}_m(n_1,\cdots,n_m)}}\prod_{i=1}^m \widehat{\varphi}_{n_i}\bigg)\bigg\|_{X^{0, -1/2+}}\\
    \lesssim\frac{T^\theta}{N_1^{2-}} \bigg\|\chi_N(n)\sum_{\substack{n=n_1+\cdots+n_m\\ \mathcal{P}_m(n_1,\cdots,n_m)}}\prod_{i=1}^m \widehat{\varphi}_{n_i}\bigg\|_{X^{0, 0}}\lesssim \frac{T^\theta}{N_1^{2-}}\|\varphi_1\|_{X^{0,0}}\prod_{i=2}^m\|\varphi_i\|_{L^\infty_tL^\infty_x},
\end{multline*}
so that Sobolev embedding concludes the proof. The second portion of the norm follows from the above estimate by Cuachy-Schwarz in $\tau$:
\[
\big\|\tfrac{1}{\langle \tau + n^5\rangle}\mathcal{F}_{x,t}(\cdot)\big\|_{\ell^2_nL^1_\tau}\lesssim \|\cdot\|_{X^{0, -1/2+}}.
\]

We now assume case $II$, in which case we first take a minor loss, which will make later work easier:
\begin{align}
    &\bigg\|P_N\mathcal{F}_x^{-1}\bigg(\sum_{\substack{n=n_1+\cdots+n_m\\ \mathcal{P}_m(n_1,\cdots,n_m)}}\prod_{i=1}^m \widehat{\varphi}_{n_i}\bigg)\bigg\|_{X^{0, -3/4}_T}\label{Equation: Proposition CS reduction}\\
    &\qquad\qquad\lesssim  T^\theta\bigg\|P_N\mathcal{F}_x^{-1}\bigg(\sum_{\substack{n=n_1+\cdots+n_m\\ \mathcal{P}_m(n_1,\cdots,n_m)}}\prod_{i=1}^m\widehat{\varphi}_{n_i}\bigg)\bigg\|_{X^{0, -(1/2-)}}.\nonumber
\end{align}
and invoke duality for $w\in X^{0, 1/2-}$
\begin{align}
    &\bigg\|P_N\mathcal{F}_x^{-1}\bigg(\sum_{\substack{n=n_1+\cdots+n_m\\ \mathcal{P}_m(n_1,\cdots,n_m)}}\prod_{i=1}^m \widehat{\varphi}_{n_i}\bigg)\bigg\|_{X^{0, -(1/2-)}}\nonumber\\
    &\qquad\qquad= \int_{\tau = \tau_1+\cdots \tau_m}\sum_{\substack{n=n_1+\cdots+n_m\\ \mathcal{P}_m(n_1,\cdots,n_m)}}\widehat{w}(n,\tau)\prod_{i=1}^m\mathcal{F}_{x,t}\left(\varphi_i\right)(n_i, \tau_i)\,d\Gamma\nonumber\\
    &\qquad\qquad\lesssim \frac{1}{N_1}\int_{\tau = \tau_1+\cdots \tau_m}\sum_{\substack{n=n_1+\cdots+n_m\\ \mathcal{P}_m(n_1,\cdots,n_m)}}\widehat{w}(n,\tau)\langle \tau_j+n_j^5\rangle^{1/4}\prod_{i=1}^m\mathcal{F}_{x,t}\left(\varphi_i\right)(n_i, \tau_i)\,d\Gamma\label{Equation: Modulation proposition xsb to show}.
\end{align}
If $j=1$, then we find
\begin{multline*}
\eqref{Equation: Modulation proposition xsb to show}\lesssim \frac{1}{N_1}\|w\|_{L^{\infty-}_tL^{p}_x}\left\|\mathcal{F}^{-1}_{n_1,\tau_1}\left(\langle \tau_1+n_1^5\rangle^{1/4}\mathcal{F}_{x,t}(\varphi_1)\right)\right\|_{L^2_tL^2_x}\\\times\|\varphi_2\|_{L^{2(m-1)+}_tL_x^{\frac{2p}{p-2}}}\prod_{i=3}^m \|\varphi_i\|_{L^{2(m-1)+}_tL^\infty_x},
\end{multline*}
where Sobolev embedding and modulation completes the estimate. When $j > 1$ then a similar estimate holds
\begin{align*}
\eqref{Equation: Modulation proposition xsb to show}\lesssim \frac{1}{N_1}\|w\|_{L^{\infty-}_tL^{p}_x}\|\varphi_1\|_{L^{2+}_tL^2_x}\left\|\mathcal{F}^{-1}_{n_j,\tau_j}\left(\langle \tau_j+n_j^5\rangle^{1/4}\mathcal{F}_{x,t}(\varphi_j)\right)\right\|_{L^2_tL^\frac{2p}{p-2}_x}\prod_{\substack{i=2 \\i\ne j}}^m \|\varphi_i\|_{L^{\infty}_tL^\infty_x},
\end{align*}
where again Sobolev embedding and modulation will complete the claim. The case that $j\ne 2$ is exactly the same as above, except one would use $L^2_tL^\infty_x$ on $\varphi_j$ and $L^\infty_tL^{\frac{2p}{p-2}}$ on $\varphi_2$. 

The second portion of the norm in case $II$ follows from the same Cuachy-Schwarz argument:
\[
\big\|\tfrac{1}{\langle \tau+n^5\rangle}\mathcal{F}_{x,t}(\cdot)\big\|_{\ell^2_nL^1_\tau}\lesssim \|\cdot\|_{X^{0, -(1/2-)}},
\]
and the prior estimates via the reduction done on \eqref{Equation: Proposition CS reduction}. 
\end{proof}

\begin{lemma}[Propagator Terms]\label{Lemma: Propagator Terms} Let $s> 1/2$ and $0\leq \varepsilon < 1$. Then
\[
\|B_1(\tilde{u}) + B_2(\tilde{u})\|_{C^0_tH^{s+\varepsilon}_x}\lesssim \|\tilde{u}\|_{C^0_tH^s_x}^2+\|\tilde{u}\|_{C^0_tH^s_x}^3.
\]
\end{lemma}
\begin{proof}
We first apply Plancherel's and estimate the $B_1$ term:
\begin{multline*}
    \|B_1(\tilde{u})\|_{C^0_tH^{s+\varepsilon}_x}\sim \bigg\|\underset{\substack{n = n_1+n_2\\\star}}{\int\sum}\frac{n(n_1^2+n_2^2)\langle n\rangle^{s+\varepsilon}}{\Phi_2(n_1,n_2)}\widehat{\tilde{u}}(n_1)\widehat{\tilde{u}}(n_2)\bigg\|_{C^0_t\ell^2_n}\\\lesssim\bigg\|\underset{n = n_1+n_2}{\int\sum}\frac{\langle n\rangle^{s+\varepsilon}}{n_1n_2}\widehat{\tilde{u}}(n_1)\widehat{\tilde{u}}(n_2)\bigg\|_{C^0_t\ell^2_n}\lesssim \|\tilde{u}\|_{C^0_tH^s}\|J_x^{-1}\tilde{u}\|_{C^0_tL^{\infty}_x}\lesssim \|\tilde{u}\|_{C^0_tH^s}^2,
\end{multline*}
for $0\leq \varepsilon \leq 1$.

As for $B_2$, note that the symbol is defined at the beginning of Section \ref{Definition: Second DBP terms} and by assumption we have that $\mathcal{P}_3(n_1,n_2,n_3)$ holds. Thus $|\Phi_3(n_1,n_2,n_3)|\gtrsim (n_1^*)^4$ and the symbol is $O(1/n),$ so 
\[
\|B_2(\tilde{u})\|_{C^0_tH^{s+\varepsilon}}\lesssim \|\tilde{u}\|_{C^0_tH^s}^3,
\]
by Sobolev embedding in exactly the same way we found the prior estimate.
\end{proof}

\begin{lemma}[Boundary Terms]\label{Lemma: boundary Terms}
Let $s> 1/2$. Then
\[
\|B_1(\tilde{u})(t)+B_2(\tilde{u})(t)\|_{Y^{s}_T}\lesssim \|\tilde{u}\|_{Y^{s}_T}^2+\|\tilde{u}\|_{Y^s_T}^3.
\]
\end{lemma}
\begin{proof}
We first estimate the $X^{s,1/4}$ portion of $B_1$. If, on the one hand, there exists an $1\leq i\leq 2$  with $\langle \tau_i + n_i^5\rangle\gtrsim \langle \tau + n^5\rangle$, then we have
\begin{align*}
    \|B_1(\tilde{u})\|_{X^{s,1/4}}&\lesssim \bigg\|\underset{\substack{n = n_1+n_2\\\tau = \tau_1+\tau_2}}{\int\sum}\frac{n(n_1^2+n_2^2)\langle n\rangle^s\langle\tau_1+n_1^5\rangle^{1/4}}{\Phi_2(n_1,n_2)}\widehat{\tilde{u}}(n_1)\widehat{\tilde{u}}(n_2)\bigg\|_{\ell^2_tL^2_\tau}\\
    &\lesssim\bigg\|\underset{\substack{n = n_1+n_2\\\tau = \tau_1+\tau_2}}{\int\sum}\frac{\langle n\rangle^s\langle\tau_1+n_1^5\rangle^{1/4}}{n_1n_2}\widehat{\tilde{u}}(n_1)\widehat{\tilde{u}}(n_2)\bigg\|_{\ell^2_tL^2_\tau}.
\end{align*}
At worst we have that $|n_1|\gtrsim |n_2|$, which results in 
\[
\left\|\mathcal{F}^{-1}_{n,\tau}\left(\langle n\rangle^{s-1}\langle\tau+n^5\rangle^{1/4}\widehat{\tilde{u}}\right)\right\|_{L^2_{x,t}}\|J_x^{-1}\tilde{u}\|_{L^\infty_{x,t}}\lesssim\|\tilde{u}\|_{X^{s,1/4}}\|\tilde{u}\|_{Y^s}.
\]
This holds for $B_2$ as well, since the symbol will be
\[
O\bigg(\frac{1}{n_1(n_1+n_2)(n_2+n_3)(n_1+n_3)}\bigg).
\]
As $n_1(n_1+n_2)(n_2+n_3)(n_1+n_3)\gtrsim n_1^*n_3^*$, we may carry out the same argument as before.

We may now assume that
\[
\langle \tau_i + n_i^5\rangle\ll\langle \tau + n^5\rangle\sim |nn_1n_2|(n_1^2+n_2^2+n^2)\sim|\Phi_2(n_1,n_2)|,\]
for $i = 1, 2$. It follows then, that for $|n_1|\gtrsim |n_2|$
\begin{multline*}
    \|B_1(\tilde{u})\|_{X^{s,1/4}}\lesssim \bigg\|\underset{\substack{n = n_1+n_2\\\tau = \tau_1+\tau_2}}{\int\sum}\frac{n(n_1^2+n_2^2)\langle n\rangle^s}{\Phi_2(n_1,n_2)^{3/4}}\widehat{\tilde{u}}(n_1)\widehat{\tilde{u}}(n_2)\bigg\|_{\ell^2_tL^2_\tau}\\
    \lesssim \bigg\|\underset{\substack{n = n_1+n_2\\\tau = \tau_1+\tau_2}}{\int\sum}\frac{\langle n\rangle^s}{n_2^{3/4}}\widehat{\tilde{u}}(n_1)\widehat{\tilde{u}}(n_2)\bigg\|_{\ell^2_tL^2_\tau}
    \lesssim \|\tilde{u}\|_{X^{s,0}_T}\|J_x^{-3/4}\tilde{u}\|_{L^\infty_{x,t}}
    \lesssim \|\tilde{u}\|_{Y^s_T}^2.
\end{multline*}
For $B_2$, the fact that $\mathcal{P}_3(n_1,n_2,n_3)$ holds implies that the symbol multiplied by $\Phi_3^{1/4}$ will be 
\[
O\bigg(\frac{(n_2^2+n_3^2)}{n_1n_1^*}\bigg) = O( 1),
\]
from which it follows that
\[
\|B_2(\tilde{u})\|_{X^{s,1/4}_T}\lesssim \|\tilde{u}\|_{X^{s,0}_T}\|\tilde{u}\|_{Y^{s}_T}^2.
\]

The $\ell^2_nL^1_\tau$ portion of the norms is fairly straight forward. For $B_1$ we assume $|n_1|\gtrsim |n_2|$ so that the convolution structure together with the symbol of $B_1$ being $O(1)$ gives
\begin{align*}
\|\langle n\rangle^sB_1(\tilde{u})\|_{\ell^2_nL^1_\tau}&\lesssim \left\|\|J^s\mathcal{F}_t(\tilde{u})\|_{L^1_\tau}\|\mathcal{F}_t(\tilde{u})\|_{L^1_\tau}\right\|_{L^2_x}\lesssim \|\langle n_1\rangle^s\widehat{\tilde{u}}\|_{\ell^2_nL^1_\tau}\|\mathcal{F}_t(\tilde{u})\|_{L^\infty_xL^1_\tau}\lesssim \|\tilde{u}\|_{Y^s_T}^2.
\end{align*}
Since the second portion for $B_2$ follows in the exact same manner, we conclude the proof.
\end{proof}

\begin{lemma}[Correction Term Penalties]\label{Lemma: Correction Term Penalties} Let $s > 1/2$ and $0\leq \varepsilon < 1$. Then there is a $\theta > 0$ so that
\[
\|R_1(\tilde{u})+R_2(\tilde{u})\|_{Z^{s+\varepsilon}_T}\lesssim T^\theta\left( \|\tilde{u}\|_{Y^s_T}^4 + \|\tilde{u}\|_{Y^s_T}^5\right).
\]
\end{lemma}
\begin{proof}
Recall that these terms are defined on \eqref{Definition: R1} and \eqref{Definition: R4}. We begin with the $X^{s,-3/4}$ portion, where the symbols associated to these are (modulo $\mathcal{K}$ terms)
\begin{align}\label{Equation: Bilinear R Bound}
    |m_{R}^1(n_1+n_2)| &= \bigg|\frac{n(n_1+n_2)(n_1^2+n_2^2)}{\Phi_2(n_1,n_2)}\bigg|\lesssim 1,\\
    |m_{R}^2(n_1+n_2+n_3)| &= \bigg|\frac{n(n_1^2+(n_2+n_3)^2)(n_2+n_3)(n_2^2+n_3^2)(n_1+n_2+n_3)}{\Phi_2(n_1,n_2+n_3)\Phi_3(n_1,n_2,n_3)}\bigg|\label{Equation: Trilinear R Bound}\\
    &\qquad\lesssim \bigg|\frac{n_1+n_2+n_3}{n_1(n_1+n_2)(n_2+n_3)(n_1+n_3)}\bigg|\lesssim 1.\nonumber
\end{align}
Furthermore, we note that the second carries the restriction $|\Phi_3(n,n_1,n_2,n_3)|\gtrsim (n_1^*)^4$. Since we may use modulation considerations for all the terms present, we focus instead on the trilinear term $R_2$. We first turn the trilinear operator into a 5-linear operator by rewriting it as
\begin{align*}
    \widehat{R_2}(\tilde{u})(n) &=\frac{4i}{5}\sum_{\substack{n = n_1+n_2+n_3+n_4+n_5\\0=n_4+n_5\\n\ne n_1,\,n_2,\, n_3}}\frac{m_{R}^2(n_1+n_2)}{n_4n_5}\prod_{j=1}^5\widehat{\tilde{u}}(n_j),
\end{align*}
and note that $\mathcal{P}_5(n_1, n_2,n_3,n_4,-n_4)$ holds true, as 
\[
|\Phi_5(n_1,n_2,n_3,n_4,-n_4)| = |\Phi_3(n_1,n_2,n_3)|\gtrsim (n_1^*)^4.
\]
We then perform a Littlewood-Paley decomposition and invoke Proposition \ref{Proposition: general modulation lemma} under the assumption that $N\sim N_1$ and $p=2$ to obtain
\begin{align*}
    \|R_2(\tilde{u}))\|_{Z^{s+\varepsilon}_T}&\lesssim T^\theta\sum_{N\sim N_1\geq N_2\geq N_3}\frac{N^{s+\varepsilon}}{N_1}N_2^{1/2+}N_3^{1/2+}\prod_{i=1}^3\|P_{N_i}(\eta(t/T)\tilde{u})\|_{Y^0}\lesssim T^\theta\|\tilde{u}\|_{Y^{s}_T}^5,
\end{align*}
for $s > 1/2$ and $0\leq \varepsilon  < 1$. Since the case that $N_1\sim N_2\gg N$ follows similarly, we conclude the proof of the lemma.
\end{proof}

The following lemma concerns the resonant terms from our analysis. The main tool in our analysis will be Corollary \ref{Corollary: Resonant Multiplier Bounds}. These estimate are slightly involved due to the summation restrictions, but the key tool is still the high-frequency approximation from the aforementioned corollary. 

\begin{lemma}[Resonant Terms]\label{Lemma: Resonant Duhamel Bounds} Let $s > 1/2$ and $\varepsilon < \min(2s-1, 1)$. Then there is a $\theta>0$ so that
\[
\|\mathcal{RE}_1(\tilde{u})+\mathcal{RE}_2(\tilde{u})+\mathcal{RE}_3(\tilde{u})+\mathcal{RE}_4(\tilde{u})\|_{Z^{s+\varepsilon}}\lesssim T^\theta\left(\|\tilde{u}\|_{Y^{s}_T}^3 + \|\tilde{u}\|_{Y^s_T}^4\right).
\]
\end{lemma}
\begin{proof}
We first note that these terms are defined on \eqref{Definition: RE1}, \eqref{Definition: RE2}, \eqref{Definition: RE3}, and \eqref{Definition: RE4}. We will proceed by bounding the $X^{s,b}$ portion of the norm for every term under consideration. 

Consider first the $\mathcal{RE}_1$ term. We may split the sum into when $n_3 > 0$ and when $n_3 < 0$. We get
\begin{align}
\widehat{\mathcal{RE}_1}(\tilde{u}) &\sim \widehat{\tilde{u}}(n)\sum_{n_3>0}\left(\frac{n^2n_3}{n^2+(n+n_3)^2+n_3^2}-\frac{n^2n_3}{n^2+(n-n_3)^2+n_3^2}\right)\widehat{\tilde{u}}(n_3)\widehat{\tilde{u}}(-n_3)\nonumber\\
&\qquad+\widehat{\tilde{u}}(n)\sum_{n_3}\frac{nn_3^2}{n^2+(n+n_3)^2+n_3^2}\widehat{\tilde{u}}(n_3)\widehat{\tilde{u}}(-n_3)\nonumber\\
&=-\widehat{\tilde{u}}(n)\sum_{n_3>0}\frac{4n^3n_3^2}{(n^2+(n+n_3)^2+n_3^2)(n^2+(n-n_3)^2+n_3^2)}\widehat{\tilde{u}}(n_3)\widehat{\tilde{u}}(-n_3)\label{Equation: RE1 first sum}\\
&\qquad+\widehat{\tilde{u}}(n)\sum_{n_3}\frac{nn_3^2}{n^2+(n+n_3)^2+n_3^2}\widehat{\tilde{u}}(n_3)\widehat{\tilde{u}}(-n_3).\label{Equation: RE1 second equation}
\end{align}
If $|n|\gg |n_3|$ then we find that both symbols are $O(n_3^2/n)$, whereas if $|n|\lesssim |n_3|$ they are $O(n)=O(n_3)$. Since the latter is easy to handle, we assume that 
$|n|\gg |n_3|$ and write
\begin{equation*}
\frac{n_3^2\langle n\rangle^{s+\varepsilon}}{n} \lesssim \langle n\rangle^{s+\varepsilon}\langle n_3\rangle^{2s}\langle n_3\rangle^{2(1-s)}\langle n\rangle^{-1}\lesssim \langle n\rangle^{s+\varepsilon}\langle n_3\rangle^{2s}\langle n\rangle^{\max(2(1-s),0)-1}.
\end{equation*} 
Hence, for $s > 1/2$ both of these sums satisfy
\[
\|\eqref{Equation: RE1 first sum} + \eqref{Equation: RE1 second equation}\|_{X^{s+\varepsilon,-3/4}}\lesssim \|\eqref{Equation: RE1 first sum}+ \eqref{Equation: RE1 second equation}\|_{X^{s+\varepsilon,0}}\lesssim \|\tilde{u}\|_{X^{s+\varepsilon - \min(2(s-1/2),1 ),0}}\|\tilde{u}\|_{L^\infty_tH^{s}_x}^2.
\]
Note that this also handles the second portion of the norm by Cauchy-Schwarz, as in the proof of Proposition \ref{Proposition: general modulation lemma}. 

We now handle $\mathcal{RE}_2$ and the other multiple resonances from Remark \ref{Remark: Second Order Resonances D_1 and D-2}. We will handle just a prototypical 4-linear term that appears from double resonances in $D_2$. By lemma \ref{Lemma: Nonresonant Weight Bound} and the fact that the trilinear term in $\mathcal{RE}_2$ has symbol $O(n)$, we may assume that the symbol on the $4$-linear term under consideration is $O(n)$. It then suffices by Cauchy-Schwarz to bound
\[
\left\|\langle n\rangle^{s+\varepsilon}n\widehat{\tilde{u}}(n)\widehat{\tilde{u}}(n)\sum_{-n=n_1+n_2}\widehat{\tilde{u}}(n_1)\widehat{\tilde{u}}(n_2)\right\|_{L^2_t \ell^2_n}.
\]
We then bound this by
\begin{align*}
\bigg\|\langle n\rangle^{1/2+\varepsilon/2}\widehat{\tilde{u}}(n)&\langle n\rangle^{1/2+\varepsilon/2}\widehat{\tilde{u}}(n)\langle n\rangle^s\sum_{-n=n_1+n_2}\widehat{\tilde{u}}(n_1)\widehat{\tilde{u}}(n_2)\bigg\|_{L^2_t \ell^2_n}\\
&\lesssim \bigg\|\|\langle n\rangle^{1/2+\varepsilon/2}\widehat{\tilde{u}}(n)\|_{\ell^2_n}^2\big\|\langle n\rangle^s\sum_{-n=n_1+n_2}\widehat{\tilde{u}}(n_1)\widehat{\tilde{u}}(n_2)\big\|_{\ell^2_n}\bigg\|_{L^2_t}\\
&\lesssim \bigg\|\|\langle n\rangle^{1/2+\varepsilon/2}\widehat{\tilde{u}}(n)\|_{\ell^2_n}^2\|\langle n\rangle^s\widehat{\tilde{u}}(n)\|_{\ell^2_n}\|\widehat{\tilde{u}}(n)\|_{\ell^1_n}\bigg\|_{L^2_t}\\
&\lesssim \bigg\|\|\langle n\rangle^{1/2+\varepsilon/2}\widehat{\tilde{u}}(n)\|_{\ell^2_n}^2\|\langle n\rangle^s\widehat{\tilde{u}}(n)\|_{\ell^2_n}\|n^{1/2+}\widehat{\tilde{u}}(n)\|_{\ell^2_n}\bigg\|_{L^2_t}\\
&\lesssim \|\tilde{u}\|_{X^{s,0}_T}\|\tilde{u}\|_{Y^{1/2+\varepsilon/2}_T}^2\|\tilde{u}\|_{Y^{1/2+}_T}.
\end{align*}
This is acceptable so long as $0\leq \varepsilon < 2(s-1/2).$

As for $\mathcal{RE}_3$ and $\mathcal{RE}_4$, we restrict ourselves to simply bounding the harder term, $\mathcal{RE}_4$. We may assume that $i = 3$ and decompose the symbol of the operator into
\begin{align}\label{Equation: B3 decomp}
    m_{B_3} = m_{B_3}\chi_{|n|\gg n_2^*} + m_{B_3}\chi_{|n|\lesssim n_2^*},
\end{align}
and define the symbol (with vanishing associated multilinear operator, by Remark \ref{Remark: Why Resonant Smoother Bound Holds})
\[
\Psi = \left(\frac{n}{25n_1n_4} +\frac{3n_2+6n_4}{25n_1n_4}\right),
\]
and the splitting
\begin{align}
   \Psi&= \Psi\chi_{\mathcal{P}_3(n_1,n_2,n+n_4)}\chi_{|n|\gg n_2^*}\label{Equation: resonance for cancellation}\\
    &+\Psi\chi_{\mathcal{P}_3^c(n_1,n_2,n+n_4)}\chi_{|n|\gg n_2^*}\label{Equation: Resonance Decay term}\\
    &+\Psi\chi_{|n|\lesssim n_2^*}\label{Equation: Resonance big 1}.
\end{align}

We handle first the \eqref{Equation: Resonance big 1} term. Note that, after Cauchy-Schwarz in $\tau$, the quartilinear operator with symbol \eqref{Equation: Resonance big 1} reduces to bounding
\begin{align}\label{Equation: Approximation with large frequencies}
    \bigg\|\langle n\rangle^{s+\varepsilon}\tilde{u}(n)\sum_{\substack{0=n_1+n_2+n_4\\|n|\lesssim n_2^*}}\frac{n_2^*}{n_1n_4}\widehat{\tilde{u}}(n_1)\widehat{\tilde{u}}(n_2)\widehat{\tilde{u}}(n_4)\bigg\|_{L^2_t\ell^2_n}.
\end{align}
By the algebraic relationship we may assume that $n_1\sim n_2\gtrsim n$, so that
\begin{align*}
    \eqref{Equation: Approximation with large frequencies}\lesssim \bigg\|\langle n\rangle^{s}\tilde{u}(n)\sum_{0=n_1+n_2+n_4}\langle n_1n_2\rangle^{\varepsilon/2}\widehat{\tilde{u}}(n_1)\widehat{\tilde{u}}(n_2)\widehat{\tilde{u}}(n_4)\bigg\|_{L^2_t\ell^2_n}\lesssim \|\tilde{u}\|_{X^{s,0}_T}\left\|(J_x^{\frac{\varepsilon}{2}}\tilde{u})^2\tilde{u}\right\|_{L^\infty_tL^1_x}.
\end{align*}
Sobolev Embedding and Cauchy-Schwarz on the last term completes the estimate, which will be good so long as $0\leq \varepsilon \leq 2s-1.$

As for \eqref{Equation: Resonance Decay term}, we find from the assumption that \[\Phi_3(n_1,n_2,n+n_4)\ll\max\{|n_1|,|n_2|, |n+n_4|\}^4\]
and Proposition \ref{Proposition: small symbol decomp} that one of the following must hold:
\begin{itemize}
    \item[I)] $n\in\{n_1, n_2,n+n_4\}$, or
    \item[II)] $n_1\sim n_2\sim n+n_4$.
\end{itemize}
If $I$ holds, then by the assumption of mean-zero $\tilde{u}$ we must have that either $n_1 = n$ or $n_2 = n$. If $II$ holds then we find that $n_1\sim n_2\sim n+n_4\sim n$. In both cases we then have, by the algebraic relationship $0=n_1+n_2+n_4$, that two of $n_1, n_2, n_4$ are $\gtrsim n$, where the implied constant might be different from the one in the relationship $n\gg n_2^*$. It then follows that we may conclude this case as in the bound for \eqref{Equation: Approximation with large frequencies}, which holds for $0\leq \varepsilon < \min(2s-1, 1)$.

The last easy bound to show is the bound for $\mathcal{RE}_4$ with symbol given by $m_{B_3}\chi_{n\lesssim n_2^*}$. In this case we find by \eqref{Equation: mD2 nonresonant bound} that
\begin{equation*}
m_{B_3}\chi_{|n|\lesssim n_2^*} \lesssim \max\{|n_1|, |n_2|, |n_4|\}.
\end{equation*}
By symmetry and Cauchy-Schwarz we find that it suffices to bound
\begin{equation}\label{Equation: RE4 multiple big n2 big}
\bigg\|\langle n\rangle^{s+\varepsilon}\widehat{\tilde{u}}(n)\sum_{\substack{0=n_1+n_2+n_4\\|n_1|\gtrsim n}}n_1\widehat{\tilde{u}}(n_1)\widehat{\tilde{u}}(n_2)\widehat{\tilde{u}}(n_4)\bigg\|_{L^2_\tau \ell^2_n}.
\end{equation}

Since the algebraic relation forces there to be two large frequencies within the sum, we find
\[
\eqref{Equation: RE4 multiple big n2 big}\lesssim \|\tilde{u}\|_{X^{s, 0}_T}\left\|(J_x^{\frac{1+\varepsilon}{2}}\tilde{u})^2\tilde{u}\right\|_{L^\infty_t L^1_x}\lesssim \|\tilde{u}\|_{X^{s, 0}_T}\|\tilde{u}\|_{Y^s_T}^3,
\]
by taking the $\sup$ in $n$ of the interior sum, for $0\leq \varepsilon < 2s-1.$

The last bound we need to show is \eqref{Equation: resonance for cancellation} grouped with the first term of \eqref{Equation: B3 decomp}, which reduces to bounding
\begin{equation}\label{Equation: Grouped terms smoothing reference}
\bigg\|\langle n\rangle^{s+\varepsilon}\widehat{\tilde{u}}(n)\sum_{\substack{0=n_1+n_2+n_4\\|n|\gg |n_1|,|n_2|,|n_4|}}\frac{(n_2^*)^2}{n}\widehat{\tilde{u}}(n_1)\widehat{\tilde{u}}(n_2)\widehat{\tilde{u}}(n_4)\bigg\|_{L^2_\tau \ell^2_n},
\end{equation}
by Corollary \ref{Corollary: Resonant Multiplier Bounds}. Without loss of generality we assume $|n_1|\sim |n_2|$, so that for $0\leq\varepsilon < \min(2(s-1/2), 1)$ we find
\[
\frac{(n_1)^2}{n^{1-\varepsilon}} \lesssim (n_1n_2)^{\min(s, 1)},
\]
which yields
\[
\|\tilde{u}\|_{X^{s,0}_T}\|\tilde{u}\|_{Y^{\min(s,1)}_T}^2\|\tilde{u}\|_{Y^{1/2+}_T},
\]
by the same method as in \eqref{Equation: RE4 multiple big n2 big}.

In every instance the second portion of the norm follows from the first portion via Cauchy-Schwarz, and we similarly find a factor of $T$ by modulation considerations for every bound, completing the proof.
\end{proof}

\subsection{Main Terms in the Duhamel Estimate}\label{Section: Main Term Estimates}
Before we begin with the main Duhamel terms, we would like to note that we could have consolidated the two conditions $n_1^*+n_2^* = 0$ and $n_3\gtrsim n_1$ into $n_3\gtrsim n$ in Proposition \ref{Proposition: Symbol Decomp} However, the condition $n_1^* + n_2^* = 0$ provides better bounds, as is evident by Corollary \ref{Corollary: Case 3 replacement bounds}. Moreover, from a well-posedness perspective, a bound relating $n_3$ to $n$ is not nearly as good as a bound relating $n_3$ with $n_1$. In particular, the latter bound gives more information about $high\times high\times low \cdots$ interactions.

\begin{corollary}\label{Corollary: Case 3 replacement bounds}
Suppose that $n_1^*+n_2^* = 0$, in the language of Proposition \ref{Proposition: Symbol Decomp}. Then 
\[
m_{D_1} = O(1), \qquad\qquad m_{D_2} = O(n_2^*).
\]

\end{corollary}
\begin{proof}
$ $\newline
Note that by the external derivatives in the nonlinearity of \eqref{Equation: Toy Fifth Order} we have
\begin{itemize}
    \item[I)] for $m_{D_1}$ we cannot have $n_1+n_2= 0$, and
    \item[II)] for $m_{D_2}$ we cannot have $n_3+n_4 = 0$. 
\end{itemize}
It follows that the denominators in Lemma \ref{Lemma: Nonresonant Weight Bound} are large. The result then follows from the bounds \eqref{Equation: mD1 nonresonant bound} and \eqref{Equation: mD2 nonresonant bound}.
\end{proof}

We now perform a decomposition motivated by the next section, and disjointly decompose $\mathbb{Z}^k$ into $\mathcal{A}_k\sqcup\mathcal{B}_k$, where 
\[
\mathcal{A}_k := \left\{(n_1,\cdots,n_k)|\,\,\mathcal{P}_k(n_1, \cdots, n_k)\right\}
=\left\{(n_1,\cdots,n_k)|\,\,|\Phi_k( n_1,\cdots,n_k)|\gtrsim (n_1^*)^4\right\}.
\]
Using this, we split $D_i$ into
\begin{align*}
D_1 &= \mathcal{NR}_{2} + D_{\mathcal{A}, 1}\\
D_2 &= \mathcal{NR}_{3} + D_{\mathcal{A}, 2},
\end{align*}
where $D_{\mathcal{A}, i}$ are defined by the symbols $\chi_{\mathcal{A}_4}m_{D_i}$, and $\mathcal{NR}_i$ by $\chi_{\mathcal{B}_4}m_{D_i}$. 

\begin{lemma}[$\mathcal{NR}$ Terms\footnote{This should be read as \textit{nearly-resonant}, as opposed to \textit{non-resonant}.}]\label{Lemma: B portion of D smooth}
Let $a\geq 0$ be such that \eqref{Definition: Property Pa} holds, $s > \frac{1+a}{2}$, and 
\[
0 <\varepsilon < \min\left(2s-1-a, 1\right).
\]
Then
\[
\|\mathcal{NR}_1(\tilde{u}) + \mathcal{NR}_2(\tilde{u}) + \mathcal{NR}_3(\tilde{u})\|_{Z_T^{s+\varepsilon}}\lesssim_\varepsilon  T^\theta\|\tilde{u}\|_{X_T^{s,1/4}}\left(\|\tilde{u}\|_{Y_T^{s}}^2+\|\tilde{u}\|_{Y_T^{s}}^3\right).
\]
\end{lemma}
\begin{proof}
Since $\mathcal{NR}_1$ has three large similar frequencies, we will instead show the bound for the harder terms. Note that we have no resonance by the definition of $D_1$ and $D_2$, and by the definition of $\mathcal{B}_4$, Proposition \ref{Proposition: Symbol Decomp}, and Remark \ref{Remark: Redundant Case} we have the following cases:

\begin{itemize}
    \item[I)] $n_1^* + n_2^* = 0$,
    \item[II)] $(n_3^*)^4n_4^*\gtrsim (n_1^*)^4$,
\end{itemize}

where the symbols under consideration are $O\left(\frac{(n_1^*)^2}{n_3^*n_4^*}\right)$.

\textbf{Case I, $n_1^*+n_2^* = 0$:}
By Corollary \ref{Corollary: Case 3 replacement bounds} it suffices (for $\mathcal{NR}_3$, which is strictly harder than $\mathcal{NR}_2$) to estimate
\begin{align}\label{Equation: Main WP estimate 3 large first term}
    \bigg\|\mathcal{F}_x^{-1}\bigg(\sum_{\substack{n = n_1+n_2+n_3+n_4\\n_1+n_2 = 0}}n_2^*\prod_{i=1}^4\widehat{\tilde{u}}(n_i)\bigg)\bigg\|_{X^{s+\varepsilon,-3/4}}.
\end{align}
We find, by taking the supremum in $n$ of the inner sum and Cauchy-Schwarz:
\begin{multline*}
\eqref{Equation: Main WP estimate 3 large first term}\lesssim \bigg\|\langle n\rangle^s\sum_{n=n_3+n_4}\widehat{\tilde{u}}(n_3)\widehat{\tilde{u}}(n_4)\sum_{n_1}|n_1|^{1+\varepsilon}\widehat{\tilde{u}}(n_1)\widehat{\tilde{u}}(-n_1)\bigg\|_{L^2_t\ell^2_n}\lesssim \|\tilde{u}\|_{C^0_tH_x^{\frac{1+\varepsilon}{2}}}^2\|\tilde{u}^2\|_{X^{s,0}_T}\\
\lesssim \|\tilde{u}\|_{Y_T^{\frac{1+\varepsilon}{2}}}^2\|\tilde{u}\|_{Y^{s}_T}\|\tilde{u}\|_{X^{1/2+,0}_T}
\lesssim T^\theta \|\tilde{u}\|_{X^{s,1/4}_T}\|\tilde{u}\|_{Y^{s}_T}^3,
\end{multline*}
which holds for $0\leq \varepsilon < 2s-1.$

The other portion of the norm is handled in exactly the same way, since we may reduce to the $X^{s,0}$ norm, which is the starting portion of the above display.

\textbf{Case II, $(n_3^*)^4n_4^*\gtrsim (n_1^*)^4$:}
By Lemma \ref{Lemma: Nonresonant Weight Bound} and the assumption of Case II, the symbols are (on this region of $\{n\}\times\mathbb{N}^4$)
\begin{equation}\label{Equation: NR Case 2 Weight bounds}
O\left(\frac{(n_1^*)^2}{n_3^*n_4^*}\right) = O\left(\frac{n_3^*}{\sqrt{n_4^*}}\right) = O\left(\sqrt{\frac{n_2^*n_3^*}{n_4^*}}\right).
\end{equation}

If $\langle \tau + n^5\rangle \gtrsim (n_1^*)^4$, then by the reduction in \eqref{Equation: NR Case 2 Weight bounds} we find
\[
O\left(\sqrt{\frac{n_2^*n_3^*}{n_4^*\langle \tau+n^5\rangle^{\frac{1}{2}-}}}\right) = O\left(\sqrt{\frac{(n_2^*n_3^*)^{0+}}{n_4^*}}\right).
\]
Writing
\begin{equation}\label{Equation: NR case B rewriting smoothing term}
\langle n\rangle^\varepsilon \lesssim (n_2^*n_3^*)^{\varepsilon/2}(n_4^*)^{\varepsilon/4},
\end{equation}
and invoking Parseval's, it then suffices to estimate
\[
\|J_x^{s}\tilde{u}(J_x^{\varepsilon/2+}\tilde{u})^2J_x^{\varepsilon/4-1/2}\tilde{u}\|_{X^{0,-(1/2-)}}.
\]
Sobolev Embedding then easily provides
\begin{multline*}
\|\mathcal{NR}_2(\tilde{u}) + \mathcal{NR}_3(\tilde{u})\|_{X^{s+\varepsilon,-(1/2-)}}\lesssim\|(J_x^s\tilde{u})(J_x^{\varepsilon/2+}\tilde{u})^2(J_x^{\varepsilon/4-1/2}\tilde{u})\|_{L^2_{x,t}}\\
\lesssim \|\tilde{u}\|_{X^{s,0}_T}\|\tilde{u}\|_{Y_T^{1/2+\varepsilon/2+}}^2\|\tilde{u}\|_{Y_T^{\varepsilon/4+}},
\end{multline*}
which holds for $0\leq \varepsilon < 2s-1.$ A factor of $T^{1/4}$ arises from the space in the modulation of the first term. 

It follows that we may assume that 
\begin{equation}\label{Equation: Trade modulation for derivatives}
\langle \tau + n^5\rangle \ll (n_1^*)^4,
\end{equation}
and hence
\begin{equation}\label{Equation: NR terms handeling a terms}
\langle \tau+n^5\rangle^{0+}|n|^a \lesssim (n_2^*n_3^*)^{a/2+}(n_4^*)^{a/4+}.
\end{equation}
We then invoke duality for $w\in X^{-s-\varepsilon, 1/2+}$, the above reduction, \eqref{Equation: NR terms handeling a terms}, \eqref{Equation: NR Case 2 Weight bounds}, and \eqref{Equation: NR case B rewriting smoothing term} to find
\begin{align*}
     \|\mathcal{NR}_2(\tilde{u}) + \mathcal{NR}_3(\tilde{u})\|_{X^{s+\varepsilon,-(1/2-)}}&\lesssim \|\mathcal{NR}_2(\tilde{u}) + \mathcal{NR}_3(\tilde{u})\|_{X^{s+\varepsilon+,-(1/2+)}}\\
    &\lesssim\sup_{\|w\|_{X^{-s-\varepsilon,1/2-}=1}} \|J_x^{-s-\varepsilon-(a+)}w\|_{L^{8+}_{x,t}}\|J_x^{s}\tilde{u}\|_{L^{24/7-}_{x,t}}\\
     &\qquad\times\|J_x^{1/2+a/2+ \varepsilon/2+}\tilde{u}\|_{L^{24/7}_{x,t}}^2\|J_x^{\varepsilon/4 +a/4- 1/2}\tilde{u}\|_{L^{\infty}_{x,t}}\nonumber\\
    &\lesssim \sup_{\|w\|_{X^{-s-\varepsilon,1/2-}=1}}\|w\|_{X^{-s-\varepsilon,1/2-}}\|\tilde{u}\|_{X^{s,1/4-}_T}\|\tilde{u}\|_{Y^{\frac{1+a+\varepsilon}{2}+}_T}^3,
\end{align*}
for $s \geq (1+a+\varepsilon)/2$ and $0 < \varepsilon < 2(s-(1+a)/2)$. A (small) factor of $T$ arrises in this estimate through fudging the implicit space in the $-$ notation. Note that \eqref{Equation: NR terms handeling a terms} was necessary because of the loss of $0+$ temporal derivatives and $a$ spatial derivatives following from the $L^8$ estimate.

Since the other potion of the norm is bounded by the same quantity via Cauchy-Schwarz, we are done with the proof of this lemma.
\end{proof}

The following terms, as well as the boundary terms in $Y^s$, are the only terms that we are unable to (in this section) prove smoothing for, which is why they were separated from the $D_1$ and $D_2$ terms. While we are unable to prove smoothing in this section, we are still able to establish local well-posedness.

\begin{lemma}[Main Terms]\label{Lemma: Main Duhamel Bounds}
If $s > 1/2$, then there is $\theta > 0$ so that 
\[
\|D_{\mathcal{A}, 1}+D_{\mathcal{A}, 2}\|_{Z^s_T}\lesssim T^\theta\|\tilde{u}\|_{Y^s_T}^4.
\]
\end{lemma}
\begin{proof}
Recall that by Lemma \ref{Lemma: Nonresonant Weight Bound} we have the (slightly non-optimal)
\begin{align*}
     m_{D_1}, m_{D_2} &=O\left(\frac{(n_1^*)^2}{\max(n_3^*, |n|)}\right),
\end{align*}
so that it suffices to estimate
\begin{equation}\label{Equation: 5.8 to show}
\bigg\|\mathcal{F}_{n,\tau}^{-1}\bigg(\underset{\substack{n=n_1+n_2+n_3+n_4\\\tau=\tau_1+\tau_2+\tau_3+\tau_4\\\mathcal{P}_4(n_1,n_2,n_3,n_4)}}{\int\sum}\tfrac{(n_1^*)^2\langle n\rangle^{s}}{\max(n_3^*,|n|)}\prod_{j=1}^4\widehat{\tilde{u}}(n_j,\tau_j)\,d\Gamma_4\bigg)(x,t)\bigg\|_{Z^s}
\end{equation}
Notice that after the symbol reduction we may assume that $|n_1|\geq |n_2|\geq|n_3|\geq |n_4|$, so that $|n_1| = n_1^*$. Moreover, by construction, $D_{\mathcal{A}, 1}$ and $D_{\mathcal{A}, 2}$ are non-resonant quartilinear operators supported where $\mathcal{P}_4(n_1,n_2,n_3,n_4)$ holds. It follows that we must necessarily have
\[
\max(\langle \tau +n^5\rangle, \langle \tau_1 + n_1^5\rangle, \cdots, \langle \tau_4+n_4^5\rangle)\gtrsim n_1^4.
\]

We now perform another Littlewood-Paley decomposition, invoke Proposition \ref{Proposition: general modulation lemma} with $p=2$, and assume $N_1\sim N$ to obtain
\begin{align*}
\eqref{Equation: 5.8 to show}&\lesssim T^\theta\sum_{N\sim N_1\geq \cdots \geq N_4}\frac{N_1^2N^s}{N_1\max(N_3, N)}N_2^{1/2+}N_3^{1/2+}N_4^{1/2+}\prod_{i=1}^4\|P_{N_i}(\tilde{u})\|_{Y^0_T}\\
&\lesssim T^\theta\|\tilde{u}\|_{Y^0_T}^4.
\end{align*}
When $N_1\sim N_2\gg N$, we again use Proposition \ref{Proposition: general modulation lemma} for $p\geq 2$ to find
\begin{align*}
\eqref{Equation: 5.8 to show}&\lesssim T^\theta\sum_{\substack{N_1\sim N_2\geq \cdots \geq N_4\\ N_1\sim N_2\gg N}}\frac{N_1^2N^{s+\frac{1}{2}-\frac{1}{p}}}{N_1\max(N_3, N)}N_2^{1/p}N_3^{1/2+}N_4^{1/2+}\prod_{i=1}^4\|P_{N_i}(\tilde{u})\|_{Y^0_T}.
\end{align*}
When $1/2 < s < 1$ we may choose $\frac{1}{p} = s-\frac{1}{2}$, so that we find
\begin{align*}
\frac{N_1^2N^{s+\frac{1}{2}-\frac{1}{p}}}{N_1\max(N_3, N)}N_2^{1/p}N_3^{1/2+}N_4^{1/2+} &= \frac{N_1}{\max(N_3, N)}N_2^{s-1/2}N_3^{1/2+}N_4^{1/2+}\\
&\sim\frac{N_1^{s}}{\max(N_3, N)}N_2^{1/2}N_3^{1/2+}N_4^{1/2+},
\end{align*}
which allows us to conclude the proof when $1/2 < s < 1$. When $s\geq 1$ we simply let $p = 2$, and the above proofs go through. We note that the second case here actually is smoother, but since the $N_1\sim N$ case is not, there is no real gain here.
\end{proof}
As noted, the bound used for $m_{D_1}$ in Lemma \ref{Lemma: Main Duhamel Bounds} is non-optimal. In fact, we have the bound
\[
m_{D_1} = O\left(\frac{(n_1^*)^2}{n\max(n_3^*, |n|)}\right)
\]
by Lemma \ref{Lemma: Nonresonant Weight Bound}. We then have the following corollary.
\begin{corollary}\label{Corollary: D_1 smoother}
If $s > 1/2$ and $0\leq \varepsilon \leq 1$ then there is $\theta > 0$ so that 
\[
\|D_{\mathcal{A},1}(\tilde{u})\|_{Z^{s+\varepsilon}}\lesssim T^\theta \|\tilde{u}\|^4_{Y^s_T}.
\]
\end{corollary}
\begin{proof}
We find by the above comment that
\[
\langle n\rangle^\varepsilon m_{D_1} = O\left(\frac{(n_1^*)^2}{\max(n_3^*, |n|)}\right),
\]
for $0\leq \varepsilon\leq 1$. The proof then follows by the proof of Lemma \ref{Lemma: Main Duhamel Bounds}, as $\langle n\rangle^\varepsilon m_{D_1} = O(m_{D_2}).$
\end{proof}

\subsection{Proof of Theorem \ref{Theorem: Wellposedness for Toy with Pa}}
\begin{proof}
Consider $\Gamma$ as defined at \eqref{Equation: Contractive Operator Definition}. $\Gamma: Y^s\mapsto Y^s$ being a contraction on small enough balls for $\|u_0\|_{H^s}\ll \delta$ follows from standard arguments and Lemmas \ref{Lemma: Propagator Terms}, \ref{Lemma: boundary Terms}, \ref{Lemma: Correction Term Penalties}, \ref{Lemma: Resonant Duhamel Bounds}, \ref{Lemma: B portion of D smooth}, and \ref{Lemma: Main Duhamel Bounds} with $\varepsilon = 0$.

Since the fifth order KdV has the same scaling as the KdV, we may easily scale in order to obtain arbitrary data well-posedness. For more details, see \cite{colliander2003sharp}. Alternatively, one could define a different operator by splitting the frequency space into $n > N$ and $n\leq N$, as in \cite{erdougan2016dispersive}. The benefit of this second approach is that it doesn't require the use of scaling, and will be demonstrated in the proof of Corollary \ref{Theorem: Unconditional Well-posedness}.

As noted, \ref{Definition: Property Pa} holds with $a = 3/32$, so we find Theorem \ref{Theorem: Wellposedness for Toy}. Since we may obtain $u$ from $\tilde{u}$, we have the well-posedness of $\eqref{Equation: Toy Fifth Order}$ and hence well-posedness of \eqref{Equation: General Fifth Order} by Appendix \ref{Section: Appendix A}.
\end{proof}

\section{Smoothing}\label{Section: Smoothing}
Recall that our equation is given by a fixed point of $\Gamma$, and is, up to constants, given by \eqref{Equation: Contractive Operator Definition}.
Lemma \ref{Lemma: Propagator Terms} shows that $B_1(u_0)$ and $B_2(u_0)$ are actually fairly smooth, while we can handle the other $B$ terms in $C^0_tH^{s+\varepsilon}_x$. In fact, every term but $D_2$ is a full derivative smoother. 

The first issue we need to tackle is that we will need to differentiate a portion of $D_2$ again in order to place the term that will eventually have smoothing into $C^0_tH^{s+\varepsilon}_x$.  In particular, Lemma \ref{Lemma: Main Duhamel Bounds} is the best we may obtain without any more work.

In light of Proposition \ref{Proposition: Symbol Decomp}, we may again return to the decomposition of $\mathbb{Z}^k$ into $\mathcal{A}_k\sqcup\mathcal{B}_k$, where $
\mathcal{A}_k = \left\{(n_1,\cdots,n_k)|\,\,\mathcal{P}_k(n_1,\cdots,n_k)\right\}.$
Using this, we again split $D_2$ into $D_2 = \mathcal{NR}_{3} + D_{\mathcal{A}, 2},$
where $D_{\mathcal{A}, 2}$ is defined by the symbol $\chi_{\mathcal{A}_4}m_{D_2}$. We then differentiate-by-parts $D_{\mathcal{A}, 2}$ and write (thanks to \eqref{Equation: Fourier Toy Fifth 3: for w})
\begin{align}
   W_{-t}\widehat{D}_{\mathcal{A},2}(\tilde{u}) =-\partial_tW_{-t}\widehat{B_3}(\tilde{u}) - W_{-t}\widehat{C}(\tilde{u})+l.o.t.,
\end{align}
where
\begin{align*}
\widehat{B_3}(\tilde{u})&=\sum_{\substack{n=n_1+\cdots+n_4\\\mathcal{A}_4}}\frac{m_{D_2}}{\Phi_4(n_1,\cdots, n_4)}\prod_{j=1}^4\widehat{\tilde{u}}(n_j)\\
\widehat{C}(\tilde{u}) &=\sum_{\substack{n=n_1+\cdots+n_5}}(m_\mathcal{A}^1+\cdots+m_\mathcal{A}^4)\prod_{j=1}^5\widehat{\tilde{u}}(n_j)\\
m_{\mathcal{A}}^\ell &= \frac{m_{D_2}(n_\ell+n_{\ell+1})(n_\ell^2+n_{\ell+1}^2)}{\Phi_4(n_1,\cdots, n_\ell+n_{\ell+1},\cdots, n_5)}\chi_{\mathcal{P}_4(n_1,\cdots, n_\ell+n_{\ell+1},\cdots, n_5)},
\end{align*}
for $1\leq \ell\leq 4$. Notice that, by the definition of $\mathcal{A}_4$ and Lemma \ref{Lemma: Nonresonant Weight Bound}, we find that 
\begin{equation}\label{Equation: mA bound}
    m_\mathcal{A}^\ell = O\left(\frac{n_\ell^2+n_{\ell+1}^2}{|n|\max(|n_1|, \cdots, |n_{\ell}+n_{\ell+1}|, \cdots, |n_5|)}\right),
\end{equation}
on the domain of the sum in $\widehat{C}$.

We will need to handle the resonant terms that appear from $C$ by performing the following decomposition.
\begin{lemma}\label{Lemma: Final Resonance Bound}
Let $1\leq\ell\leq 4$, $1\leq j\leq 2$, and $m_{\mathcal{A},R_j}^\ell$ denote resonance in $C$ of the form $n_{\ell+j-1}=n.$ Then
\begin{align*}
m^1_{\mathcal{A}, R_j}, m^2_{\mathcal{A}, R_j} &= O\left(\frac{n_2^*}{|n|}\right),\\
m^3_{\mathcal{A}, R_j}, m^4_{\mathcal{A}, R_j} &= O\left(1+ \frac{n_2^*}{n}\right).
\end{align*}

In particular, for $\ell\in\{3,4\}$ we let
\[
H_{\ell, j} = 125n^{12}\left(\sum_{\substack{r=2\\r\ne \ell+j-1}}^5 n_r\right)\left(\sum_{\substack{r=3\\r\ne \ell+j-1}}^5 n_r\right)(n_{\ell+j-1}+n_{\ell+j}-n).
\]
Then, off the zero-set of $H_{\ell,j}$ and for $|n|\gg (n_2^*)^{5/4}$, we have
\begin{align*}
m^{\ell}_{\mathcal{A},R_j} &- \frac{1}{H_{\ell,j}}= O_{|n|\gg (n_2^*)^{5/4}}\left(\frac{\max(n_1,\cdots, n_{\ell+j-2},n_{\ell+j}, \cdots, n_5)}{n}\right).
\end{align*}

Furthermore, on the zero-set of $H_{\ell,j}$ we must have that the resonant multilinear operator with symbol $m^\ell_{\mathcal{A}, R_j}$ satisfies $|n|\lesssim (n_2^*)^{5/4}.$
\end{lemma}
\begin{proof}
For $\ell= 1,2$ we need a little more than \eqref{Equation: mA bound} gives us. Specializing to $\ell=2$ and $j=1$ we find by \eqref{Equation: mD2 nonresonant bound} that these symbols satisfy
\[
m^2_{\mathcal{A}, R_1} = O\left(\frac{(n^2+n_3^2)(n_4^2+n_5^2)}{\max(|n_1|,|n+n_3|,|n_4|,|n_5|)^4}\right) = O\left(\frac{n_2^*}{n}\right),
\]
with a slightly better bound holding for $m^{1}_{\mathcal{A}, R_j}$. The bound when $\ell=3,4$ follows immediately from \eqref{Equation: mA bound} and considering when $n\gg \max_{j\ne \ell+k-1}(|n_j|)$ and $n\lesssim \max_{j\ne \ell+k-1}(|n_j|)$.

When $\ell = 3, 4$ we now have $4$ terms giving an $n^3$ factor, which leaves us a main term that we must handle. We now assume that $\ell = 3$ and $j = 1$, as the other configurations are exactly the same. We then have that the denominator expands as
\begin{align*}
    \Phi_2(n_1,n_2+n+n_4+n_5)&\Phi_3(n_1,n_2,n+n_4+n_5)\Phi_4(n_1,n_2,n+n_4,n_5)\\
    & = 125n^{12}(n_2+n_4+n_5)(n_4+n_5)n_4 +O((n_2^*)^{12}+n^{11}n_2^*)\\
    &= -125n^{12}n_1(n_4+n_5)n_4+ O((n_2^*)^{12}+n^{11}n_2^*). 
\end{align*}
Similarly, the numerator is $n^{12} + O(n^{11}n_2^*).$
Taking
\[
H_{3,1} = -125n^{12}n_1(n_4+n_5)n_4
\]
immediately gives the bound claimed.

In order to show the bound on the zero-set, we note that the mean-zero assumption implies that
\[
n_2+n_4+n_5, n_4\ne 0,
\]
so that we only have to worry about when $n_4+n_5= 0$. 

We notice that $m^3_{\mathcal{A},R_1}$ comes with the restriction that 
\[
|\Phi_3(n_1,n_2, n+n_4+n_5)|\gtrsim \max(n_1,n_2,n+n_4+n_5)^4,
\]
and hence if $n_4+n_5=0$, then 
\[
|\Phi_3(n_1, n_2, n)|\gtrsim n^4.
\]
From this it easily follows that $(n_1^*)^5\gtrsim n^4$, and the algebraic relationship among the remaining frequencies forces $(n_2^*)^{5/4}\gtrsim |n|$.
\end{proof}

When summing over a symmetric subset we find a property analogous to the one discussed in Remark \ref{Remark: Why Resonant Smoother Bound Holds}. In particular, we make the following remark.

\begin{remark}\label{Remark: Decay of 0 resonance}
Much like Remark \ref{Remark: Why Resonant Smoother Bound Holds}, we have a vanishing property in the above resonance. In particular,
\[
f := \widehat{\tilde{u}}(n)\sum_{\substack{0=n_1+n_2+n_4+n_5\\|n|\gg (n_2^*)^{5/4}}} m^{3}_{\mathcal{A}, R_1}\frac{1}{H_{3,1}}\prod_{\substack{r=1\\r\ne 3}}^5\widehat{\tilde{u}}(n_r) = 0,
\]
where the same result holds for the other configurations of $3\leq \ell \leq 4$ and $1\leq j\leq 2$. Indeed, recall that, up to a constant, $H_{3,1} = cn^{12}n_1(n_4+n_5)n_4.$
Letting $j = n_4+n_5$ we may rewrite $f$ as
\begin{align*}
    \sum_{j\ne 0}\frac{c}{j}\widehat{\tilde{u}}(n)\sum_{\substack{-j=n_1+n_2\\|n|\gg \max(|n_1|,|n_2|)^{5/4}}}\frac{1}{n_1}\widehat{\tilde{u}}(n_1)\widehat{\tilde{u}}(n_2)\sum_{\substack{j=n_4+n_5\\|n|\gg \max(|n_4|,|n_5|)^{5/4}}} \frac{1}{n_4}\widehat{\tilde{u}}(n_4)\widehat{\tilde{u}}(n_5).
\end{align*}
It then follows that the $j$ and $-j$ terms in the above sum cancel, so that $f = 0$.
\end{remark}
We now define the operator $\mathcal{RE}_5(\phi)$ to be
\begin{equation}\label{Definition: RE5}
\widehat{\mathcal{RE}_5}(\phi) = \widehat{\phi}(n)\sum_{\substack{0=n_1+\cdots+n_{\ell-1}+n_{\ell+1}+\cdots+n_5\\ 1\leq \ell\leq 4}}m^\ell_{\mathcal{A},R_j}\prod_{\substack{k=1\\k\ne\ell+j-1}}^4\widehat{\phi}(n_k),
\end{equation}
which, by symmetry, will handle all of the the new single resonances coming from $C$. After removal of $\mathcal{RE}_5$ and placing the double resonances of $C$ into $\mathcal{RE}_2$, we denote $C_{*}$ to be the non-resonant part of $C$.

\begin{remark}\label{Remark: Higher order resonances and smoothing}
Bear in mind, much like in the definition of $\mathcal{RE}_2$, we have multiple resonances popping up in the above, and we overload the definition of $\mathcal{RE}_2$ to contain them as well. The higher order resonances arising from $m_{A_i}$ and $m_{B_i}$ are already smoother, as they will have 3 high frequencies ($\gtrsim |n|$). Similarly, the bound \eqref{Equation: mA bound} with the added information that the denominator must be $\gtrsim n^2$ gives $m^\ell_{\mathcal{A},R_j} = O(1)$, and the fact that there are three frequencies $\gtrsim |n|$ gives $2s$ smoothing (beginning at the local wellposedness level). With this in mind, we see that it really is sufficient to prove smoothing for the cubic double resonance term of $\mathcal{RE}_2$.
\end{remark}

Because of the property discussed in Remark \ref{Remark: Decay of 0 resonance}, we find that we don't have to perform another removal. Our final equation then reads (up to constants)
\begin{align}\label{Equation: Final Smoothing Equation}
    \partial_t&(W_{-t}\tilde{u}) = \partial_tW_{-t}B_1(\tilde{u})+\partial_tW_{-t}B_2(\tilde{u})+\partial_t W_{-t}B_3(\tilde{u})\\
    &+W_{-t}(\mathcal{NR}_1(\tilde{u})+\mathcal{NR}_2(\tilde{u})+\mathcal{NR}_3(\tilde{u})+C_*(\tilde{u}))\nonumber\\
    &+W_{-t}(R_1(\tilde{u})+R_2(\tilde{u})+R_3(\tilde{u}))\nonumber\\
    &+W_{-t}(\mathcal{RE}_1(\tilde{u})+\mathcal{RE}_2(\tilde{u})+\mathcal{RE}_3(\tilde{u})+\mathcal{RE}_4(\tilde{u}) + \mathcal{RE}_5(\tilde{u}))\nonumber,
\end{align}
where $R_3$ is the remainder term resulting from the substitution of 
\begin{equation}\label{Definition: R5}
\mathcal{F}^{-1}_x\left(in\mathcal{K}(\tilde{u})\widehat{\tilde{u}}(n) \right)
\end{equation}
into the various components of $B_3$. In the interest if readability, we note that the $B$ terms can be understood as \textit{Boundary} terms, the $\mathcal{NR}$ terms as \textit{Nearly Resonant}, the $R$ terms as \textit{Remainder} terms, $\mathcal{RE}$ as the \textit{Resonant} terms, and $C_*$ as the non-resonant portion of $C$. 

\subsection{Auxilliary Smoothing Lemmas}\label{Section: Easy Smoothing Lemmas}
Because of the extra work we did in the prior section, we only have to prove smoothing for a small subset of the terms present in Equation \eqref{Equation: Final Smoothing Equation}. In particular, we only\footnote{We have the additional higher order resonances from $C$ that were added to $\mathcal{RE}_2$, but, as mentioned, they follow in the exact same way as the estimate for $\mathcal{RE}_2$. See Lemma \ref{Lemma: Resonant Duhamel Bounds} and Remark \ref{Remark: Higher order resonances and smoothing}.} have to deal with $B_3$, $\mathcal{RE}_5$, $R_3$, and $C_*$.
\begin{lemma}[Propagator Terms]\label{Lemma: Propagator Terms Smoothing}
Let $s > 1/2$ and $0 < \varepsilon < 1$. Then
\[
\|B_3(\tilde{u})\|_{C^0_tH^{s+\varepsilon}_x}\lesssim \|\tilde{u}\|_{C^0_tH^s_x}^4.
\]
\end{lemma}
\begin{proof}
The proof is exactly the same as in Lemma \ref{Lemma: Propagator Terms}, as the symbol of $B_3$ is $O(1/n_1^*)$.
\end{proof}
\begin{lemma}[Correction Term Penalties]\label{Lemma: Correction Term smoothing} Let $s > 1/2$. Then for $0 < \varepsilon < 1$,
\[
\|R_3(\tilde{u})\|_{Z^{s+\varepsilon}_T}\lesssim \|\tilde{u}\|_{Y^s_T}^6.
\]
\end{lemma}
\begin{proof}
Recall the definition of $R_3$ given on \eqref{Definition: R5}. In particular, it is a quartilinear operator with symbol given by 
\begin{equation}\label{Equation: 4-linear R Bound}
\bigg|\frac{m_{D_2}\chi_{\mathcal{P}(n_1,\cdots,n_4)}\sum_{j=1}^4n_j\mathcal{K}(\tilde{u})}{\Phi_4(n_1,\cdots, n_4)}\bigg|\lesssim\frac{n_1^*|\mathcal{K}(\tilde{u})m_{D_2}|\chi_{\mathcal{P}(n_1,\cdots,n_4)}}{|\Phi_4(n_1,\cdots, n_4)|}\lesssim\frac{|\mathcal{K}(\tilde{u})|}{n_1^*n_3^*},
\end{equation}
where the final inequality followed from Lemma~\ref{Lemma: Nonresonant Weight Bound} and the restriction to $\mathcal{P}(n_1,\cdots,n_4)$. In other words, we have found (after Cauchy-Schwarz in $\tau$) that it is sufficient to estimate
\begin{equation*}
\|R_3(\tilde{u})\|_{Z^{s+\varepsilon}}\lesssim \|R_3\|_{X^{s+\varepsilon,0}}\lesssim
\bigg\|\sum_{\substack{n=n_1+\cdots+n_4\\\mathcal{A}_4}}\frac{\langle n\rangle^{s+\varepsilon}|\mathcal{K}(\tilde{u})|}{n_1^*n_3^*}\prod_{j=1}^4\widehat{\tilde{u}}(n_j)\big)\bigg\|_{L^2_t\ell^2_x}.
\end{equation*}
Thus, by turning it into a 6-linear operator as in Lemma \ref{Lemma: Correction Term Penalties} and cancelling the $\langle n\rangle ^\varepsilon$ factor with the $n_1^*$ factor, we find that the result follows immediately from
\[
\bigg\|\sum_{\substack{n=n_1+\cdots+n_4\\\mathcal{A}_4}}\frac{\langle n\rangle^{s+\varepsilon}|\mathcal{K}(\tilde{u})|}{n_1^*n_3^*}\prod_{j=1}^4\widehat{\tilde{u}}(n_j)\big)\bigg\|_{L^2_t\ell^2_x}\lesssim \|\tilde{u}^6\|_{X^{s,0}}\lesssim \|\tilde{u}\|_{L^2_tH^s_x}\|\tilde{u}\|_{Y^{1/2+}_T}^6,
\]
where we have invoked Sobolev Embedding and the $Y^{s}\hookrightarrow C^0_tH^s_x$ embedding. This similarly handles the other portion of the norm as demonstrated in the prior lemmas.

It's interesting to note that we did not have to use modulation considerations, as in the proof of Lemma \ref{Lemma: Correction Term Penalties}.
\end{proof}

\begin{lemma}[Resonant Terms]\label{Lemma: Resonant terms without M are smooth} Let $s > 1/2$ and $0 <\varepsilon < \min(2(s-1/2), 1)$. Then there is a $\theta > 0$ so that
\begin{align*}
&\left\|\mathcal{RE}_3(\tilde{u})\right\|_{Z^{s+\varepsilon}_T}\lesssim T^\theta \|\tilde{u}\|_{Y^s_T}^5.
\end{align*}
\end{lemma}
\begin{proof}
Recall that this term is defined on \eqref{Definition: RE5} as
\[
\mathcal{RE}_3(\tilde{u}) = \widehat{\tilde{u}}(n)\sum_{\substack{0=n_1+\cdots+n_{\ell-1}+n_{\ell+1}+\cdots+n_5\\ 1\leq \ell\leq 4}}m^\ell_{\mathcal{A},R_j}\prod_{\substack{k=1\\k\ne\ell+j-1}}^4\widehat{\tilde{u}}(n_k).
\]
Additionally, this lemma almost follows from Lemma \ref{Lemma: Resonant Duhamel Bounds} and Remarks \ref{Remark: Decay of 0 resonance} and \ref{Remark: Higher order resonances and smoothing}.

We ignore the emergence of double resonances, as they are bounded in exactly the same way as $\mathcal{RE}_2$. Additionally, we see by Lemma \ref{Lemma: Final Resonance Bound} that the cases $\ell = 1, 2$ follow in exactly the same way as the proof for $\mathcal{RE}_3$. We now assume that resonance is occurring on $\ell = 3$ with $j = 1$. 

We now, in order to introduce the summation restrictions, split $m^3_{\mathcal{A}, R_1}$ into
\begin{equation}\label{Equation: RE5 symbol split}
    m^3_{\mathcal{A}, R_1} = m^3_{\mathcal{A}, R_1}\chi_{|n|\gg (n_2^*)^{5/4}} + m^3_{\mathcal{A}, R_1}\chi_{|n|\lesssim (n_2^*)^{5/4}},
\end{equation}
and
\begin{align}
    \frac{1}{H_{3,1}} &= \frac{1}{H_{3,1}}\chi_{|n|\gg (n_2^*)^{5/4}}\chi_{\mathcal{P}_3(n_1,n_2,n+n_4+n_5)}\chi_{\mathcal{P}_4(n_1,n_2,n+n_4,n_5)}\label{Equation: H symbol split 1}\\
    &+\frac{1}{H_{3,1}}\chi_{n\gg (n_2^*)^{5/4}}\chi_{\mathcal{P}_3(n_1,n_2,n+n_4+n_5)}\chi_{\mathcal{P}_4^c(n_1,n_2,n+n_4,n_5)}\label{Equation: H symbol split 2}\\
    &+\frac{1}{H_{3,1}}\chi_{n\gg (n_2^*)^{5/4}}\chi_{\mathcal{P}_3^c(n_1,n_2,n+n_4+n_5)}\chi_{\mathcal{P}_4(n_1,n_2,n+n_4,n_5)}\label{Equation: H symbol split 3}\\
    &+\frac{1}{H_{3,1}}\chi_{n\gg (n_2^*)^{5/4}}\chi_{\mathcal{P}_3^c(n_1,n_2,n+n_4+n_5)}\chi_{\mathcal{P}_4^c(n_1,n_2,n+n_4,n_5)}\label{Equation: H symbol split 4}\\
    &+\frac{1}{H_{3,1}}\chi_{n\lesssim (n_2^*)^{5/4}}\label{Equation: H symbol split 5},
\end{align}
where we are assuming that we are off the zero set of $H_{3,1}$. We note that by Lemma \ref{Lemma: Final Resonance Bound} the zero set of $H_{3,1}$ is contained in the set defined by $n\lesssim (n_2^*)^{5/4}$, and hence to handle $m_{\mathcal{A},R_1}^3$ when $H_{3,1}=0$ it is sufficient to handle the second term of \eqref{Equation: RE5 symbol split}.

We handle \eqref{Equation: H symbol split 2}, \eqref{Equation: H symbol split 3}, and \eqref{Equation: H symbol split 4} with \eqref{Equation: H symbol split 5}. In order to do this, we handle only the reduction for the situation that $\Phi_4(n, n_1, n_2, n+n_{4}, n_5)\ll n^4$, as the other is similar after the observation that $n_4+n_5=0$ is not in the domain of $H_{3,1}$. In this situation we must have, by Proposition \ref{Proposition: Symbol Decomp} and Remark \ref{Remark: Redundant Case}, that one of the following must be true:
\begin{itemize}
    \item[I)] $n\in\{n_1,n_2, n+n_4, n_5\}$, or
    \item[II)] $n_1\sim n_2\sim n+n_3$, or
    \item[III)] $n_1^*+n_2^* = 0$, or
    \item[IV)] $(n_3^*)^4n_4^*\gtrsim (n_1^*)^4$.
\end{itemize}
First assume that $I$ occurs, in which case we find that $n+n_4\ne n$ by the mean-zero assumption on $\tilde{u}$. This leaves us with $n_1, n_2$ or $n_5$ being equal to $n$. This is ruled out by the assumption that $n\gg (n_2^*)^{5/4}$. This similarly rules out $II$ and $III$. All that remains is $IV$, which may hold with a constant differing from the one in the characteristic function. Regardless, though, we may handle this with \eqref{Equation: H symbol split 5}\footnote{The only difference in the situation that $\Phi_3\ll n^4$ is the observation that $0\ne n_2+n_4$ as it is \textit{not} in the domain of $H_{3,1}$. In particular, the final claim of Lemma \ref{Lemma: Final Resonance Bound} allows us to handle this with \eqref{Equation: H symbol split 5}.}. 

We now handle \eqref{Equation: H symbol split 5} with the second portion of \eqref{Equation: RE5 symbol split}, as $m^3_{\mathcal{A}, R_1}=O(1)$. By the algebraic relationship among the remaining frequencies and \eqref{Equation: mA bound} we find
\[
\eqref{Equation: H symbol split 5} + m^3_{\mathcal{A}, R_1}\chi_{n\lesssim (n_2^*)^{5/4}} = O_{n\lesssim (n_2^*)^{5/4}}\left(1 + \frac{n^2+(n_2^*)^2}{n\max(n_1, n_2, n+n_{4}, n_5)}\right) = O_{n\lesssim (n_2^*)^{5/4}}\left( 1+\frac{n_2^*}{n}\right).
\]

The $O\left(\frac{n_2^*}{n}\right)$ provides smoothing after killing the $\langle n\rangle^\varepsilon$ factor using the denominator, and then estimating in the same way we estimated \eqref{Equation: RE4 multiple big n2 big}. This leaves us with estimating the $O(1)$ term.

By using the fact that $n\lesssim (n_2^*)^{5/4}$ and the algebraic relationship among the remaining frequencies, we find $
\langle n\rangle^\varepsilon \lesssim (n_2^*n_3^*)^{5/8\varepsilon}.$
We then have by Cauchy-Schwarz in $\tau$ that
\begin{align*}
    \bigg\|\langle n\rangle^{s+\varepsilon}\widehat{\tilde{u}}(n)&\sum_{\substack{0=n_1+n_2+n_4+n_5\\n\lesssim (n_2^*)^{5/4}}}\widehat{\tilde{u}}(n_1)\widehat{\tilde{u}}(n_2)\widehat{\tilde{u}}(n_4)\widehat{\tilde{u}}(n_5)\bigg\|_{L^2_t\ell^2_n}\\
    &\lesssim \bigg\|\langle n\rangle^{s}\widehat{\tilde{u}}(n)\sum_{\substack{0=n_1+n_2+n_4+n_5\\n\lesssim (n_2^*)^{5/4}}}\langle n_1\rangle^{5/8\varepsilon}\widehat{\tilde{u}}(n_1)\langle n_2\rangle^{5/8\varepsilon}\widehat{\tilde{u}}(n_2)\widehat{\tilde{u}}(n_4)\widehat{\tilde{u}}(n_5)\bigg\|_{L^2_t\ell^2_n}\\
    &\lesssim \|\tilde{u}\|_{X^{s,0}_T}\|(J_x^{5/8\varepsilon}\tilde{u})^2\tilde{u}\|_{L^\infty_t L^1_x}\\
    &\lesssim \|\tilde{u}\|_{X^{s,0}_T}\|\tilde{u}\|_{Y^{s}_T},
\end{align*}
for $0\leq \varepsilon < 8s/5.$

All that remains is the first portion of \eqref{Equation: RE5 symbol split} with \eqref{Equation: H symbol split 1}. By Lemma \ref{Lemma: Final Resonance Bound} we see that $H_{3, 1}$ is non-zero and the symbol after adding these two terms together is
\begin{multline*}
m^{3}_{\mathcal{A},R_1}\chi_{n\gg(n_2^*)^{5/4}} - \frac{1}{H_{3,1}}\chi_{n\gg (n_2^*)^{5/4}}\chi_{\Phi_3(n, n_1, n_{2}, n+n_{4}+ n_5)\gtrsim n^4}\chi_{\Phi_4(n, n_1, n_2, n+n_{4}, n_5)\gtrsim n^4}\\= O\left(\frac{\max(|n_1|,|n_2|,|n_{4}|,|n_5|)}{n}\right) = O\left(\frac{n_2^*}{n}\right).
\end{multline*}
This then follows in a similar (and easier) manner than \eqref{Equation: Grouped terms smoothing reference}.

As the second portion of the norm follows in the same manner in all of the above cases, we're done after noting that we always have space in the modulation for a factor of $T$.
\end{proof}

\subsection{Smoothing for the Main terms}
We now proceed to prove smoothing for the main term. The main tools for this will be \eqref{Equation: mA bound} and Lemma \ref{Lemma: Final Resonance Bound}.

\begin{lemma}\label{Lemma: Main term smoothing estimate}
Let $a\geq 0$ be such that \eqref{Definition: Property Pa} holds, $s > \frac{1+a}{2}$, and $0 < \varepsilon < \min(2s-1-a, 1)$. Then there is a $\theta> 0$ so that
\[
\|C_*(\tilde{u})\|_{Z^{s+\varepsilon}}\lesssim T^\theta\|\tilde{u}\|_{Y^s_T}^5.
\]
\end{lemma}
\begin{proof}
Recall that this term is given by
\[
\widehat{C}(\tilde{u}) =\sum_{\substack{n=n_1+\cdots+n_5}}(m_\mathcal{A}^1+\cdots+m_\mathcal{A}^4)\prod_{j=1}^5\widehat{\tilde{u}}(n_j),
\]
where, by \eqref{Equation: mA bound}, we have
\begin{align*}
     m_\mathcal{A}^\ell = O\left(\frac{n_\ell^2+n_{\ell+1}^2}{n\max(n_1, \cdots, n_{\ell}+n_{\ell+1}, \cdots, n_5)}\right).
\end{align*}
We then either have a $high\times high\mapsto low$ interaction leaving us with
\begin{align}\label{Equation: Smoothing HHL}
     m_\mathcal{A}^\ell = O\left(\frac{n_1^* n_{2}^*}{n\max(n_1, \cdots, n_{\ell}+n_{\ell+1}, \cdots, n_5)}\right),
\end{align}
or we have $high\times low\mapsto high$ or $high\times high\mapsto high$, giving
\begin{align}\label{Equation: Smoothing HLH}
     m_\mathcal{A}^\ell = O\left(\frac{n_1^*}{n}\right).
\end{align}
By Proposition \ref{Proposition: Symbol Decomp} and the fact that we have already handled resonances, one of the following must be true
\begin{itemize}
    \item[I)] $n^5-n_1^5-n_3^5- n_4^5 - n_5^5\gtrsim (n_1^*)^4$,
    \item[II)] $n_1^*+n_2^*=0$,
    \item[III)] $(n_3^*)^4n_4^*\gtrsim (n_1^*)^4$.
\end{itemize}

\textbf{Case $I$, $n^5-n_1^5-n_3^5- n_4^5 - n_5^5\gtrsim (n_1^*)^4$:} We must have that
\[
\max(\langle \tau +n^5\rangle, \langle \tau_1 + n_1^5\rangle, \cdots, \langle \tau_5+n_5^5\rangle)\gtrsim (n_1^*)^4.
\]
Moreover, we assume for the proof of Case I that $|n_1|\geq|n_2|\geq\cdots\geq |n_5|$, and $\ell = 1$.

We first perform a Littlewood-Paley decomposition as in Proposition \ref{Proposition: Symbol Decomp}. If we have $N_1\sim N$, then we have by Proposition \ref{Proposition: general modulation lemma} that
\begin{align*}
\|C_*(\tilde{u})\|_{Z^{s+\varepsilon}}&\lesssim T^\theta\sum_{N\sim N_1\geq N_2\cdots N_5}N^{\varepsilon - 1}\|P_{N_1}(\tilde{u})\|_{Y^s}\prod_{i=2}\|P_{N_i}(\tilde{u})\|_{Y^{1/2+}}\\
&\lesssim T^\theta\|\tilde{u}\|_{Y^s_T}^5,
\end{align*}
for $0 \leq \varepsilon < 1.$ If we have $N_1\sim N_2\gg N$ then we find
\begin{align*}
\|C_*(\tilde{u})\|_{Z^{s+\varepsilon}}&\lesssim T^\theta\sum_{\substack{N_1\sim N_2\geq \cdots \geq N_5\\ N_1\sim N_2\gg N}}\frac{N_1^2N^{s+\varepsilon+\frac{1}{2}-\frac{1}{p}}}{N_1N^2}N_2^{1/p}N_3^{1/2+}N_4^{1/2+}\prod_{i=1}^5\|P_{N_i}(\tilde{u})\|_{Y^0}.
\end{align*}
When $1/2 < s < 1$ we may again choose $\frac{1}{p} = s-\frac{1}{2}$ and $0\leq \varepsilon\leq 1$, so that we find
\begin{align*}
\frac{N_1^2N^{s+\varepsilon+\frac{1}{2}-\frac{1}{p}}}{N_1N^2}N_2^{1/p}N_3^{1/2+}N_4^{1/2+}N_5^{1/2+} &= N_1N_2^{s-1/2}N_3^{1/2+}N_4^{1/2+}N_5^{1/2+}\\
&\sim N_1^sN_2^{1/2}N_3^{1/2+}N_4^{1/2+}.
\end{align*}
Cauchy-Schwarz then finishes both of these cases. When $s\geq 1$ the proof follows by the above estimates and taking $p = 2.$

\textbf{Case II, $n_1^*+n_2^* = 0$:}
In this case we will assume that $n_2^*\gg n_3^*$, as otherwise we can handle it in Case III. Since we cannot have $n_{\ell} + n_{\ell+1} = 0$, we see that (just as in the $High\times High\mapsto Low$ situation):
\[
\max(n_1, \cdots, n_{\ell}+n_{\ell+1}, \cdots, n_5)\gtrsim n_1^*,
\]
and hence $m_{\mathcal{A}}^\ell = O(n_1^*/n) = O(n_2^*/n)$, as we're in case II. 

For $0 < \varepsilon\leq 1$ and $s > 1/2$ it follows by taking the supremum in $t$ and invoking the algebra property of $H^s$ that
\begin{align*}
    \|C_*(\tilde{u})\|_{X^{s+\varepsilon,0}}&\lesssim \left\|\sum_{n_1+n_2=0}|n_1|\widehat{\tilde{u}}(n_1)\widehat{\tilde{u}}(n_2)\right\|_{C^0_t}\left\|\langle n\rangle^{s+\varepsilon-1}\sum_{n_1+n_2+n_3=n}\widehat{\tilde{u}}(n_1)\widehat{\tilde{u}}(n_2)\widehat{\tilde{u}}(n_3)\right\|_{L^2_t\ell^2_n}\\
    &\lesssim \|\tilde{u}\|_{C^0_tH^{1/2}_x}^2\|\tilde{u}\|_{Y^s_T}^{3}\lesssim \|\tilde{u}\|_{Y^{s}_T}^5.
\end{align*}

\textbf{Case III, $(n_3^*)^4n_4\gtrsim (n_1^*)^4$:}
Note first that $\max(n_1, \cdots, n_{\ell}+n_{\ell+1}, \cdots, n_5)\gtrsim n_3^*$, and hence if $n\gtrsim n_4^*$ then $m^\ell_\mathcal{A}$ satisfies \eqref{Equation: NR Case 2 Weight bounds}, and the proof of Lemma \ref{Lemma: B portion of D smooth} holds with an extra application of $Y^{1/2+}\hookrightarrow L^\infty_{x,t}$. We now assume that $n_4^*\gg n$, in which case
\[
\langle n\rangle^{s+\varepsilon} m_\mathcal{A}^\ell = O\left(\frac{|n|^{s+\varepsilon}(n_1^*)^s(n_3^*)^2\sqrt{n_4^*}}{nn_3^*}\right) = O\left(\frac{(n_1^*)^{s-(1/2+)}(n_2^*n_3^*n_4^*)^{\frac{1}{2}+\frac{\varepsilon}{3}+}}{|n|^{1/2+}}\right).
\]
It then follows by duality with $w\in X^{-s-\varepsilon,1/2-}$ that we have
\begin{align*}
    \|C_*(\tilde{u})\|_{X^{s+\varepsilon,-(1/2-)}} &\lesssim \|J_x^{-s-\varepsilon-(1/2+)}w\|_{L^{\infty}_xL^{\infty-}_t}\|J_x^{s-(1/2+)}\tilde{u}\|_{L^{\infty}_{x,t}}\\
     &\qquad\times\|J_x^{1/2+\varepsilon/3+}\tilde{u}\|_{L^{24/7-}_{x,t}}^3\|\tilde{u}\|_{L^{\infty}_{x,t}}\nonumber\\
    &\lesssim \|w\|_{X^{-s-\varepsilon,1/2-}}\|\tilde{u}\|_{Y^s_T}\|\tilde{u}\|_{X^{s,1/4-}_T}\|\tilde{u}\|_{Y^{\frac{1}{2}+}_T}^3,
\end{align*}
for 
\[
0\leq\varepsilon < 3\left(s-\frac{1}{2}\right).
\]
We find a small factor of $T$ from the space in the $L^{{24/7-}}_{x,t}$ estimates. As the other portion of the norm follows in exactly the same fashion after Cauchy-Schwarz in $\tau$ we conclude the proof.

Note that the case when $n\gtrsim n_4^*$ requires the hypothesis on the $L^8$ estimate, as in Lemma \ref{Lemma: B portion of D smooth}. Additionally, the other portion of the $Z$ norm is handled by the above estimates as in the previous lemmas.
\end{proof}
\begin{proof}[Proof of Smoothing]
$ $\newline
Upon integration we find the equivalent Duhamel formula as
\begin{align}\label{Equation: Duhamel formula for smoothing}
    \tilde{u} -&\{B_1(\tilde{u})+B_2(\tilde{u})+B_3(\tilde{u})\} + W_t\{B_1(u_0)+B_2(u_0)+B_3(u_0)-u_0\}\\
    &=\int_0^tW_{t-s}\left(C_*(\tilde{u}) +smooth\right)\,ds.\nonumber
\end{align}
For $s> \frac{1+a}{2},$ $0\leq \varepsilon < \min(2s-1-a, 1)$, and $t < T = T(\|u_0\|_{H^s})$ we apply the lemmas in this section to obtain
\begin{align*}
    &\left\|\mbox{LHS of }\eqref{Equation: Duhamel formula for smoothing}\right\|_{C^0_tH^{s+\varepsilon}_x}\lesssim \|C_*(\tilde{u}) +smooth\|_{Z^{s+\varepsilon}_{T}}\lesssim C(\|u\|_{Y^s})\lesssim C(\|u_0\|_{H^s_x}).
\end{align*}
An application of Lemma \ref{Lemma: Propagator Terms Smoothing} and the local well-posedness bound handles all of the terms remaining on the left hand side of \eqref{Equation: Duhamel formula for smoothing}, giving the desired result
\[
\|\tilde{u}-W_tu_0\|_{C^0_tH^{s+\varepsilon}_x}\lesssim C(\|u_0\|_{H^s_x}),
\]
by the triangle inequality.
\end{proof}
\section{Unconditional Well-posedness}\label{Section: UCWP}
For this section we will not need to rely on showing cancellation in the resonant terms with symbols that are $O(1)$, so our equation under consideration (again, modulo constants) is
\begin{align}\label{Equation: Unconditional Wellposedness Toy Equation}
    \partial_tW_{-t}&\tilde{u} = \partial_tW_{-t}B_1(\tilde{u})+\partial_tW_{-t}B_2(\tilde{u})+\partial_t W_{-t}B_3(\tilde{u})\\
    &\qquad+W_{-t}(\mathcal{NR}_1(\tilde{u})+\mathcal{NR}_2(\tilde{u})+\mathcal{NR}_3(\tilde{u})+C(\tilde{u}))\nonumber\\
    &\qquad+W_{-t}(R_1(\tilde{u})+R_2(\tilde{u})+R_3(\tilde{u}))\nonumber\\
    &\qquad+W_{-t}(\mathcal{RE}_1(\tilde{u})+\mathcal{RE}_2(\tilde{u})+\mathcal{RE}_3(\tilde{u})+\mathcal{RE}_4(\tilde{u}))\nonumber.
\end{align}
Since the only meaningful differences between estimates for $Y^s$ and $C^0_tH^s_x$ are applications of Strichartz estimates and modulation considerations, we will only show the proofs for the $\mathcal{NR}_i$, $R_i$, and $C(\tilde{u})$ terms. 
\begin{lemma}\label{Lemma: UCWP D terms}
Let $s> 1$. Then following bounds hold
\begin{align*}
\|(\mathcal{NR}_1+\mathcal{NR}_2+\mathcal{NR}_3)(\tilde{u})\|_{C^0_tH^s_x}&\lesssim \|\tilde{u}\|_{C^0_tH^s_x}^3+\|\tilde{u}\|_{C^0_tH^s_x}^4\\
\|(R_1+R_2+R_3)(\tilde{u})\|_{C^0_tH^s_x}&\lesssim \|\tilde{u}\|_{C^0_tH^s_x}^4+\|\tilde{u}\|_{C^0_tH^s_x}^6\\
\|C(\tilde{u})\|_{C^0_tH^s_x}&\lesssim \|\tilde{u}\|_{C^0_tH^s_x}^5.
\end{align*}
\end{lemma}
\begin{proof}
Beginning with the first term, we recall that we lack resonance and that the situation $n_1^*+n_2^* = 0$ was already handled in Lemma \ref{Lemma: B portion of D smooth} (without appeal to Strichartz estimates or modulation considerations). It thus suffices to consider the case that $(n_3^*)^4n_4^*\gtrsim (n_1^*)^4$.

As $\mathcal{NR}_1$ carries three large similar frequencies, we have by Lemma \ref{Lemma: Nonresonant Weight Bound} that the symbol for all terms is $O((n_1^*)^2/n_3^*)$. It follows by the case restriction that all of the symbols satisfy
\[
O\left(\frac{(n_1^*)^2}{n_3^*}\right) = O\left(\sqrt{n_2^*n_3^*n_4^*}\right).
\]
Restricting ourselves to the harder $\mathcal{NR}_2$ and $\mathcal{NR}_3$, we find by H\"olders and Sobolev embedding that
\begin{align*}
    \|\mathcal{NR}_2(\tilde{u})+\mathcal{NR}_3(\tilde{u})\|_{C^0_tH^s_x}&\lesssim \|(J_x^s\tilde{u})(J_x^{1/2}\tilde{u})^3\|_{C^0_tL^2_x}\lesssim \|\tilde{u}\|_{C^0_tH^s_x}^4,
\end{align*}
for $s > 1$.

The second bound follows from the fact that, by \eqref{Equation: Bilinear R Bound}, \eqref{Equation: Trilinear R Bound}, and \eqref{Equation: 4-linear R Bound} all of their symbols are $O(1)$. Restricting ourselves to $R_3$, and find by the symbol being $O(1)$ that
\begin{align*}
    \|R_3\|_{C^0_tH^s_x}\lesssim \|\tilde{u}^4\|_{C^0_tH^s_x}\lesssim \|\tilde{u}\|_{C^0_tH^s_x}^4,
\end{align*}
for $s > 1/2$.

As for the last bound, we note that \eqref{Equation: mA bound} gives
\begin{align*}
     m_\mathcal{A}^\ell = O\left(\frac{n_\ell^2+n_{\ell+1}^2}{|n|\max(n_1, \cdots, n_{\ell}+n_{\ell+1}, \cdots, n_5)}\right) = O\left(\frac{(n_1^*)^2}{n^2}\right).
\end{align*}
If $n_1^*\gg n_2^*$, then the symbol is $O(1)$ and the result is immediate as in the $R_j$ estimate. We now assume that $n_1^*\sim n_2^*\gg |n|$, so that we obtain by Cauchy-Schwarz and Sobolev embedding:
\begin{equation*}
    \|C(\tilde{u})\|_{C^0_tH^s_x}\lesssim \|J_x^{s-2}((J_x\tilde{u})^2u^3)\|_{C^0_tL^2_x}\\
    \lesssim\|(J_x^s\tilde{u})^2\tilde{u}^3\|_{C^0_tL^1_x}\lesssim \|\tilde{u}\|_{C^0_tH^s_x}^2\|\tilde{u}\|_{C^0_tH^{1/2+}_x}^3,
\end{equation*}
for $s\geq 1$. 
\end{proof}
\begin{proof}[Proof of Corollary \ref{Theorem: Unconditional Well-posedness}]
Modulo constants, we define the operators $\Gamma_1$ and $\Gamma_2$ by
\begin{align*}
\Gamma_1[v] &= W_t\{B_1(u_0)+B_2(u_0)+B_3(u_0)-u_0\} - \{B_1(v)+B_2(v)+B_3(v)\}\\
    &\qquad+\int_0^tW_{t-s}\left(C(v) + \mathcal{NR}_1(v)+\mathcal{NR}_2(v)+\mathcal{NR}_3(v) +l.o.t\right)\,ds,
\end{align*}
and
\begin{align*}
\Gamma_2[v] &= W_tu_0+\int_0^tW_{t-s}\left(\partial_x\left((\partial_x^2v)v+v(\partial_x^2v)\right) +\mathcal{K}(v)\partial_xv\right)\,ds.
\end{align*}
These are the operators associated to the Duhamel representation of the equivalent formulation of the problem after and before the differentiation-by-parts procedure, respectively. Combining these two, we obtain the operator
\[
    \Gamma^N[v] = \begin{cases}
                        \widehat{\Gamma_1[v]}(n) & |n|\geq N\\
                        \widehat{\Gamma_2[v]}(n) & |n| < N,
                  \end{cases}
\]
which acts as a workaround for the lack of any time factor $T$ on the $B$ terms.

Since $B_1$, $B_2,$ and $B_3$ are all smoother, we find
\begin{align*}
\|\widehat{\Gamma_1(v)}\|_{C^0_tH^s_x}&\lesssim T(\|v\|_{C^0_tH^s_x}^3 + \|v\|_{C^0_tH^s_x}^5)\\
&+ N^{-\alpha}(\|v\|_{C^0_tH^s_x}^2 + \|u_0\|_{H^s_x}^2 + \|v\|_{C^0_tH^s_x}^5 + \|u_0\|_{H^s_x}^5) + \|u_0\|_{H^s_x},
\end{align*}
for some $\alpha > 0$. We similarly find 
\begin{align*}
\|\widehat{\Gamma_2(v)}\|_{C^0_tH^s_x}&\lesssim T(N^{3}\|v\|_{C^0_tH^s_x}^3 + N\|v\|_{C^0_tH^s_x}^3)+\|u_0\|_{H^s_x}.
\end{align*}
Since similar bounds hold for the difference, we see that we may select $N = N(\|u_0\|_{H^s_x})$ large enough to make 
\[
N^{-\alpha}(\|u_0\|_{H^s_x}^2 + \|u_0\|_{H^s_x}^5)\ll \|u_0\|_{H^s_x},
\]
and then $T$ small enough to make both
\begin{align*}
    T(\|u_0\|_{H^s_x}^3 + \|u_0\|_{H^s_x}^5)&\ll \|u_0\|_{H^s_x}\\
    T(N^3\|u_0\|_{H^s_x}^2 + N\|u_0\|_{H^s_x}^3) &\ll \|u_0\|_{H^s_x}.
\end{align*}

This enables us to perform a contraction argument in $C^0_tH^s_x$ for $s > 1$, and hence standard arguments allow us to conclude unconditional well-posedness.
\end{proof}

\section{Acknowledgements}
The author would like to thank Professor Burak Erdo{\u{g}}an for not only recommending this problem, but for his support and guidance. Similarly, the author would like to thank Professor Nikolaos Tzirakis for several fruitful conversations over the vaguest of questions. 

\section{Appendix: Toy Equation Reduction Justification}\label{Section: Appendix A}
Here we will justify the reduction to \eqref{Equation: Toy Fifth Order}. Consider the full Equation \eqref{Equation: General Fifth Order}, given again here in the more symmetric form on the Fourier side as
\begin{align*}
    \begin{cases}
        \partial_t\widehat{u} - ik^5 \widehat{u} = \alpha_1 in \sum_{n=n_1+n_2+n_3}\widehat{u}_1\widehat{u}_2\widehat{u}_3\\
        \qquad\qquad+ \alpha_2 in\sum_{n=n_1+n_2}n_1n_2\widehat{u}_1\widehat{u}_2 + \alpha_3 in\sum_{n=n_1+n_2}(n_1^2+n_2^2)\widehat{u}_1\widehat{u}_2\\
        u(x,0) = u_0\in H^s(\mathbb{T}).
    \end{cases}
\end{align*}
First note that the cubic term is given as
\begin{align}
\alpha_1 in\sum_{n=n_1+n_2+n_3}\widehat{u}(n_1)\widehat{u}(n_2)\widehat{u}(n_3) &= \alpha_1 in \sum_{\substack{n = n_1+n_2+n_3\\ \star}}\widehat{u}(n_1)\widehat{u}(n_2)\widehat{u}(n_3)\nonumber\\
&\qquad+3\alpha_1 in\widehat{u}(n)\sum_{\substack{0=n_2+n_3}}\widehat{u}(n_2)\widehat{u}(n_3)\nonumber\\
&\qquad-2\alpha_1in\widehat{u}(n)\widehat{u}(n)\widehat{u}(-n)\nonumber\\
&=\alpha_1 in \sum_{\substack{n = n_1+n_2+n_3\\ \star}}\widehat{u}(n_1)\widehat{u}(n_2)\widehat{u}(n_3)\label{Equation: Appendix first Cubic Expression}\\
&\qquad+3\alpha_1in\widehat{u}(n)\sum_{\substack{n_2}}|\widehat{u}|^2(n_2)\nonumber\\
&\qquad-2\alpha_1in\widehat{u}(n)|\widehat{u}|^2(n)\label{Equation: Appendix second Cubic Expression},
\end{align}
under the assumption of real $u$. By using a change of variables as before we may remove 
\[
3\alpha_1in\widehat{u}(n)\sum_{\substack{n_2}}|\widehat{u}|^2(n_2).
\]
With only one derivative appearing in the cubic expressions \eqref{Equation: Appendix first Cubic Expression} and \eqref{Equation: Appendix second Cubic Expression}, we see that these term estimatable in $Y^s$ (resp. $C^0_tH^s_x$) for $s > 1/2$ (resp. $s > 1$), so this term will not need to be differentiated-by-parts.

We then switch to interaction variables and perform differentiate-by-parts for each remaining term. Specifically, we let $v = W_{-t}u.$ We will often ignore subscripts on the summations and, when obvious, ignore the entries in $\Phi$. Moreover, we ignore terms that ought to appear due to the decomposition Proposition \ref{Proposition: Symbol Decomp}, such as when $\Phi_k\ll (n_1^*)^4$. For simplicity we 
write $v_j$ for $\widehat{v}(n_j)$ and define the notation 
\[
\mathcal{N}_k = e^{-it\Phi_k}\prod_{j=1}^kv_j.
\]
For multilinear operators $\mathcal{M}, \mathcal{O},$ and $\mathcal{P}$ of the form 
\[
\widehat{\mathcal{Q}} = \sum_{n=n_1+\cdots +n_k}q(n,n_1,\cdots, n_k)\prod_{j=1}^k v_j
\]
with symbols $m, o,$ and $p$ respectively, we write
\[
\mathcal{M} = \mathcal{O}+O\left(\mathcal{P}\right), \mbox{ if }m-o = O(p).
\]
For multilinear operators of the form
\[
\widehat{v}(n)\sum_{0=n_2+\cdots + n_k}q(n,n_2, \cdots, n_k)\prod_{j=2}^k v_j
\]
we write $n_2^*$ to denote the largest frequency, in magnitude, among $n_2, \cdots, n_k$ (This is inherited from the fact that such an operator arises from setting $n_1 = n$). Additionally, we will frequently perform large frequency single resonance approximations, in that we will assume that $n_j=n$ for some $j$ and that $n^4\gg (n_2^*)^5$ (We note that the latter restriction is only used for one resonance as in Lemma \ref{Lemma: Resonant terms without M are smooth}). In order to efficiently work with these large frequency approximations, we adopt the notation
\[
\mathcal{M} = \mathcal{O}+O_{A}\left(\mathcal{P}\right), \mbox{ if }m-o = O(p)
\]
for $(n, n_1, \cdots, n_k)\in A$, where $\mathcal{M},$ $\mathcal{O}$, and $\mathcal{P}$ are $k$-linear operators of the form above and $A\subset \mathbb{Z}^{k+1}$. Specifically for $k\in\mathbb{N}$ and $1\leq i\leq k$, we define the set 
\begin{align*}
    \mathcal{C}_i &= \{(n, n_1, \cdots, n_k) \,:\,n = n_i\gg n_2^*\},
\end{align*}
which will be used in the above context.

For what is to come, we focus solely on the resonance interactions, for otherwise the estimates are strictly easier than those performed in the body of this manusript. We also make the observation that resonance, by assumption, can only occur in the variables introduced by the most recent temporal differentiation. Indeed, we always remove prior resonances in order to assume that $\Phi_k\gtrsim (n_1^*)^4$. 

The last comment we make before continuing with the calculation is that we drop references to $\Phi_k(n_1, \cdots n_k)\gtrsim \max(n_1,\cdots, n_k)^4$. Inded, all resonances technically carry a chain of such assumptions as in the proof of Lemma \ref{Lemma: Resonant terms without M are smooth} and they are handled in the same manner.

\textbf{Cubic Term:} We find
\begin{align}
    \alpha_1 in \sum_{\substack{n = n_1+n_2+n_3\\ \star}}e^{-it\Phi_3(n,n_1,n_2,n_3)}v_1v_2v_3 &= -\alpha_1\partial_t\left(\sum\tfrac{n}{\Phi_3(n,n_1,n_2,n_3)}\mathcal{N}_3\right)\nonumber\\
    &+3\alpha_1^2\sum\tfrac{in(n_1+n_2+n_3)}{\Phi_3(n,n_1+n_2+n_3,n_4,n_5)}\mathcal{N}_5\label{Equation: Appendix A1A1 term}\\
    &+3\alpha_1\alpha_2\sum\tfrac{in(n_1+n_2)n_1n_2}{\Phi_3(n,n_1+n_2, n_3,n_4)}\mathcal{N}_4\label{Equation: Appendix A1A2 term}\\
    &+3\alpha_1\alpha_3\sum\tfrac{in(n_1+n_2)(n_1^2+n_2^2)}{\Phi_3(n,n_1+n_2, n_3,n_4)}\mathcal{N}_4\label{Equation: Appendix A1A3 term},
\end{align}
where we find that the symbols of \eqref{Equation: Appendix A1A1 term} and \eqref{Equation: Appendix A1A2 term} are $O(1/n^2)$ and $O(\tfrac{n_1^*n_2^*}{n\max(|n_3^*|,|n|)})$, respectively. It follows that their resonances have symbols that are $O(n_2^*/n)$, and hence we may estimate these terms in the same manner as \eqref{Equation: Grouped terms smoothing reference} and Lemma \ref{Lemma: Main term smoothing estimate}.

As the symbol of \eqref{Equation: Appendix A1A3 term} is $O(\tfrac{(n_1^*)^2}{\max(|n_3^*|,|n|)^2})$, we find that the only non-smooth\footnote{That is, a resonance with symbol worse than $O(1)$. In particular, the term above \eqref{Equation: Appendix A1A3 term} has $\mathcal{C}_i$ approximation that has a symbol that is $O(n_2^*/n)$.} resonance here is the single resonance in \eqref{Equation: Appendix A1A3 term}. We find easily that the $n\gg n_2^*$ approximation to the single resonant symbol is
\begin{equation}\label{Equation: Appendix A1A3 large frequency approxmination}
    \eqref{Equation: Appendix A1A3 term} = 3\alpha_1\alpha_3\widehat{v}(n)\sum_{0=n_2+n_3+n_4}\frac{i}{5n_2}v_2v_3v_3 + O_{\mathcal{C}_1}\bigg(v(n)\sum_{0=n_2+n_3+n_4}\frac{n_2^*}{nn_3^*}v_2v_3v_4\bigg),
\end{equation}
where the double resonance has symbol that is $O(1)$. It follows that we may handle all but the single resonance using the same methods as in Lemmas \ref{Lemma: Resonant Duhamel Bounds} and \ref{Lemma: Main term smoothing estimate}.

{\bfseries First Quadratic Term:}
\begin{align}
    \alpha_2 in \sum_{\substack{n = n_1+n_2\\ \star}}ie^{-it\Phi_2(n,n_1,n_2)}n_1v_1n_2v_2 &= -\alpha_2\partial_t\left(\sum\tfrac{nn_1n_2}{\Phi_2(n,n_1,n_2)}\mathcal{N}_2\right)\nonumber\\
    &+2\alpha_2\alpha_1\sum\tfrac{in(n_1+n_2+n_3)^2n_4}{\Phi_2(n,n_1+n_2+n_3,n_4)}\mathcal{N}_4\label{equation: Appendix A2A1 term}\\
    &+2\alpha_2^2\sum\tfrac{in(n_1+n_2)^2n_1n_2n_3}{\Phi_2(n,n_1+n_2, n_3)}\mathcal{N}_3\label{Equation: Appendix A2A2 term}\\
    &+2\alpha_2\alpha_3\sum\tfrac{in(n_1+n_2)^2(n_1^2+n_2^2)n_3}{\Phi_2(n,n_1+n_2, n_3)}\mathcal{N}_3\label{Equation: Appendix A2A3 term},
\end{align}
where the symbol of \eqref{equation: Appendix A2A1 term} is $O(n_2^*/n)$. It follows that we may estimate \eqref{equation: Appendix A2A1 term} in the same manner as before. However, the other two terms have symbols that are $O(n_1^*n_2^*/\max(|n|, |n_3|))$ and $O((n_1^*)^2/\max(|n|, |n_3|))$, respectively. The non-resonant portion of these can be estimated as in Lemma \ref{Lemma: Main term smoothing estimate}, but the resonant portion will require more care. The first of these terms has large frequency single resonance approximation
\begin{align}
    \eqref{Equation: Appendix A2A2 term} &= \frac{2\alpha_2^2}{5}\widehat{v}(n)\sum_{0=n_2+n_3}in_3v_2v_3 + O_{\mathcal{C}_1}\bigg(\widehat{v}(n)\sum_{0=n_2+n_3}\frac{n_2^*n_3}{n}v_2v_3\bigg)\nonumber\\
    &=O_{\mathcal{C}_1}\bigg(\widehat{v}(n)\sum_{0=n_2+n_3}\frac{n_2^*n_3}{n}v_2v_3\bigg),
\end{align}
by realness and parity considerations.

Similarly, we find
\begin{align}
    \eqref{Equation: Appendix A2A3 term} &= \frac{2\alpha_2\alpha_3in}{5}\widehat{v}(n)\sum_{0=n_2+n_3}\frac{in_3}{n_2}v_2v_3 +\frac{4\alpha_2\alpha_3i}{5}\sum_{0=n_2+n_3}n_3v_2v_3\nonumber\\
    &\qquad+O_{\mathcal{C}_1}\bigg(\widehat{v}(n)\sum_{0=n_2+n_3}\frac{n_2^2}{n}v_2v_3\bigg)\nonumber\\
    &=\frac{2\alpha_2\alpha_3in}{5}\widehat{v}(n)\sum_{0=n_2+n_3}\frac{n_3}{n_2}v_2v_3 + O_{\mathcal{C}_1}\bigg(\widehat{v}(n)\sum_{0=n_2+n_3}\frac{n_2^2}{n}v_2v_3\bigg),
\end{align}
by similar considerations as in the prior display.

{\bfseries Second Quadratic Term:}

\begin{align}
    \alpha_3 in \sum_{\substack{n = n_1+n_2\\ \star}}e^{-it\Phi_2(n,n_1,n_2)}v_1v_2 &= -\alpha_3\partial_t\left(\sum\tfrac{n(n_1^2+n_2^2)}{\Phi_2(n,n_1,n_2)}\mathcal{N}_2\right)\nonumber\\
    &+2\alpha_3\alpha_1\sum\tfrac{in(n_1+n_2+n_3)((n_1+n_2+n_3)^2+n_4^2)}{\Phi_2(n,n_1+n_2+n_3,n_4)}\mathcal{N}_4\label{Equation: Appendix A3A1 term}\\
    &+2\alpha_3\alpha_2\sum\tfrac{in(n_1+n_2)((n_1+n_2)^2+n_3^2)n_1n_2}{\Phi_2(n,n_1+n_2, n_3)}\mathcal{N}_3\label{Equation: Appendix A3A2 term}\\
    &+2\alpha_3^2\sum\tfrac{in(n_1+n_2)((n_1+n_2)^2+n_3^2)(n_1^2+n_2^2)}{\Phi_2(n,n_1+n_2, n_3)}\mathcal{N}_3\label{Equation: Appendix A3A3 term},
\end{align}
where we find the above terms to have symbols that are $O(1/n_4^*)$, $O(n_1^*n_2^*/n_3^*)$, and $O((n_1^*)^2/n_3^*)$, respectively. When non-resonant, we note that the proof of Lemma \ref{Lemma: B portion of D smooth} handles the all of these (the first two explicitly, and the last is easier by Proposition \ref{Proposition: small symbol decomp} by noting that we'll further differentiate-by-parts the last term, so that we only must worry about the case that $n_3^*\sim n_1^*$).

We now demonstrate their large frequency resonant approximations in order.
\begin{align}
    \eqref{Equation: Appendix A3A1 term} = \frac{2\alpha_3\alpha_1}{5} \sum_{0=n_2+n_3+n_4}\frac{i}{5(n_2+n_3)}v_2v_3v_4 + O_{\mathcal{C}_1}\bigg(\widehat{v}(n)\sum_{0=n_2+n_3+n_4}\frac{n_2^*}{n}v_2v_3v_4\bigg).\label{Equation: Appendix A3A1 large frequency approxmination}
\end{align}
The second term will have a non-trivial $n$ term:
\begin{align*}
\eqref{Equation: Appendix A3A2 term} &= \frac{2\alpha_3\alpha_2in}{5}\widehat{v}(n)\sum_{0=n_2+n_3}v_2v_3 +\frac{8\alpha_3\alpha_2}{5}\sum_{0=n_2+n_3}n_2v_2v_3\\
    &\qquad+O_{\mathcal{C}_1}\bigg(\widehat{v}(n)\sum_{0=n_2+n_3}\frac{n_2^2}{n}v_2v_3\bigg)\nonumber\\
    &=\frac{2\alpha_3\alpha_2in}{5}\widehat{v}(n)\sum_{0=n_2+n_3}v_2v_3 + O_{\mathcal{C}_1}\bigg(\widehat{v}(n)\sum_{0=n_2+n_3}\frac{n_2^2}{n}v_2v_3\bigg)\label{Equation: Appendix A3A2 large frequency approxmination},
\end{align*}
again by parity considerations.

Finally, we have
\begin{align}
\eqref{Equation: Appendix A3A3 term} &= \frac{2\alpha_3^2in^2}{5}\widehat{v}(n)\sum_{0=n_2+n_3}\frac{1}{n_2}v_2v_3 +\frac{6\alpha_3^2in}{5}\sum_{0=n_2+n_3}v_2v_3\\
    &\qquad+\frac{2\alpha_3^2}{5}\sum_{0=n_2+n_3}in_2v_2v_3+O_{\mathcal{C}_1}\bigg(\widehat{v}(n)\sum_{0=n_2+n_3}\frac{n_2^2}{n}v_2v_3\bigg)\nonumber\\
    &=\frac{6\alpha_3^2in}{5}\sum_{0=n_2+n_3}v_2v_3 + O_{\mathcal{C}_1}\bigg(\widehat{v}(n)\sum_{0=n_2+n_3}\frac{n_2^2}{n}v_2v_3\bigg),\label{Equation: Appendix A3A3 large frequency approxmination}    
\end{align}
    
{\bfseries D.B.P. of Non-resonant \eqref{Equation: Appendix A2A3 term} Term:}

We now assume that it is non-resonant and perform the differentiation-by-parts procedure to it to find
\begin{align}
    \eqref{Equation: Appendix A2A3 term} &= 2\alpha_2\alpha_3\partial_t\left(\sum_{\substack{n=n_1+n_2+n_3\\\star}}\tfrac{n(n_1+n_2)^2(n_1^2+n_2^2)n_3}{\Phi_2(n,n_1+n_2, n_3)\Phi_3}\mathcal{N}_3\right)\nonumber\\
    &-2\alpha_2\alpha_3\alpha_3 \sum \tfrac{in(n_1+n_2+n_3)^2(n_1+n_2)((n_1+n_2)^2+n_3^2)(n_1^2+n_2^2)n_4}{\Phi_2(n,n_1+n_2+n_3, n_4)\Phi_3(n,n_1+n_2,n_3,n_4)}\mathcal{N}_4\label{Equation: Apppendix A2A3A3 term}\\
    &+l.o.t.,\nonumber
\end{align}
where $l.o.t.$ denotes the substitution of the cubic and first quadratic into any of the terms, as well as the second quadratic into the $n_3$ term. We do not include these because they are smoother already. Indeed, note that we find (on the region we differentiated-by-parts) by the factorization of $\Phi_2$ and the bound on $\Phi_3$ that
\[
\frac{n(n_1+n_2)^2(n_1^2+n_2^2)n_3}{\Phi_2(n,n_1+n_2, n_3)\Phi_3} = O\left(\frac{|n_1|+|n_2|}{\max(|n_1+n_2|,|n_3|)\max(|n_1|,|n_2|,|n_3|)^3}\right).
\]
It follows that if we substitute the cubic term or the first quadratic term into $n_1$ or $n_2$ then the symbol becomes at worst
\begin{align*}
O\big(\tfrac{(|n_1+n_2|+|n_3|)nn_1n_2}{\max(|n_1+n_2+n_3|,|n_3|)\max(|n_1+n_2|,|n_3|,|n_4|)^3}\big)&= O\big(\tfrac{n_1n_2}{\max(|n_1+n_2+n_3|,|n_3|)\max(|n_1+n_2|,|n_3|,|n_4|)}\big)\\
&= O\big(\tfrac{n_1n_2}{nn_3^*}\big).
\end{align*}
As the high frequency resonant approximation is $O_{\mathcal{C}_1}(n_2^*/n)$, we may estimate these as in Lemmas \ref{Lemma: Resonant Duhamel Bounds} and \ref{Lemma: Main term smoothing estimate}. Similarly, if we substitute the second cubic into the $n_3$ term, then we find that the symbol is
\begin{align*}
O\big(&\tfrac{(|n_1|+|n_2|)(n_3+n_4)(n_3^2+n_4^2)}{\max(|n_1+n_2|,|n_3+n_4|)\max(|n_1|,|n_2|,|n_3+n_4|)^3}\big)=O\big(\tfrac{n_1^*n_2^*}{nn_3^*}\big).
\end{align*}
As the resonant approximation is again seen to have symbol $O_{\mathcal{C}_i}(n_2^*/n)$, we find that we may estimate these using the same methods as in the prior two cited lemmas.

As for the remaining term, we find the following large frequency resonant approximation:
\begin{align}
    \eqref{Equation: Apppendix A2A3A3 term} &= -\frac{4\alpha_2\alpha_3^2i}{25}\widehat{v}(n)\sum_{0=n_2+n_3+n_4}\frac{n_4}{n_2(n_2+n_3)}v_2v_3v_4+O_{\mathcal{C}_1}\bigg(\widehat{v}(n)\sum_{0=n_2+n_3+n_4}\frac{n_2^*}{n}v_2v_3v_4\bigg)\label{Equation: Appendix A2A3A3 large frequency approxmination}
\end{align}

{\bfseries D.B.P. of Non-resonant \eqref{Equation: Appendix A3A2 term} Term:}

We obtain
\begin{align}
    \eqref{Equation: Appendix A3A2 term} &= 2\alpha_3\alpha_2\partial_t\left(\sum\tfrac{n(n_1+n_2)((n_1+n_2)^2+n_3^2)n_1n_2}{\Phi_2(n,n_1+n_2, n_3)\Phi_3(n,n_1,n_2,n_3)}\mathcal{N}_3\right)\nonumber\\
    &-4\alpha_3\alpha_2\alpha_3\sum\tfrac{in(n_1+n_2+n_3)(n_1+n_2)^2((n_1+n_2+n_3)^2+n_4^2)(n_1^2+n_2^2)n_3}{\Phi_2(n,n_1+n_2+n_3, n_3)\Phi_3(n,n_1+n_2,n_3,n_4)}\mathcal{N}_4\label{Equation: Apppendix A3A2A3 term}\\
    &+l.o.t.,\nonumber
\end{align}
where again the $l.o.t.$ are generated by substitution of the cubic term anywhere, and substitution of any term into the $n_3$ variable. Indeed, we find by the factorization of $\Phi_2$ and the assumed region of the sum that
\[
O\left(\frac{n(n_1+n_2)((n_1+n_2)^2+n_3^2)n_1n_2}{\Phi_2(n,n_1+n_2, n_3)\Phi_3(n,n_1,n_2,n_3)}\right) = O\left(\frac{n_1}{n_3\max(|n_1|,|n_2|,|n_3|)^3}\right).
\]
It follows that the substitution of the cubic into any term will be easily boundable, and that the substitution into $n_3$ of either quadratic term will have the exact same bounds as in the \eqref{Equation: Appendix A2A3 term} case.

We now turn our attention to the resonant terms, which have large frequency resonant approximations given by
\begin{align}
    \eqref{Equation: Apppendix A3A2A3 term} &= -\frac{4\alpha_2\alpha_3^2i}{25}\widehat{v}(n)\sum_{0=n_2+n_3+n_4}\frac{n_3}{n_2(n_2+n_3)}v_2v_3v_4+O_{\mathcal{C}_1}\left(\widehat{v}(n)\sum_{0=n_2+n_3+n_4}\frac{n_2^*}{n}v_2v_3v_4\right)\label{Equation: Appendix A3A2A3 large frequency approxmination}.
\end{align}

{\bfseries D.B.P. of Non-resonant \eqref{Equation: Appendix A3A3 term} Term:}

The resonances for this term were already discussed in previous sections, so we simply proceed with the differentiation-by-parts. We find
\begin{align}
    \eqref{Equation: Appendix A3A3 term} &= 2\alpha_3^2\partial_t\left(\sum\tfrac{n(n_1+n_2)((n_1+n_2)^2+n_3^2)(n_1^2+n_2^2)}{\Phi_2(n,n_1+n_2, n_3)\Phi_3(n,n_1,n_2,n_3)}\mathcal{N}_3\right)\nonumber\\
    &-4\alpha_3^2\alpha_2\sum\tfrac{in(n_1+n_2+n_3)(n_1+n_2)((n_1+n_2+n_3)^2+n_4^2)((n_1+n_2)^2+n_3^2)n_1n_2}{\Phi_2(n,n_1+n_2+n_3, n_4)\Phi_3(n,n_1+n_2,n_3,n_4)}\mathcal{N}_4\label{Equation: Appendix A3A3A2 term}\\
    &-4\alpha_3^3\sum\tfrac{in(n_1+n_2+n_3)(n_1+n_2)((n_1+n_2+n_3)^2+n_4^2)((n_1+n_2)^2+n_3^2)(n_1^2+n_2^2)}{\Phi_2(n,n_1+n_2+n_3, n_4)\Phi_3(n,n_1+n_2,n_3,n_4)}\mathcal{N}_4\label{Equation: Appendix A3A3A3 term}\\
    &+l.o.t.,\nonumber
\end{align}
where the $l.o.t.$ are generated by the substitution of the cubic term into any term, and the substitution of any term into the $n_3$ component. The $l.o.t.$ satisfy the same resonant and non-resonant symbol bounds as the prior terms. Indeed, we find by the factorization of $\Phi_2$ and the assumed region of the sum that
\begin{equation}\label{Equation: A3A3 L.o.t. justification}
O\left(\frac{n(n_1+n_2)((n_1+n_2)^2+n_3^2)(n_1^2+n_2^2)}{\Phi_2(n,n_1+n_2, n_3)\Phi_3(n,n_1,n_2,n_3)}\right) = O\left(\frac{|n_1|+|n_2|}{n_3\max(|n_1|,|n_2|,|n_3|)^3}\right).
\end{equation}
It follows that the substitution of the cubic into any term can easily be handled, and that the substitution into $n_3$ of either quadratic term will have the exact same bounds as in the \eqref{Equation: Appendix A2A3 term} case.

Turning our attention to the resonant terms, we find that the high-frequency resonant decomposition of \eqref{Equation: Appendix A3A3A2 term} is:
\begin{align}
    \eqref{Equation: Appendix A3A3A2 term} &= -\frac{4\alpha_3^2\alpha_2i}{25}\widehat{v}(n)\sum_{0=n_2+n_3+n_4}\frac{n_2}{n_2(n_2+n_3)}v_2v_3v_4\label{Equation: Appendix A3A3A2 large frequency approxmination}\\
    &\qquad\qquad+O_{\mathcal{C}_1}\bigg(\widehat{v}(n)\sum_{0=n_2+n_3+n_4}\frac{n_2^*}{n}v_2v_3v_4\bigg)\nonumber
\end{align}

As the \eqref{Equation: Appendix A3A3A3 term} term was analyzed in \eqref{Equation: mB resonant bound}, and shown to have decaying $n^0$ term in Remark \ref{Remark: Why Resonant Smoother Bound Holds}.

{\bfseries Non-resonant \eqref{Equation: Appendix A3A3A3 term} Term:}

We proceed with the differentiation-by-parts to find
\begin{align}
    \eqref{Equation: Appendix A3A3A3 term}  &= -4\alpha_3^3\partial_t\bigg(
    \sum_{n=n_1+\cdots+n_4}\tfrac{n(n_1+n_2+n_3)(n_1+n_2)((n_1+n_2+n_3)^2+n_4^2)((n_1+n_2)^2+n_3^2)(n_1^2+n_2^2)}{\Phi_2(n,n_1+n_2+n_3, n_4)\Phi_3(n,n_1+n_2,n_3,n_4)\Phi_4}\mathcal{N}_4\bigg)\nonumber\\
    &+8\alpha_3^4\sum\tfrac{in(n_1+n_2+n_3+n_4)(n_1+n_2+n_3)((n_1+n_2+n_3+n_4)^2+n_5^2)}{\Phi_2(n,n_1+n_2+n_3+n_4, n_5)\Phi_3(n,n_1+n_2+n_3,n_4,n_5)}\label{Equation: Appendix A3A3A3A3 term}\\
    &\qquad\qquad\qquad\times\tfrac{((n_1+n_2+n_3)^2+n_4^2)((n_1+n_2)^2+n_3^2)(n_1+n_2)(n_1^2+n_2^2)}{\Phi_4(n,n_1+n_2,n_3,n_4,n_5)}\mathcal{N}_4\nonumber\\
    &+l.o.t.
\end{align}
The $l.o.t$ arise from all other substitutions besides the substitution of the second quadratic into either $n_1$ or $n_2$. We note that, by \eqref{Equation: A3A3 L.o.t. justification}, we have that the symbol satisfies
\begin{align*}
    O\bigg(&\frac{n(n_1+n_2+n_3)(n_1+n_2)((n_1+n_2+n_3)^2+n_4^2)((n_1+n_2)^2+n_3^2)}{\Phi_2(n,n_1+n_2+n_3, n_4)\Phi_3(n,n_1+n_2,n_3,n_4)}\nonumber\\
    &\qquad\qquad\qquad\qquad\times\frac{(n_1^2+n_2^2)}{\Phi_4(n,n_1,n_3,n_4,n_5)}\bigg)\\
    &=O\left(\frac{(n_1^2+n_2^2)}{n_4n\max(|n_1|,|n_2|,|n_3|,|n_4|)^4}\right)=O\left(\frac{1}{n_4n\max(|n_1|,|n_2|,|n_3|,|n_4|)^2}\right).
\end{align*}
It follows that if we substitute any term into either $n_3$ or $n_4$ then we find that the symbol satisfies (at worst)
\begin{equation}\label{Equation: A3A3A3 l.o.t order}
O\left(\frac{(n_3+n_4)(n_3^2+n_4^2)}{n_5n\max(|n_1|,|n_2|,|n_3+n_4|,|n_5|)^2}\right)=O\left(\frac{(n_1^*)^2}{nn_3^*}\right).
\end{equation}
This, together with the fact that resonances are again $O(n_2^*/n)$ is good enough to repeat the lemmas cited above. Now, if we substitute the cubic or first quadratic into $n_1$ or $n_2$ we again find the bound \eqref{Equation: A3A3A3 l.o.t order}-- with $O(n_2^*/n)$ symbol on the resonant approximation for free from the fact that it will be one less than the corresponding one for the second quadratic.

Additionally, we find that the single resonance induced from \eqref{Equation: Appendix A3A3A3A3 term} has already been analyzed in Lemma \ref{Lemma: Final Resonance Bound} and Remark \ref{Remark: Higher order resonances and smoothing}. Specifically, the main term decays as in Remark \ref{Remark: Higher order resonances and smoothing}.

\subsection{Resonance Analysis}

In this section we now combine the single resonances and show decay of the $n^0$ terms that arise. While we \textit{don't} need to perform this analysis to conclude that the well-posedness and unconditional well-posedness bounds of Theorems \ref{Theorem: Wellposedness for Toy}, \ref{Theorem: Wellposedness for Toy with Pa}, and \ref{Theorem: Unconditional Well-posedness} extend to \eqref{Equation: General Fifth Order}, we do need this information to conclude that Theorem \ref{Theorem: Smoothing} will. 

In fact, the only thing we need in order to show the extension of well-posedness and unconditional well-posedness is that we create no new non-decaying high-frequency approximations to single resonance terms of the form 
\[
in^2\widehat{v}(n)\mathcal{Q}
\]
where $\mathcal{Q}$ is a multilinear operator that is \textit{imaginary} for real $u$. Indeed-- all other cases are easily estimated by prior lemmas in this manuscript, as noted.

{\bfseries Trilinear terms:}

We see from the prior analysis that the high-frequency approximation to the $n^0$ and $n^2$ terms that arise from trilinear operators vanish completely due to parity considerations and symmetry. Specifically, they all involve a sum of the form
\[
\sum_{\substack{0=n_2+n_3\\ |n_2|,|n_3|\ll n}}p(n_2)v_2v_3,
\]
where $p$ is an odd polynomial. 

The only terms that remain are those involving $n^1$, which are are of the form
\[
in\widehat{v}(n)\mathcal{Q}
\]
where $\mathcal{Q}$ is a multilinear operator that is real for real $u$. It follows that we may use a change of variables to get rid of them, and such a transformation will preserve $H^s$ norms. The penalty for performing this change of variables is an additional term, analogous to the $R_i$ terms of Lemmas \ref{Lemma: Correction Term Penalties} and \ref{Lemma: Correction Term smoothing}. As the symbol of the resulting multilinear operators, after realizing them as not a product of two multilinear operators but as a single multilinear operator, are all $O(1)$, we find wellposedness and smoothing by modulation considerations as in the referenced lemmas.

{\bfseries Quartilinear Terms:}

By our prior comments on which variables can equal $n$, we collect all single resonances of this type. In particular, we find that we have two single resonances from \eqref{Equation: Appendix A2A3A3 large frequency approxmination}, \eqref{Equation: Appendix A3A2A3 large frequency approxmination}, \eqref{Equation: Appendix A3A3A2 large frequency approxmination}, and \eqref{Equation: Appendix A1A3 large frequency approxmination}, and three from \eqref{Equation: Appendix A3A1 large frequency approxmination}. Collecting these, we see that the $\alpha_1\alpha_3$ resonances satisfy
\begin{align*}
    \mbox{Resonances from }\eqref{Equation: Appendix A1A3 large frequency approxmination} + \eqref{Equation: Appendix A3A1 large frequency approxmination} &= 6\alpha_1\alpha_3\widehat{v}(n)\sum_{0=n_2+n_3+n_4}\left(\frac{1}{n_2}+\frac{1}{n_2+n_3}\right)v_2v_3v_4\\
    & = 0,
\end{align*}
by symmetry. Similarly, we find
\begin{align*}
    &\mbox{Resonances from }\eqref{Equation: Appendix A2A3A3 large frequency approxmination}+\eqref{Equation: Appendix A3A2A3 large frequency approxmination}+\eqref{Equation: Appendix A3A3A2 large frequency approxmination}\\
    &\qquad=-8\alpha_2\alpha_3^2\widehat{v}(n)\sum_{0=n_2+n_3+n_4}\left(\frac{n_2}{n_2(n_2+n_3)}+\frac{n_3}{n_2(n_2+n_3)}+\frac{n_4}{n_2(n_2+n_3)}\right)v_2v_3v_4\\
    &\qquad = 0,
\end{align*}
by the assumption on the summation.

{\bfseries Discussion}

As we've already shown that the quintilinear $n^0$ term vanishes, we have that all of the $n^0$ single resonant terms vanish. As every other term has symbols that are of the same magnitude as than the corresponding terms arising from \eqref{Equation: Toy Fifth Order}, we find that the corresponding theorems hold if and only if they hold for \eqref{Equation: Toy Fifth Order}. 

It's interesting to note that the above analysis made no reference to the values of $\alpha_1,$ $\alpha_2$, or $\alpha_3$. That is, the results here are a function of the structure of the equation, and not the value of the coefficients. 


\section{Appendix: Bounds on the Dimension of solutions to \eqref{Equation: General Fifth Order}}

The study of dimension bounds for the free solutions to dispersive linear PDEs goes back at least as far as \cite{Ta}, who studied the (named in their honor) Talbot effect. Oskolkov, \cite{Os1}, then showed that for bounded variation data the solution to any linear dispersive PDE on $\mathbb{T}$ with polynomial dispersion relation is a \textit{continuous} function of $x$ for \textit{every} irrational time, $t$. Using this, Rodnianski, \cite{Ro}, was able to show that the fractal dimension associated to the graph of $e^{it\partial_x^2}f$ is exactly $3/2$ for any bounded variation $f\in BV\setminus H^{1/2+}(\mathbb{T})$. Erdo{\u g}an and Shakan, \cite{ErdoganShakanFractal}, were then able to extend this to other dispersion relations, while Erdo{\u g}an and Tzirakis, \cite{erdogan2013talbot}, were able to then use nonlinear smoothing to extend the dimension claim to the nonlinear cubic NLS on $\mathbb{T}$. 

The purpose of this appendix is then to extend the results of \cite{erdogan2013talbot} to the case of the fifth order Airy group. We first recall the definition of the upper Minkowski dimension, which will be our notion of dimension 
\begin{definition}[Upper Minkowski Dimension]
For any bounded set $E\subset \mathbb{R}^N$, we define the \textit{Upper Minkowski Dimension} of $E$, denoted $\overline{\dim}(E)$ by
\[
\overline{\dim}(E) = \limsup_{\varepsilon\to 0}\frac{\log\mathcal{N}(E, \varepsilon)}{\log(\frac{1}{\varepsilon})},
\]
where $\mathcal{N}(E,\varepsilon)$ is the minimum number of $\varepsilon-$balls required to cover $E$. 

We define $D(f)$ to be the maximum of the dimension of $\Re f$ and $\Im f$.
\end{definition}
We will also need the four following Theorems, which are phrased in terms of $\gamma-$H\"older and Besov norms, with the latter being defined\footnote{Technically this requires a smooth cutoff \`a la Littlewood-Paley, but that we ignore this technicality as in \cite{ErdoganShakanFractal}.} by
\begin{equation}
    \|f\|_{B_{p,q}^s} := \|N^s\|P_N(f)\|_{L^p_x}\|_{\ell^q_N}.
\end{equation}
\begin{theorem}[\cite{ErdoganShakanFractal}, Theorem 3.9]\label{Theorem: Dimension Lower Bound}
Let $\omega$ be a real polynomial with integer coefficients. Assume that $e^{it\omega(\frac{\nabla}{i})}$ satisfies a Strichartz estimate of the form
\begin{equation*}
\|e^{it\omega(\frac{\nabla}{i})}f\|_{L^p_tL^q_x}\lesssim \|f\|_{H^s_x}
\end{equation*}
for some $s\in[0,\frac{1}{2})$, $2 < q < \infty$, and $p\in[1,\infty].$ Let $r_0 := \sup\{r\,:\,f\in H^r\} > \frac{1}{2}.$ Then for almost every $t$ we have
\begin{equation*}
    e^{it\omega(\frac{\nabla}{i})}f\not\in B_{1,\infty}^{\frac{2r_0-(r_0-s)q'}{2-q'}+},
\end{equation*}
where $q'$ is the H\"older conjugate of $q$. In particular, if $s = 0$ then 
\begin{equation*}
    e^{it\omega(\frac{\nabla}{i})}f\not\in B_{1,\infty}^{r_0+}.
\end{equation*}
\end{theorem}

\begin{theorem}[\cite{MR1152878}]\label{Theorem: Trivial Lower Bound}
The graph of a continuous function $f:\mathbb{T}\to \mathbb{R}$ has fractal dimension $D(f)\geq 2-\gamma$, provided that $f\not\in B^{\gamma+}_{1,\infty}.$
\end{theorem}

\begin{theorem}\label{Theorem: Dimension Upper Bound}
Let $\gamma\in(0,1)$ and $f:\mathbb{T}\to\mathbb{R}$. If $f\in C^\gamma$, then $D(f)\leq 2-\gamma$.
\end{theorem}
As a consequence of the resolution of Vinogradov's Mean Value Conjecture, Erdo{\u g}an and Shakan were able to establish the following general upper-bound.
\begin{theorem}[\cite{ErdoganShakanFractal}, Theorem 3.7]\label{Theorem: Free Solution Ca}
Let $\omega(n)$ be a real polynomial of degree $d$ with integer coefficients. Then for any $g\in BV$ and almost every $t$ we have $e^{it\omega(\frac{\nabla}{i})}g\in C^{\frac{1}{d(d-1)}-}.$
\end{theorem}

We are now ready to state the theorem that we aim to show.
\begin{theorem}
Let $s > 35/64$, $u_0\in (BV\cap H^s)\setminus H^{s+},$ and $u$ the solution to \eqref{Equation: General Fifth Order} emanating from $u_0$. Then for almost every $t$,
\begin{equation}\label{Equation: Dimension Bounds To Show}
    2-s\leq D(u) \leq 
    \begin{cases}
    \frac{115}{32}-3s & \frac{35}{64}<s\leq \frac{263}{480}\\
    \frac{39}{20} & \frac{263}{480} < s \leq \frac{11}{20}\\
    \frac{5}{2}-s & \frac{11}{20} < s < 1.
    \end{cases}
\end{equation}
\end{theorem}
\begin{proof}
It follows from Theorem \ref{Theorem: Free Solution Ca} for $\omega(n) = n^5$, $g\in BV$, and almost every $t$ that $W_tu_0\in C^{\frac{1}{20}-}$. Moreover, combining this with Sobolev embedding we find
\begin{equation*}
    W_t u_0\in
    \begin{cases}
    C^{\frac{1}{20}-} & \frac{1}{2}\leq s\leq \frac{11}{20}\\
    C^{s-\frac{1}{2}} & \frac{11}{20} < s < 1.
    \end{cases}
\end{equation*}
Let $\mathcal{N}(u):= \tilde{u}-W_tu_0$, so that the smoothing result of Theorem \ref{Theorem: Smoothing} implies via Sobolev Embedding that
\begin{equation*}
    \mathcal{N}(u) \in C^{s+\min(2s-35/32, 1)-\frac{1}{2}-}.
\end{equation*}
In other words, for $\frac{35}{64} < s \leq 1$ we find
\begin{equation*}
     \mathcal{N}(u) \in
     \begin{cases}
        C^{3s-\frac{51}{32}-} & \frac{35}{64} < s < \frac{83}{96}\\
        C^{1-} & \frac{83}{96}\leq s \leq 1.
     \end{cases}.
\end{equation*}

Combining these two bounds, we see that
\begin{equation*}
\begin{split}
    \tilde{u} &\in 
    \begin{cases}
    C^{\frac{1}{20}}+C^{3s-\frac{51}{32}-} & \frac{35}{64} < s\leq \frac{11}{20}\\
    C^{s-\frac{1}{2}}+C^{3s-\frac{51}{32}-} & \frac{11}{20} < s \leq \frac{83}{96}\\
    C^{s-\frac{1}{2}} + C^{1-} & \frac{83}{96} < s < 1
    \end{cases} \\
    &= 
    \begin{cases}
    C^{3s-\frac{51}{32}-} & \frac{35}{64} < s\leq \frac{263}{480}\\
    C^{\frac{1}{20}-} & \frac{263}{480} < s \leq \frac{11}{20}\\
    C^{s-\frac{1}{2}} & \frac{11}{20} < s < 1.
    \end{cases}
\end{split}
\end{equation*}
The upper bound of \eqref{Equation: Dimension Bounds To Show} then follows from Theorem \ref{Theorem: Dimension Upper Bound} and the above.

The lower bound follows from Theorem \ref{Theorem: Dimension Lower Bound} and the assumption that $u_0\not\in H^{s+}$. Indeed, the $L^{\frac{24}{7}}$ Strichartz estimate allows us to conclude that the free solution isn't in $B^{s+}_{1,\infty}$ by Theorem \ref{Theorem: Dimension Lower Bound}, and hence $\tilde{u}\not\in B^{s+}_{1,\infty}$. Since $\tilde{u}$ is continuous, we find by Theorem \ref{Theorem: Trivial Lower Bound} that for almost every $t$ we have the lower bound $D(\tilde{u})\geq 2-s.$

Since $\tilde{u}$ is just a spatial translation of $u$, we conclude both bounds for $u$, as well.
\end{proof}

\begin{remark}
The above upper bounds are only non-trivial in the regime $\frac{35}{64} < s\leq \frac{11}{20}$, which is quite a small region. Improvement on the lower bound of the dimension for $u$ as well as the regularity of the free solution are interesting avenues to improve the above result.
\end{remark}


\begin{thebibliography}{10}
    \bibitem{babin2011regularization}
    {\sc Babin, A.~V., Ilyin, A.~A., and Titi, E.~S.}
    \newblock On the regularization mechanism for the periodic {K}orteweg--de
      {V}ries equation.
    \newblock {\em Communications on Pure and Applied Mathematics 64}, 5 (2011),
      591--648.
    
    \bibitem{bourgain1993fourier}
    {\sc Bourgain, J.}
    \newblock Fourier transform restriction phenomena for certain lattice subsets
      and applications to nonlinear evolution equations.
    \newblock {\em Geometric \& Functional Analysis GAFA 3}, 3 (1993), 209--262.
    
    \bibitem{bourgain1998refinements}
    {\sc Bourgain, J.}
    \newblock Refinements of {S}trichartz inequality and applications to 2{D}-{NLS}
      with critical nonlinearity.
    \newblock {\em International Mathematics Research Notices 1998}, 5 (1998),
      253--283.
    
    \bibitem{bringmann2021global}
    {\sc Bringmann, B., Killip, R., and Visan, M.}
    \newblock Global well-posedness for the {F}ifth-{O}rder {K}d{V} {E}quation in
      {H}$^{-1}(\mathbb{R})$.
    \newblock {\em Annals of PDE 7}, 2 (2021), 1--46.
    
    \bibitem{colliander2001global}
    {\sc Colliander, J., Keel, M., Staffilani, G., Takaoka, H., and Tao, T.}
    \newblock Global well-posedness for schr{\"o}dinger equations with derivative.
    \newblock {\em SIAM Journal on Mathematical Analysis 33}, 3 (2001), 649--669.
    
    \bibitem{colliander2004multilinear}
    {\sc Colliander, J., Keel, M., Staffilani, G., Takaoka, H., and Tao, T.}
    \newblock Multilinear estimates for periodic {K}d{V} equations, and
      applications.
    \newblock {\em Journal of Functional Analysis 211}, 1 (2004), 173--218.
    
    \bibitem{colliander2003sharp}
    {\sc Colliander, J., Keel, M., Staffilani, G., and Tao, T.}
    \newblock Sharp global well-posedness for kdv and modified {K}d{V} on
      $\mathbb{T}$ and $\mathbb{R}$.
    \newblock {\em J. Amer. Math. Soc 16\/} (2003), 705.
    
    \bibitem{correia2020nonlinear}
    {\sc Correia, S., and Silva, J.~D.}
    \newblock Nonlinear smoothing for dispersive {PDE}: a unified approach.
    \newblock {\em Journal of Differential Equations\/} (2020).
    
    \bibitem{MR1152878}
    {\sc Deliu, A., and Jawerth, B.}
    \newblock Geometrical dimension versus smoothness.
    \newblock {\em Constr. Approx. 8}, 2 (1992), 211--222.
    
    \bibitem{erdogan2017smoothing}
    {\sc Erdo{\u{g}}an, M.~B., G{\"u}rel, T.~B., and Tzirakis, N.}
    \newblock Smoothing for the fractional {S}chr{\"o}dinger equation on the torus
      and the real line.
    \newblock {\em Indiana Univ. Math. J. 68\/} (2019), 369--392.
    
    \bibitem{erdougan2013global}
    {\sc Erdo{\u{g}}an, M.~B., and Tzirakis, N.}
    \newblock Global smoothing for the periodic {K}d{V} evolution.
    \newblock {\em International Mathematics Research Notices 2013}, 20 (2013),
      4589--4614.
    
    \bibitem{erdogan2011long}
    {\sc Erdo{\u{g}}an, M.~B., and Tzirakis, N.}
    \newblock Long time dynamics for forced and weakly damped {K}d{V} on the torus.
    \newblock {\em Communications on Pure \& Applied Analysis 12}, 6 (2013).
    
    \bibitem{erdogan2013talbot}
    {\sc Erdo{\u{g}}an, M.~B., and Tzirakis, N.}
    \newblock Talbot effect for the cubic non-linear {S}chr{\"o}edinger equation on
      the torus.
    \newblock {\em Mathematical Research Letters 20}, 6 (2013), 1081--1090.
    
    \bibitem{erdougan2016dispersive}
    {\sc Erdo{\u{g}}an, M.~B., and Tzirakis, N.}
    \newblock {\em Dispersive partial differential equations: wellposedness and
      applications}, vol.~86.
    \newblock Cambridge University Press, 2016.
    
    \bibitem{ErdoganShakanFractal}
    {\sc Erdo\u{g}an, M.~B., and Shakan, G.}
    \newblock Fractal solutions of dispersive partial differential equations on the
      torus.
    \newblock {\em Selecta Math. (N.S.) 25}, 1 (2019), Paper No. 11, 26.
    
    \bibitem{ginibre1997cauchy}
    {\sc Ginibre, J., Tsutsumi, Y., and Velo, G.}
    \newblock On the {C}auchy problem for the {Z}akharov system.
    \newblock {\em Journal of Functional Analysis 151}, 2 (1997), 384--436.
    
    \bibitem{hu2015local}
    {\sc Hu, Y., and Li, X.}
    \newblock Local well-posedness of periodic fifth-order {K}d{V}-type equations.
    \newblock {\em The Journal of Geometric Analysis 25}, 2 (2015), 709--739.
    
    \bibitem{hughes2019discrete}
    {\sc Hughes, K., and Wooley, T.~D.}
    \newblock Discrete restriction for $(x, x^3) $ and related topics.
    \newblock {\em arXiv preprint arXiv:1911.12262\/} (2019).
    
    \bibitem{isom2020growth}
    {\sc Isom, B., Mantzavinos, D., and Stefanov, A.}
    \newblock Growth bound and nonlinear smoothing for the periodic derivative
      nonlinear {S}chr{\"o}dinger equation.
    \newblock {\em arXiv preprint arXiv:2012.09933\/} (2020).
    
    \bibitem{kappeler2018wellposedness}
    {\sc Kappeler, T., and Molnar, J.}
    \newblock On the wellposedness of the {K}d{V}/{K}d{V}2 equations and their
      frequency maps.
    \newblock In {\em Annales de l'Institut Henri Poincar{\'e} C, Analyse non
      lin{\'e}aire\/} (2018), vol.~35, Elsevier, pp.~101--160.
    
    \bibitem{kato2012well}
    {\sc Kato, T.}
    \newblock Well-posedness for the fifth order {K}d{V} equation.
    \newblock {\em Funkcialaj Ekvacioj 55}, 1 (2012), 17--53.
    
    \bibitem{kato2018unconditional}
    {\sc Kato, T.}
    \newblock Unconditional well-posedness of fifth order {K}d{V} type equations
      with periodic boundary condition (harmonic analysis and nonlinear partial
      differential equations).
    \newblock {\em RIMS Kokyuroku Bessatsu
      70\/} (2018), 105--129.
    
    \bibitem{kenig2015well}
    {\sc Kenig, C.~E., and Pilod, D.}
    \newblock Well-posedness for the fifth-order {K}d{V} equation in the energy
      space.
    \newblock {\em Transactions of the American Mathematical Society\/} (2015),
      2551--2612.
    
    \bibitem{kenig2016local}
    {\sc Kenig, C.~E., and Pilod, D.}
    \newblock Local well-posedness for the {K}d{V} hierarchy at high regularity.
    \newblock {\em Advances in Differential Equations 21}, 9/10 (2016), 801--836.
    
    \bibitem{keraani2009smoothing}
    {\sc Keraani, S., and Vargas, A.}
    \newblock A smoothing property for the ${L}^2$-critical {NLS} equations and an
      application to blowup theory.
    \newblock {\em Annales de l'Institut Henri Poincar{\'e} (C) Non Linear Analysis
      26}, 3 (2009), 745--762.
    
    \bibitem{kwak2018low}
    {\sc Kwak, C.}
    \newblock Low regularity cauchy problem for the fifth-order modified {K}d{V}
      equations on $\mathbb{T}$.
    \newblock {\em Journal of Hyperbolic Differential Equations 15}, 03 (2018),
      463--557.
    
    \bibitem{kwon2008fifth}
    {\sc Kwon, S.}
    \newblock On the fifth-order {K}d{V} equation: local well-posedness and lack of
      uniform continuity of the solution map.
    \newblock {\em Journal of Differential Equations 245}, 9 (2008), 2627--2659.
    
    \bibitem{kwon2008well}
    {\sc Kwon, S.}
    \newblock Well-posedness and ill-posedness of the fifth-order modified {K}d{V}
      equation.
    \newblock {\em Electronic Journal of Differential Equations 2008\/} (2008).
    
    \bibitem{kwon2012unconditional}
    {\sc Kwon, S., and Oh, T.}
    \newblock On unconditional well-posedness of modified {K}d{V}.
    \newblock {\em International Mathematics Research Notices 2012}, 15 (2012),
      3509--3534.
    
    \bibitem{oh2020smoothing}
    {\sc Oh, S., and Stefanov, A.~G.}
    \newblock Smoothing and growth bound of periodic generalized {K}orteweg-de
      {V}ries equation.
    \newblock {\em arXiv preprint arXiv:2001.08984\/} (2020).

    \bibitem{Os1}
    Oskolkov, K.I.
    \newblock A class of {I. M. V}inogradov's series and its applications in
      harmonic analysis.
    \newblock In A.~A. Gonchar and E.B. Saff, editors, {\em Progress in
      approximation theory}, volume~19 of {\em Springer Ser. Comput. Math.}, pages
      353--402. Springer, New York, 1992.

    \bibitem{Ro}
    Rodnianski, I.
    \newblock Fractal solutions of {S}chr\"{o}dinger equation.
    \newblock {\em Contemp. Math.}, 255:181--187, 2000.

    \bibitem{shatah1985normal}
    {\sc Shatah, J.}
    \newblock Normal forms and quadratic nonlinear {K}lein-{G}ordon equations.
    \newblock {\em Comm. Pure Appl. Math 38\/} (1985), 685--696.

    \bibitem{Ta}
    Talbot, H.
    \newblock Facts related to optical science.
    \newblock {\em Philo. Mag.}, 9(IV):401--407, 1836.
    
    \bibitem{tao2006nonlinear}
    {\sc Tao, T.}
    \newblock {\em Nonlinear dispersive equations}, vol.~106 of {\em CBMS Regional
      Conference Series in Mathematics}.
    \newblock Published for the Conference Board of the Mathematical Sciences,
      Washington, DC; by the American Mathematical Society, Providence, RI, 2006.
    \newblock Local and global analysis.
    
    \bibitem{tsugawa2017parabolic}
    {\sc Tsugawa, K.}
    \newblock Parabolic smoothing effect and local well-posedness of fifth order
      semilinear dispersive equations on the torus.
    \newblock {\em arXiv preprint arXiv:1707.09550\/} (2017).
    
    \bibitem{zhou1997uniqueness}
    {\sc Zhou, Y.}
    \newblock Uniqueness of weak solution of the {K}d{V} equation.
    \newblock {\em International Mathematics Research Notices 1997}, 6 (1997),
      271--283.


    \end{thebibliography}
\end{document}